\documentclass{article}

\usepackage{arxiv}

\usepackage[utf8]{inputenc} % allow utf-8 input
\usepackage[T1]{fontenc}    % use 8-bit T1 fonts
\usepackage{hyperref}       % hyperlinks
\usepackage{url}            % simple URL typesetting
\usepackage{booktabs}       % professional-quality tables
\usepackage{amsfonts}       % blackboard math symbols
\usepackage{amsmath}        % Mathematical environments
\usepackage{amsthm}         % Theorems
\usepackage{amssymb}
\usepackage{bm}             % For boldface math
\usepackage{bbm}        % For the indicator symbol
\usepackage{nicefrac}       % compact symbols for 1/2, etc.
\usepackage{microtype}      % microtypography
\usepackage{graphicx}       % For including figures
\usepackage{enumitem}
\usepackage{comment}        % For removing parts of text without removing the code
\usepackage{ marvosym }     %Contradiction symbol \Lightning
\usepackage{todonotes}      % For todo notes
\usepackage{subfig}

\newcommand{\supp}{\text{supp}}
\newcommand{\lr}{\leftrightarrow}
\newcommand{\prob}{\mathbb{P}}
\newcommand{\expec}{\mathbb{E}}
\newcommand{\1}{\mathbbm{1}}
\newcommand{\IRD}{\texttt{IRD}}

\newcommand{\IRG}{\texttt{IRG}}

\newcommand{\CCI}{\texttt{RaCInG}}

\newcommand{\IEG}{\texttt{IEG}}
\newcommand{\IAG}{\texttt{IAG}}
\newcommand{\ASRG}{\texttt{ASRG}}
\newcommand{\ESRG}{\texttt{ESRG}}

\newtheorem{theorem}{Theorem}[section]
\newtheorem{corollary}[theorem]{Corollary}
\newtheorem{lemma}[theorem]{Lemma}
\newtheorem{prop}[theorem]{Proposition}
\newtheorem{definition}[theorem]{Definition}
\newtheorem{example}[theorem]{Example}
\newtheorem{assumption}[theorem]{Assumption}
\theoremstyle{definition}
\newtheorem*{remark}{Remark}
\newtheorem*{notation}{Notation}

\title{Flexible random graph modelling for cell-cell interactions}

\author{
  Mike van Santvoort \\
  Department of Mathematics and Computer Science\\
  Eindhoven University of Technology\\
  \texttt{m.v.santvoort@tue.nl} \\
   \And
   Pim van der Hoorn \\
Department of Mathematics and Computer Science\\
  Eindhoven University of Technology\\
  \texttt{w.l.f.v.d.hoorn@tue.nl}   
}

\begin{document}
\maketitle

\begin{abstract}
In this paper we explore generalisations to a random digraph model designed for modelling cell-cell interactions (RaCInG). In its baseline model -- tested both theoretically and in biological practice -- a fixed number of arcs are assigned uniformly at random to vertices that have suitable types to accept them. Generalisations to this model are explored in two directions: constructing an undirected graph version of the model, and incorporating spatial information. Both directions are driven by biological demands: better interpretability and adaptation to technological advances. We show not only that these generalisations are theoretically possible within the framework of RaCInG, but that they can be made rigorous, proving that RaCInG creates a flexible random graph framework to model cell-cell interactions.
\end{abstract}


\section{Introduction}
For many biological applications cell-cell interactions (CCI) play an important role. For example, CCIs play a central role in how immune cells differentiate \cite{Park2020AcellFormation}, how tissues maintain balance \cite{Krausgruber2020StructuralResponses} and how diseases like cancer take over parts of the body for their own benefit \cite{Qi2020SingleCoronaviruses}. Therefore, it is paramount that tools exists to extract information about CCIs and to infer them.

The challenge when modelling CCIs is twofold. First, it proves difficult to measure CCIs directly, using the available data. Much data is available, but none measures CCI{s} directly \cite{Armingol2021DecipheringExpression}. Therefore, CCIs always need {to} be inferred from measurable data. Secondly, CCIs form a complex web of reactions that all indirectly influence each other \cite{LapuenteSantana2020TowardsBlockers}. Thus, even if one is able to measure individual CCIs directly, it is still unclear how these combine to form the behaviour of an entire tissue.

Network models have proven to be a valuable tool in investigating how CCIs combine together to elicit a certain response \cite{Bafna2023CLARIFY, Kim2023CellNeighborEX, Huang2023MIGGRI, Cang2020InferringData, Pham2023RobustTissues}. In these models, it is possible to measure on a gene, cell, gene type or cell type level how different elements of a tissue combine and interact. Usually these models invoke some stochasticity, since measurable data ({e.g.}, some type of gene expression data) is often not directly compatible with the desired network structure of the model output.

The data that is used to infer CCIs {often} comes in {one of} three types: bulk, single-cell {or} spatial. Bulk data blends all gene expression of cells in a sample together, allowing us to measure the total expression of given genes. This data is the least detailed, but simultaneously the most {cost-effective} and widely applied \cite{Stark2019RNAYears}. {Single-cell} data allows us to zoom in on a sample when compared to bulk sequencing. Cells have been isolated from the sample, and gene expressions are extracted from them. This allows one to attain {context-specific} knowledge on how certain (types of) cells behave at the cost of losing the global perspective \cite{Tzec-Interian2024BioinformaticsAnalyses}. Finally, spatial data balances the benefits and limitations of the previous data types. It retains the spatial structure of a sample, allowing us to know where gene expression comes from. However, since it is the newest type of data, it is still unclear how to robustly use it in a modelling context, and we still cannot measure all genes using spatial techniques \cite{Bressan2023TheDawn}.

A good CCI model ideally is compatible with all types of data, and while some exist that are \cite{Tanevski2022ExplainableData}, they often rely on machine learning approaches that lose interpretability \cite{vanSantvoort2025AnData}. It proves difficult to combine different data types into one comprehensive interpretable package \cite{Badia-iMompel2023GeneOmics}. Therefore, it is important to construct a mathematical CCI model that can easily be adapted from context to context. A recent model we {developed for this purpose} is RaCInG \cite{vanSantvoort2023MathematicallyMicroenvironment}. It promises mathematical tractability, which has been shown through a connection to inhomogeneous random digraphs {(see e.g. \cite{Cao2020ConnectivityDigraphs})} in \cite{vanSantvoort2023FromCounts}. The applicability of RaCInG has been shown in \cite{vanSantvoort2023MathematicallyMicroenvironment} through practical examples where input data differs slightly or when fundamental network assumptions are changed. However, a mathematical analysis of these adaptations remains to be desired.

In this paper, we aim to further explore the mathematics behind RaCInG, and prove that the model allows for flexibility without losing its rigour. Flexibility will be explored through the lens of two biological problems. First, RaCInG outputs digraphs while often (undirected) graphs are more interpretable. This was solved in \cite{vanSantvoort2023MathematicallyMicroenvironment} by simply turning the arcs of output digraphs into edges. We will show this approach is not only natural, but induces a logical extension of RaCInG into the realm of undirected graphs. Second, the biological landscape is shifting towards spatial data, and there is a need to adapt RaCInG to work with this data. We will show that it is not only possible, but maintains the rigour that RaCInG was originally designed with. 

We will start by giving a general mathematical framework for RaCInG in Section~\ref{sec:RaCInG def}, and repeat the main result from \cite{vanSantvoort2023FromCounts}. Next, in Section~\ref{sec:dir to undir} we investigate how RaCInG naturally generalizes to an undirected model and the connection of this model to inhomogeneous random graphs \cite{Bollobas2007TheGraphs}. Finally, in Section~\ref{sec:SpaCInG} we show how RaCInG can be adapted to a CCI model using spatial data and how this model naturally connects to inhomogeneous random digraphs too. Proofs and heuristics behind our main results are given in Section~\ref{sec:main proofs} and proofs of minor lemmas will be {deferred} to Section~\ref{sec:lem proofs}. {Finally, Section~\ref{sec:examples} provides some illustrative examples of definitions we use throughout the paper. Figure~\ref{fig:overview models} gives an overview of all random graph models central in this article and the relationships between them.}

\begin{figure}
	\centering
	\includegraphics[scale = 0.65]{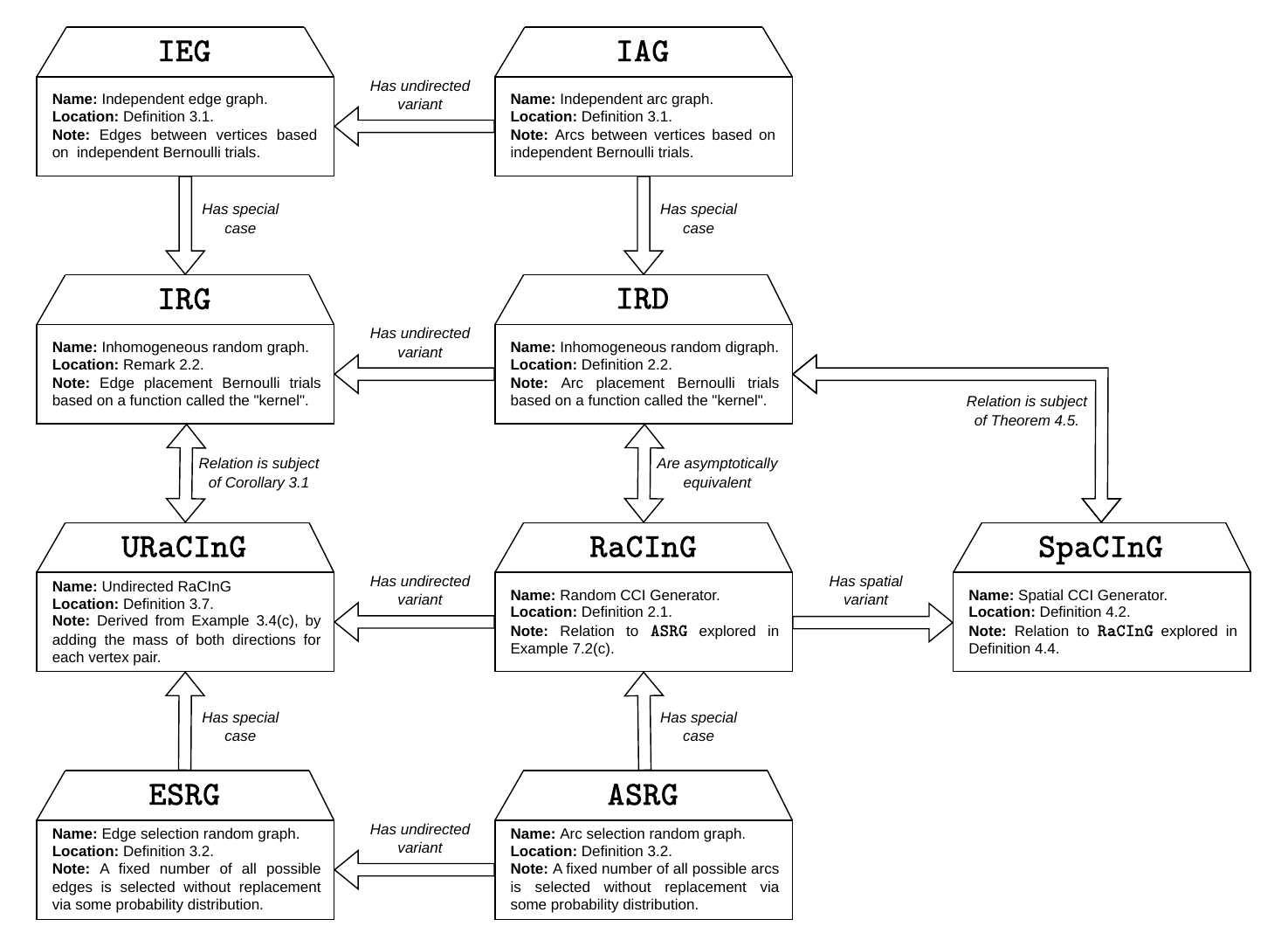}
	\caption{\emph{Overview of all models in this paper and the key relationships between them.}}
	\label{fig:overview models}
\end{figure}

\section{Modelling context and application challenges}\label{sec:RaCInG def}
RaCInG is a CCI model that {predominantly} uses bulk data as its input. It aims to create CCI networks where {vertices have colours representing cells and their types}, and where arcs represent CCI instances between a ligand (sending protein) and a receptor (receiving protein). The model transforms its input such that four pieces of information are derived from it that are used as the basis of its graph generation algorithm. These four pieces of information are:

\begin{enumerate}[]
	\item An arc colour distribution $C$ supported on a product $\mathcal{L} \times \mathcal{R}$, where $\mathcal{L}$ and $ \mathcal{R}$ are countably infinite sets.
	\item A vertex type distribution $T$ supported on a countably infinite set $\mathcal{S}$.
	\item A vertex type to out-colour compatibility matrix $I : \mathcal{S} \times \mathcal{L} \to \{0 , 1\}$.
	\item A vertex type to in-colour compatibility matrix $J : \mathcal{S} \times \mathcal{R} \to \{0, 1\}$.
\end{enumerate}

\begin{remark}
	Biologically, we can derive the four pieces of information as follows:
	\begin{enumerate}
		\item The arc colour distribution $C$ can be derived directly from bulk RNA-seq data. It gives the probability that a given interaction is between a fixed ligand type $i$ and fixed receptor type $j$.
		\item A vertex type distribution $T$ can be derived by mapping bulk gene expression data to cell type signatures \cite{Finotello2018QuantifyingData}, which gives the probability that a cell is of type $t$.
		\item A vertex type to out-colour and in-colour compatibility matrices $I$ and $J$ can {be} derived from literature curated databases \cite{Turei2021IntegratedAnalysis,Ramilowski2015ADraftHumans,Lapuente-Santana2021InterpretableInhibitors}, which tell whether ligand/receptor $i$ can be expressed by a cell type $t$.
	\end{enumerate}
\end{remark}

These four pieces of information are used in conjunction with two meta-parameters $n$ and $\mu$, depicting the number of vertices and average in-/out-degree, respectively. The following algorithm outlines how RaCInG constructs a network based on the input data.

\begin{definition}[RaCInG]\label{def:RaCInG algorithm}
	Fix a number of vertices $n$ to be generated and an average in-/out-degree $\mu$. We construct a graph $\vec{G} = (V, \vec{E})$ from RaCInG as follows:
	\begin{enumerate}[label = \arabic*.]
		\item Create $n$ vertices, and set $V = [n]$.
		\item For each vertex $v \in [n]$, sample an independent vertex type $T_v \sim T$.
		\item Create $\lfloor \mu n \rfloor$ possible arcs.
		\item For each interaction $a \in [\lfloor \mu n \rfloor]$, do the following:
		\begin{enumerate}[label = \alph*.]
			\item Assign $a$ an out- and in-colour $C_a = (C^{\text{out}}_a, C^{\text{in}}_a) \sim C$.
			\item Independently choose one vertex $v$ uniformly from the set \[\left\{ v \in [n] : I(T_v, C^\text{out}_a) = 1 \right\}.\]
			If this set is empty, go to the next interaction.
			\item  Then, independently, choose one vertex $w$ uniformly from the set \[\left\{ w \in [n] : J(T_w, C^\text{{in}}_a) = 1 \right\}.\]
			If this set is empty, go to the next interaction.
			\item Set $(v, w) \in \vec{E}$. 
		\end{enumerate}
	\end{enumerate}
	We denote this model by $\texttt{RaCInG}_{n, \mu}(T, C, I, J)$. 
\end{definition}

Definition~\ref{def:RaCInG algorithm} provides a natural way to generate instances of RaCInG, but the strength of the CCI model -- as explored in Theorem 3.6 of \cite{vanSantvoort2023FromCounts} -- is its connection to inhomogeneous random digraphs {(IRDs)}. Under mild conditions the theorem allows us to {compute} properties of RaCInG by computing the same property in an equivalent {IRD} model. We first give a definition of {IRDs}.

\begin{definition}[Inhomogeneous random digraphs]\label{def:IRD}
	Fix a type space $\mathcal{S}$, a function $\kappa_n : \mathcal{S}^2 \to \mathbb{R}_+$ {(called the model's \emph{kernel})} and a distribution $T$ supported on $\mathcal{S}$. An \emph{inhomogeneous random digraph} is a model, denoted by $\texttt{IRD}_n(T, \kappa_n)$, that outputs a directed graph $\vec{G} = ([n], \vec{E})$ using the following algorithm:
	\begin{enumerate}[label = \arabic*.]
		\item Assign each vertex $v \in [n]$ an independent type $T_v \sim T$.
		\item Given the vertex types $(T_v)_{v \in [n]}$, let $(X_{vw})_{v, w \in [n]}$ be independent Bernoulli random variables with success probability $\min\{\kappa_n(T_v, T_w)/n, 1\}$. Set $(v, w) \in \vec{E}$ for each $v, w \in [n]$ such that $v \neq w$ if and only if $X_{vw} = 1$. 
	\end{enumerate}
\end{definition}
\begin{remark}
	The undirected version of IRD (called \emph{inhomogeneous random graphs}; IRG) can be similarly defined. The main differences are that $\kappa_n$ needs {to} be symmetric and that only $X_{vw}$ for $v < w$ is considered.
\end{remark}

If we want to identify {the} IRD instance {that is equivalent to RaCInG}, then we need to specify {the kernel that does this}.

\begin{definition}[RaCInG kernel]\label{def: RaCInG kernel}
	Let $q_k := \prob(T = k)$ and $p_{ij} := \prob(C = (i, j))$. We will call $\kappa$ the \emph{RaCInG kernel} if it is given by
	\begin{equation}
		\kappa(t, s) =\mu \sum_{i \in \mathcal{L}} \sum_{j \in \mathcal{R}}\frac{p_{ij} \cdot I(t, i) J(s, j)}{\lambda_i \cdot \varrho_j}, \label{eq:asymp connection number}
	\end{equation}
	with
	\begin{subequations}
		\begin{equation}
			\lambda_i = \sum_{k \in \mathcal{S}} q_k I(k, i),\label{eq:fraction of vertices connecting to lig}
		\end{equation}
		\begin{equation}
			\varrho_j = \sum_{k \in \mathcal{S}} q_k J(k, j).\label{eq:fraction of vertices connecting to rec}
		\end{equation}
	\end{subequations}
\end{definition}
\begin{remark}
	{In Definition~\ref{def: RaCInG kernel} we implicitly assume that $\lambda_i, \varrho_j > 0$ for all $i \in \mathcal{L}$ and $j \in \mathcal{R}$. Biologically, this is a natural assumption, because otherwise ligands/receptors would be included in the model that cannot be produced by any cell. We make this assumption explicit below (see Assumption~\ref{ass:finite support}).}
\end{remark}

The translation between RaCInG and IRD {is proven in \cite{vanSantvoort2023FromCounts} at the level of \emph{graph events}. A graph event $\mathcal{Q}_n$ is a set of (possibly marked) graphs. A random graph model satisfies $\mathcal{Q}_n$ if its output graph lands in the set $\mathcal{Q}_n$. The translation between RaCInG and IRD can only be done for a specific class of graph events. The events in this class are called \emph{monotone}}.

\begin{definition}[Monotone events]\label{def:monotone events}
	Let $G_1$ and $G_2$ be two graphs such that $G_1 \subseteq G_2$. We say a collection $\mathcal{Q}_n$ of events is \emph{increasing} if $G_1 \in \mathcal{Q}_n$ implies $G_2 \in \mathcal{Q}_n$. Similarly, we say $\mathcal{Q}_n$ is \emph{decreasing} if $G_2 \in \mathcal{Q}_n$ implies $G_1 \in \mathcal{Q}_n$. Finally, we say $\mathcal{Q}_n$ is \emph{monotone} if it is either increasing or decreasing.
\end{definition}

Finally, we put some restrictions on RaCInG to ensure certain fringe cases cannot occur. 

\begin{assumption}[Finite support]\label{ass:finite support}
	We assume that $|\mathcal{S}|, |\mathcal{L}|,|\mathcal{R}| < \infty$. {Moreover, we assume for all $i \in \mathcal{L}$ and $j \in \mathcal{R}$ that $\lambda_i > 0$ and $\varrho_j > 0$.}
\end{assumption}

Note that {Assumption~\ref{ass:finite support}} is stronger than {what was assumed} in \cite{vanSantvoort2023FromCounts}. However, we will {take these stronger assumptions}, because (1) biologically the number of cell- and protein types considered is always finite and (2) it allows us to forego many of the technicalities that will obscure generalisations to the model. We can now repeat Theorem 3.6 of \cite{vanSantvoort2023FromCounts}. Note its statement is simplified, because Assumption~\ref{ass:finite support} removes the need for some technicalities in the original statement.

\begin{theorem}[RaCInG to IRD; see \cite{vanSantvoort2023FromCounts}]
	\label{thm:CCI to IRD}
	Consider $\texttt{RaCInG}_{n, \mu}(T, C, I, J)$ under Assumption~\ref{ass:finite support} and let $\kappa$ be as in Definition~\ref{def: RaCInG kernel}. {Fix a large enough constant $c > 0$ and} let $\mathcal{Q}_n$ {be} a monotone event. If there exists a number $p \in [0, 1]$ such that $\prob(\IRD_n(T, \kappa'_n) \in \mathcal{Q}_n) \to p$ for all sequences $\kappa'_n $ satisfying the inequality
	\begin{equation}\label{eq:kern sequence CCI}
		|\kappa'_n(t, s) - \kappa(t, s) | \leq c n^{-1/10},
	\end{equation}
	then we also have that $\prob(\CCI_{n, \mu}(T, C, I, J) \in \mathcal{Q}_n) \to p.$
\end{theorem}

In short, Theorem~\ref{thm:CCI to IRD} tells us that a given monotone property has limiting probability $p$ in RaCInG if {it has the same limit in IRD for a range of kernels}.

In this paper we seek to formulate theorems like Theorem~\ref{thm:CCI to IRD} in the case where RaCInG is extended to the realm of undirected graphs (Section~\ref{sec:dir to undir}) and spatial graphs (Section~\ref{sec:SpaCInG}). Such theorems would retain the main benefit of RaCInG -- rigorously known asymptotic behaviour that can be used to compute CCI properties -- whilst adapting it to different biological contexts. We will start each of the following sections by highlighting a particular biological problem that is solved. We will apply the main result of each section to count direct interactions between two vertices of a given type to illustrate how the main results can be applied in a modelling context. All in all, the results in this paper show that indeed RaCInG can be flexibly altered to new biological contexts and data modalities.

\section{Using RaCInG to model undirected interactions}\label{sec:dir to undir}

\paragraph{Biological problem.} RaCInG creates directed graphs to model cell-cell interactions, assuming that signals always go from a ligand towards a receptor. However results in \cite{vanSantvoort2023MathematicallyMicroenvironment} and \cite{Gibbs2021Patient-SpecificCancer} show that this directionality assumption does not always hold. Thus, RaCInG needs to be extended to an undirected version that does not make this assumption. Originally, this was done in \cite{vanSantvoort2023MathematicallyMicroenvironment} by adding together all directed features that would turn into the undirected feature if directions are forgotten (see Figure~\ref{fig:forgetting direction} for an illustration). Despite this approach seemingly working, there was no mathematical justification.

In this section we will validate to what extent it is possible to generalize results about RaCInG to a setting where we forget the direction of its arcs. We will do this by identifying two {broad} classes of random (di)graph models, and by showing that an equivalence in the directed models {of each class} implies a similar equivalence in the undirected classes {(Theorem~\ref{thm:direct implies undirect})}. We will link this result to RaCInG by showing that inhomogeneous random graphs (IRG; the undirected version of IRD) can be used to calculate its undirected properties {(Corollary~\ref{cor:CCI to IRG})}.

\begin{figure}
	\centering
	\includegraphics[scale = 0.7]{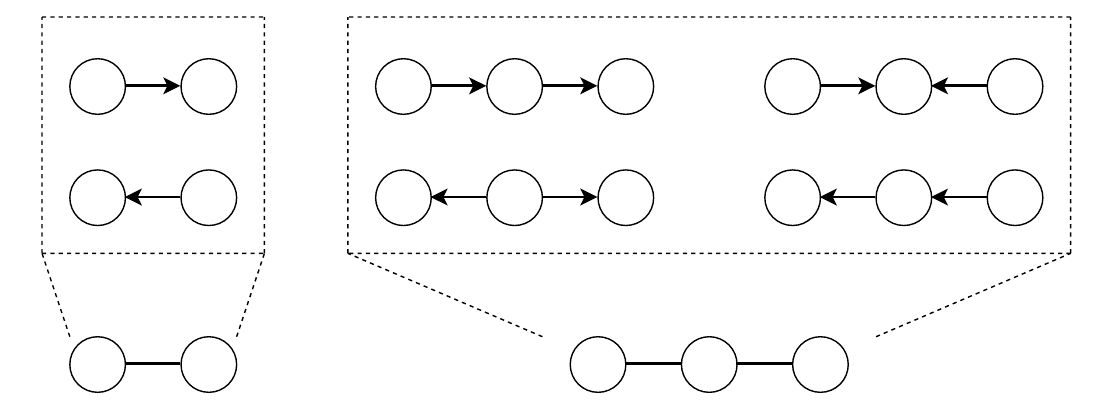}
	\caption{\emph{Examples of the directed features that would be added to get an undirected feature.}}
	\label{fig:forgetting direction}
\end{figure}

\subsection{Graphs with independent edge/arc probabilities}
The first {class} class of models we will {define}, are \emph{independent arc/edge graphs}. Their defining feature is that vertices are {first given random types}, after which existence of edges/arcs is determined by coin-flips. We will give the formal definition of this model both in the undirected and directed case.

\begin{definition}[Independent edge/arc graphs]\label{def:IAG and IEG}
	Fix a set $\mathcal{V}$ and let $(V_n)_{n \geq 1}$ be a sequence of $n$-dimensional random vectors with entries in $\mathcal{V}$ such that $|V_n| = n$. Furthermore, fix a function $\pi_n : \mathcal{V} \times \mathcal{V} \to [0, 1]$. 
	
	In the \emph{independent arc graph}, we generate a {simple} digraph $\vec{G}_n = (V_n, \vec{E}_n)$ by first realising $V_n$ followed by a sequence of independent Bernoulli random variables $I_{vw} \sim \texttt{Bern}(\pi_n(v, w))$. Given that {$v \neq w$}, we have that $(v, w) \in \vec{E}_n$ if and only if $I_{vw} = 1$. We denote this model by $\IAG_n(V_n, \pi_n)$.
	
	Similarly, for the \emph{independent edge graph}, we generate a {simple} graph $G_n = (V_n, E_n)$ by first realising $V_n$ followed by a sequence of independent Bernoulli random variables $I_{vw} \sim \texttt{Bern}(\pi_n(v, w))$ for each unordered pair $v, w \in V_n$ {such that $v \neq w$}. We have that $\{v, w\} \in E_n$ if and only if $I_{vw} = 1$. We denote this model by $\IEG_n(V_n, \pi_n)$.
	
\end{definition}
\begin{remark}
	Since $\IEG_n(V_n, \pi_n)$ generates edges only for each \emph{{un}ordered pair} $v, w\in V_n$, we require $\pi_n$ to be symmetric. In $\IAG_n(V_n, \pi_n)$ this needs not be the case.
\end{remark}

Since $\mathcal{V}$ can be any set we want, and since $V_n$ can be any vector derived from $\mathcal{V}$, the models IAG and IEG {have many well-known random graph models as special cases. See Example~\ref{ex:IEG and IAG}.}

\subsection{Graphs with a fixed number of edges and arcs}
The second class of models we consider has its number of arcs/edges fixed. It again {assigns vertices a random type}, but {then} conditional on the {types a probability mass function is defined over all edge/arc location}. The fixed number of arcs/edges are sampled with {this probability mass function}. We will call this model the \emph{edge/arc selection random (di)graph}.

\begin{definition}[Edge/arc selection random graphs]\label{def:ESRG and ASRG}
	Fix a set $\mathcal{V}$ and let $(V_n)_{n \geq 1}$ be a sequence of $n$-dimensional random vectors with entries in $\mathcal{V}$ such that $|V_n| = n$. Also, given $V_n$, fix a {probability} mass function $\nu_n : V_n \times V_n \to [0, 1]$ and a number $m_n \in \mathbb{N}$. {We assume that $\nu_n(v, v) = 0$ for all $v \in V_n$, i.e. that no self-loops are admissible.} In the \emph{edge selection random graph model}, $\nu_n$ {is a probability mass function} defined on unordered pairs {$v, w \in V_n$ such that $v \neq w$}, while in the \emph{arc selection random graph model} it is {a probability mass function defined} on ordered tuples {such that $v \neq w$}. Given $\nu_n$, we realise a (di)graph $G_n = (V_n, E_n)$ as follows:
	\begin{enumerate}\itemsep0em
		\item Sample $V_n$ and start with $E_n = \emptyset$.
		\item Draw edges/arcs independently {using probability mass function} $\nu_n$ until you hit an edge {$\{v, w\}$ (for the edge selection random graph model) or} arc $(v, w)$ {(for the arc selection random graph model)} that is not part of $E_n$ yet.
		\item Set {$\{v, w\} \in E$ (for the edge selection model) or} $(v, w) \in E_n$ {(for the arc selection model)}, and go back to step 2 until $|E_n| = m_n$.
	\end{enumerate}
	We denote the edge selection random graph model by $\ESRG_n(V_n, \nu_n, m_n)$ and the arc selection random graph model by $\ASRG_n(V_n, \nu_n, m_n)$.
\end{definition}

\begin{remark}
	Note from Definition~\ref{def:ESRG and ASRG} that $\nu_n$ must be \emph{symmetric} in $\ESRG$, but not necessarily in $\ASRG$.
\end{remark}
\begin{remark}
	{RaCInG is almost an arc selection random graph. The main difference is that Definition~\ref{def:ESRG and ASRG} does not allow for self-loops and multi arcs, while Definition~\ref{def:RaCInG algorithm} does. This discrepancy is alleviated by Lemmas 8 and 9 in \cite{vanSantvoort2023FromCounts}, showing that self-loops and multi-arcs only occur at a sub-linear rate. We see for many random graph models with a linear number of edges that robust properties are not affected by a sub-linear number of self-loops and multi-edges. See e.g. Chapter 7 in \cite{vanderHofstad2016RandomNetworks1} for a canonical examples. Also, note that $\nu_n$ in Definition~\ref{def:ESRG and ASRG} is \emph{not} the same as the arc colour in Definition~\ref{def:RaCInG algorithm}. In the context of Definition~\ref{def:RaCInG algorithm}, $\nu_n$ would be defined through the uniform random choices in Step 4b and 4c.}
\end{remark}

{Definition~\ref{def:ESRG and ASRG} is again broad and many known random graph models are special cases. See Example~\ref{ex:ESRG and ASRG}.}

\subsection{Extending equivalences from directed to undirected random graphs}
We will see that the undirected and directed counterparts of Definition~\ref{def:IAG and IEG} and~\ref{def:ESRG and ASRG} are equivalent. Using these equivalences, we can prove that an equivalence between directed instances of these model classes implies an equivalence to the undirected instances {(Theorem~\ref{thm:direct implies undirect})}.

{To work towards Theorem~\ref{thm:direct implies undirect},} we start by introducing some notation to keep track of the largest probabilistic mass in $\IRD$ and $\ASRG$ and define what the undirected versions of $\pi_n$ and $\nu_n$ are. Note that all quantities are stochastic, since they depend on $V_n$.
\begin{notation}
	Given the realization of $V_n$ in $\IAG_n(V_n, \pi_n)$ and $\ASRG_n(V_n, \nu_n, m_n)$ we will set
	\[
	\pi^\uparrow_n := \max_{i, j \in V_n} \pi_n(i, j), \quad \text{and} \quad \nu^\uparrow_n := \max_{i, j \in V_n} \nu_n(i, j).
	\]
	Also, for all $v, w \in V_n$ we will set
	\[
	\pi^\circ_n(v, w) := \pi_n(v, w) + \pi_n(w,v), \quad \text{and} \quad \nu^\circ_n := \nu_n(v,w) + \nu_n(w,v).
	\]
\end{notation}

\begin{remark}
	{If $\pi_n(v,w) = \Theta(1)$, then the current definition of $\pi_n^\circ(v,w)$ might \emph{not} be a probability. Formally, it would therefore be more precise to define $\pi_n^\circ(v,w) = \min\{\pi_n(v,w)+\pi_n(w,v),1\}$. However, we will later be assuming that $\pi_n^\uparrow \to 0$ in probability, meaning that for all intents and purposes $\min\{\pi_n(v,w)+\pi_n(w,v),1\} = \pi_n(v,w)+\pi_n(w,v)$. Because of this reason, and to keep the notation consistent with $\nu^\circ_n$, we opted to define $\pi_n^\circ$ as merely the sum. Note that all proofs still work when we define $\pi_n^\circ$ as the truncated sum; we would simply need to add the argument of this remark to all proofs.}
\end{remark}

{To show equivalence between the directed and undirected models in Definition~\ref{def:IAG and IEG} and~\ref{def:ESRG and ASRG}, we will couple them such that they have arcs/edges in the same locations as often as possible. However, we will see in this coupling that} the undirected model will have {some extra edges}. {We need to assume that these extra edges do not change the probability of the graph event $\mathcal{Q}_n$ we are interested in. To do this, we will make the assumption that adding a fixed number of edges/vertices will \emph{never} change the limiting probability of $\mathcal{Q}_n$ no matter how we add them. We will call this assumption \emph{event insensitivity}. In practice, this assumption can be checked by considering the ``worst case'' in which an event $\mathcal{Q}_n$ can be perturbed, and arguing that this will not make a difference in the limit.}
\begin{definition}[Event insensitivity]\label{def: event insensitivity}
	Fix {a random graph model $\texttt{RG}_n$ and} an event $\mathcal{Q}_n$ on it. {Next, fix} two rates $r_n, r'_n \in \mathbb{N}$ {that are possibly random}. Let $\texttt{RG}_n^\uparrow(r_n, r_n')$ be derived from the realisation of $\texttt{RG}_n$ by first adding {any possible choice of} $r_n'$ vertices (with degree zero), and then {any possible choice of} $r_n$ edges. Similarly, define $\texttt{RG}_n^\downarrow(r_n, r_n')$ by first removing {any possible choice of} $r_n$ edges, and then {any possible choice of} $r_n'$ vertices with degree zero. We say that $\mathcal{Q}_n$ is \emph{insensitive for increase at rate} $(r_n, r_n')$ if we always have that
	\[\prob(\texttt{RG}_n \in \mathcal{Q}_n) = \prob(\texttt{RG}_n^\uparrow(r_n, r_n') \in \mathcal{Q}_n) + o(1).\] 
	Moreover, we say that $\mathcal{Q}_n$ is \emph{insensitive for decrease at rate} $(r_n, r_n')$ if we always have that
	\[\prob(\texttt{RG}_n \in \mathcal{Q}_n) = \prob(\texttt{RG}_n^\downarrow(r_n, r_n') \in \mathcal{Q}_n) + o(1).\] 
	If both are true, then $\mathcal{Q}_n$ is \emph{completely insensitive at rate} $(r_n, r_n')$.
\end{definition}
\begin{remark}
	Despite insensitivity being defined for random graphs in the previous definition, we can analogously define it for random digraphs. Also, note that event insensitivity requires asymptotic probabilities to be conserved for any possible choice of adding edges/vertices. Often, this requirement is too strong if we know the process in which vertices/edges get added (which we will for the two model extensions in this paper). We can then always replace the insensitivity requirement by an appropriate weaker condition. However, we will not do this, because (1) it perturbs a uniform language of error-control between models, (2) it adds extra technicalities to proofs{, and (3) many properties satisfy Definition~\ref{def: event insensitivity} (see Example~\ref{ex:insensitivity}).}
\end{remark}

Finally, before we can state {Theorem~\ref{thm:direct implies undirect}}, we need to define how events in a digraph model relate to those in an undirected random graph model. In essence, we get an undirected random graph event by taking a directed random graph and simply changing all the arcs into edges. If through this process a multi-edge emerges, we collapse them into one edge. We call this process \emph{forgetting} the directions.
\begin{definition}[Forgetful map]\label{def: forgetful map}
	Let $\mathbb{G}_n$ denote the set of simple graphs with $n$ vertices, and let $\vec{\mathbb{G}}_n$ denote the set of simple digraphs with $n$ vertices. We define the \emph{forgetful map} $\mathcal{U}: \vec{\mathbb{G}}_n \to \mathbb{G}_n$ that maps a digraph $\vec{G} = (V, \vec{E})$ to a graph $G = (\mathcal{U}(V), \mathcal{U}(\vec{E}))$ as follows:
	\begin{enumerate}[label = \Roman*.]\itemsep0em
		\item $\mathcal{U}(V) = V$ {with vertex properties preserved}, and
		\item $\mathcal{U}(\vec{E}) = \{\{v,w\} \, : \, (v,w) \in \vec{E} \text{ or } (w,v) \in \vec{E}\}$.
	\end{enumerate}
\end{definition}

If we are given an undirected graph event and aim to turn it into a set of directed graphs, we can simply take the pre-image $\mathcal{U}^{-1}$ of the forgetful map in Definition~\ref{def: forgetful map}. We can now formulate {Theorem~\ref{thm:direct implies undirect}}. {It is proven} in Section~\ref{sec: main result dir to indir}.

\begin{theorem}[Directed equivalence to undirected equivalence]\label{thm:direct implies undirect}
	Consider $\IAG_n(V_n, \pi_n)$ and $\ASRG_n(V_n, \nu_n, m_n)$ assuming that {$\pi_n^\uparrow \to 0$ and} $m_n \nu_n^\uparrow \to 0$ in probability. Moreover, let $\omega(n) \to \infty$ arbitrarily slowly and fix
	\begin{equation}\label{eq:Translation rate dir undir}
		r_n := {\left\lceil \max\left\{ \omega(n)(n \pi_n^\uparrow)^2, \; \omega(n) m_n^2 \nu^\uparrow_n , \; 1 \right\} \right\rceil}.
	\end{equation}
	Finally, consider an event $\mathcal{Q}_n$ for which both of the following two conditions are satisfied:
	\begin{enumerate}
		\item At rate $(r_n, 0)$, either $\mathcal{Q}_n$ is insensitive for decrease in $\IEG_n(V_n, \pi_n^\circ)$ or $\mathcal{U}^{-1}(\mathcal{Q}_n)$ is insensitive for increase in $\IAG_n(V_n, \pi_n)$.
		\item At rate $(r_n, 0)$, either $\mathcal{Q}_n$ is insensitive for decrease in $\ESRG_n(V_n, \nu_n^\circ, m_n)$ or $\mathcal{U}^{-1}(\mathcal{Q}_n)$ is insensitive for increase in $\ASRG_n(V_n, \nu_n, m_n)$.
	\end{enumerate}
	Then,
	\begin{equation}\label{eq:IAG ASRG equivalence}
		\prob(\IAG_n(V_n, \pi_n) \in \mathcal{U}^{-1}(\mathcal{Q}_n)) = \prob(\ASRG_n(V_n, \nu_n, m_n) \in \mathcal{U}^{-1}(\mathcal{Q}_n)) + o(1),
	\end{equation}
	implies,
	\begin{equation}\label{eq:IEG IAG equivalence}
		\prob(\IEG_n(V_n, \pi_n^\circ) \in \mathcal{Q}_n) = \prob(\ESRG_n(V_n, \nu_n^\circ, m_n) \in \mathcal{Q}_n) + o(1).
	\end{equation}
\end{theorem}

{\subsection{Application to RaCInG}}

A corollary to {Theorem~\ref{thm:direct implies undirect}} will {generalise} Theorem~\ref{thm:CCI to IRD} {to} undirected models {(Corollary~\ref{cor:CCI to IRG})}, validating the approach in \cite{vanSantvoort2023MathematicallyMicroenvironment}. {We will first define exactly what we mean with the ``undirected version of RaCInG'' before stating this corollary. Finally, we will give an example that shows how undirected features in RaCInG can be computed (Proposition~\ref{prop:direct racing}).\\

	\begin{definition}[Undirected RaCInG]\label{def:uracing}
		Consider $\texttt{RaCInG}_{n, \mu}(T, C, I, J)$. Let $V_n$ and $\nu_n$ as in Example~\ref{ex:ESRG and ASRG}\ref{ex:RaCInG}. We define \emph{undirected RaCInG} as $\texttt{ESRG}_n(V_n, \nu_n^\circ, \lfloor \mu n \rfloor)$, which we write as $\texttt{URaCInG}_{n, \mu}(T, C, I, J)$.
	\end{definition}
}

\begin{corollary}[{Equivalence of IRG to an undirected RaCInG}]
	\label{cor:CCI to IRG}
	Consider $\texttt{RaCInG}_{n, \mu}(T, C, I, J)$ {and its undirected counterpart $\texttt{URaCInG}_{n, \mu}(T, C, I, J)$}. Let $\mathcal{Q}_n$ be a monotone event and fix an arbitrary function $\omega(n) \to \infty$. Furthermore, set $\kappa$ as in Definition~\ref{def: RaCInG kernel}. 
	
	Suppose there exists a number $p \in [0, 1]$ such that
	\[
	\prob(\texttt{IRG}_n(T, \kappa^\circ_n) \in \mathcal{Q}_n) \to p,
	\]
	as $n \to \infty$ for all sequences $\kappa^\circ_n = \kappa_n(t,s) + \kappa_n(s,t)$ where $\kappa_n$ satisfies the inequality
	\begin{equation}\label{eq:kernel bound undirect}
		|\kappa_n(t, s) - \kappa(t, s) | \leq c n^{-1/10},
	\end{equation}
	for large enough fixed $c > 0$. Moreover, suppose that both of the following are true:
	\begin{enumerate}
		\item At rate $(\omega(n), 0)$ for all $\kappa_n$ with probability tending to one, either $\mathcal{Q}_n$ is insensitive for decrease in $\texttt{IRG}_n(T, \kappa_n^\circ)$ or $\mathcal{U}^{-1}(\mathcal{Q}_n)$ is insensitive for increase in $\texttt{IRD}_n(T, \kappa_n)$.
		\item At rate $(\omega(n), 0)$ with probability tending to one, either ${\mathcal{U}^{-1}(}\mathcal{Q}_n{)}$ is insensitive for decrease in $\texttt{RaCInG}_{n, \mu}(T, C, I, J)$ or $\mathcal{Q}_n$ is insensitive for increase in {$\texttt{URaCInG}_{n, \mu}(T, C, I, J)$}.
	\end{enumerate}
	Then, we also have that
	\[
	\prob({\texttt{URaCInG}_{n, \mu}(T, C, I, J)}  \in \mathcal{Q}_n) \to p.
	\]
\end{corollary}

\begin{prop}[Direct interactions]\label{prop:direct racing}
	Let $E_n^{ts}$ denote the number of edges between two different vertex types $t, s \in \mathcal{S}$ in $\texttt{ESRG}_n(V_n,\nu_n^\circ,\lfloor\mu n\rfloor)$ with $\nu_n$ as in Example~\ref{ex:ESRG and ASRG}\ref{ex:RaCInG}. We have that $E_n^{ts}/n \to (\kappa(t, s) + \kappa(s, t))q_tq_s$ in probability with $\kappa$ as in Definition~\ref{def: RaCInG kernel}.
\end{prop}
\begin{remark}
	Note how the limit in Proposition~\ref{prop:direct racing} adds together both possible arc directions, validating the approach taken in {\cite{vanSantvoort2023MathematicallyMicroenvironment}}. Similar results can also be obtained when e.g. counting the number of wedges (paths of length 2; see Figure~\ref{fig:forgetting direction}). However, these require more technical arguments to show the required event insensitivity.
\end{remark}

\section{Using RaCInG to model interaction{s} in space}\label{sec:SpaCInG}

\paragraph{Biological problem.} RaCInG was designed for so-called bulk RNA sequencing data. This is data about gene expression where all signals in a tissue are blended together. Nowadays, researchers have access to more sophisticated data where the location of gene signals is also known. This data is called spatial omics \cite{vanSantvoort2025AnData}. We want RaCInG to be applicable to this data too. 

In this section we shall generalize RaCInG to a cell-cell interaction model with spatial structure, applicable to spatial omics data. We will first explore how Definition~\ref{def:RaCInG algorithm} can be adapted to incorporate spatial information. This adaptation will be called SpaCInG. Thereafter, we will state an analogue to Theorem~\ref{thm:CCI to IRD} in case of this new model {(Theorem~\ref{thm:SCCI to SIRD})}.

\subsection{SpaCInG as an algorithm}
Spatial omics data exists in multiple forms, but generally cell-cell inference models condense its information into distinct spatial blocks called ``spots'' \cite{vanSantvoort2025AnData}. Each spot has a discrete location and its own gene expression data. Using literature databases these genes can be split into genes used by cells to send information (ligands) and genes used by cells to receive it (receptors). Moreover, the gene expression per spot can be deconstructed using known gene expression profiles of cells to find the likelihood of each cell appearing at each spot. This process is called deconvolution.

Based on the description of the available data, together with the general structure of RaCInG, we can leverage the following mathematical objects for a spatial {extension}: 
\begin{enumerate}[label = \arabic*.]
	\item A {discrete} set of spatial locations $\mathcal{X}$.
	\item A vertex type distribution for each $x \in \mathcal{X}$ given by $T^{(x)}$ supported on $\mathcal{S}$.
	\item An arc colour distributions for each $x \in \mathcal{X}$ given by $C^{(x)}$ supported on $\mathcal{C}$.
	\item A binary arc colour compatibility matrix $\1$ supported on $\mathcal{C}^2$.
	\item A vertex type compatibility matrix with out-arc colours $I$ supported on $\mathcal{S} \times \mathcal{C}$.
	\item A vertex type compatibility matrix with in-arc colours $J$ supported on $\mathcal{S} \times \mathcal{C}$.
\end{enumerate}
\begin{remark}
	Note that the vertex type compatibility with out- and in-arcs can be inferred from the cell compatibility matrix together with knowledge about what genes serve as ligand/receptors in the LR compatibility matrix. In other words, $I$ and $J$ are filtrations of the cell compatibility matrix.
\end{remark}

From the described spatial omics data it is unclear how many vertices need be sampled at each spatial location. These locations encode the spatial structure of a tissue. Inside each of them a number of vertices reside that drive cell-cell interactions. It is only known that each of the locations roughly contain the same amount of vertices (because locations are usually spaced equidistantly) \cite{Moses2022MuseumTranscriptomics}, but the exact amount cannot always be determined or even reliably estimated from gene expression data \cite{Tian2022TheExpandingTranscriptomics}. Therefore, to ensure that each $x \in \mathcal{X}$ has the opportunity to equally contribute to CCIs, we assume that each location contains the same amount of vertices.

Additionally, we assume that the support of vertex types, colours and spatial locations is finite. We do this for two reasons: (1) biologically the number of cells in a tissue is much larger than the types of cells and genes considered, and (2) we have seen in the proof of Theorem~\ref{thm:CCI to IRD} in~\cite{vanSantvoort2023FromCounts} that technical difficulties arise when supports are taken to be infinite that distract from the main arguments of equivalence. We can summarise the mathematical assumptions for an algorithmic description of SpaCInG as follows.
\begin{assumption}[SpaCInG assumptions]\label{ass:SpaCInG}
	We assume that $\mathcal{S}$, $\mathcal{C}$ and $\mathcal{X}$ are finite sets. Moreover, we assume that the number of vertices at each $x \in \mathcal{X}$ is {asymptotically} the same.
\end{assumption}

To stay in line with RaCInG we will use the same modelling principles in SpaCInG. This means that we will have a fixed number of vertices $n \in \mathbb{N}$ and a fixed number of interactions $\lfloor \mu n \rfloor$ for some $\mu > 0$. We will build up an instance of a CCI network by first using input {2} to independently sample the cell types per spot. Then {3} and {4} will be combined with a distance penalty $d {: \mathcal{X}^2 \times \mathcal{C}^2 \to [1, \infty)}$ into a spatial interaction distribution. From this interaction distribution all $\lfloor \mu n \rfloor$ interactions will be sampled independently, and attached to vertices. This is where input {5} and {6} will be used to choose uniformly which vertices get assigned each interaction, because we have no information about the affinity that each interaction has for each cell type. All in all, this leads us to the following algorithmic description of SpaCInG:

\begin{definition}[SpaCInG as algorithm]\label{def:SpaCInG algorithm}
	We fix a number of vertices $n$ to be generated, an average in-/out-degree $\mu$ and a distance penalty ${d: \mathcal{X}^2 \times \mathcal{C}^2 \to [1, \infty)}$. We construct graphs $\vec{G} = (V, \vec{E})$ as follows:
	\begin{enumerate}[label = \arabic*.]
		\item Set $V = [n]$ and assign all vertices a spatial location $x_v$ as follows:
		\begin{enumerate}[label = \alph*.]
			\item Designate one $y \in \mathcal{X}$ to be the \emph{final} spatial location.
			\item For each $x \in \mathcal{X}$ such that $x \neq y$ assign $\lfloor n / |\mathcal{X}|\rfloor$ vertices the location $x$.
			\item Assign the remaining $n - {(|\mathcal{X}| - 1)} \cdot  \lfloor n / |\mathcal{X}| \rfloor$ vertices to location $y$.
		\end{enumerate}
		\item For each $v \in [n]$ sample an independent type $T_v$ from the distribution $T^{(x)}$ if $x_v = x$.
		\item Sample $\lfloor \mu n \rfloor$ arcs. Each arc is defined by a quadruple $(x, i, y, j)$ with $x, y \in \mathcal{X}$ and $i, j \in \mathcal{C}$. For each arc, this quadruple is sampled independently and proportionally to \[ \frac{\prob\left(C^{(x)} = i \right) \prob\left( C^{(y)} = j\right) \1\{i \to j\}}{ d(\|x - y\|, i, j)}, \] where $\|x - y\|$ denotes the Euclidean distance between locations $x, y \in \mathcal{X}$.
		\item Loop over all arcs, and given their quadruple $(x, i, y, j)$ assign the arc to vertices as follows:
		\begin{enumerate}[label = \alph*.]
			\item First, choose a vertex $v$ uniformly from the set
			\[
			\{ v \in [n] : x_v = x \enspace \wedge \enspace I(T_v, i) = 1\}.
			\]
			If this set is empty, go to the next arc.
			\item Then, independently, choose a vertex $w$ uniformly from the set
			\[
			\{ w \in [n] : x_w = y \enspace \wedge \enspace J(T_{{w}}, j) = 1\}.
			\]
			If this set is empty, go to the next arc.
			\item Set $(v, w) \in \vec{E}$. 
		\end{enumerate}
	\end{enumerate}
	We denote this model by $\texttt{SpaCInG}_{n, \mu}$, omitting the random variables and constants needed in the definition of the model. 
\end{definition}
{
	\begin{remark}
		Because both $\mathcal{X}$ and $\mathcal{C}$ are discrete, there are no assumptions needed on $d$ for this paper. However, in practice it is often assumed that the distance penalty increases in $\|x - y\|$ (see \cite{vanSantvoort2025AnData}).
	\end{remark}
}

To aid in the analysis of SpaCInG, we will end this section by introducing some notation, similarly to the RaCInG case, that will prove to be useful.

\begin{notation}\label{not:spacing mathematical notation}
	We will set $p_i^x := \prob(C^{(x)} = i)$ and $q_t^x := \prob(T^{(x)} = t)$. Furthermore, defining
	\[
	\alpha := \sum_{x \in \mathcal{X}} \sum_{y \in \mathcal{X}} \sum_{i \in \mathcal{C}}\sum_{j \in \mathcal{C}} \frac{p_i^x p_j^y \1\{i \to j\}}{ d(\|x - y\|, i, j)},
	\]
	we also set \[p_{ij}^{xy} := \frac{p_i^x p_j^y \1\{i \to j\}}{\alpha \cdot d(\|x - y\|, i, j)} .\]
\end{notation}

\subsection{Model equivalence}
We expect SpaCInG to be equivalent to IRDs (recall Definition~\ref{def:IRD}), because Definition~\ref{def:SpaCInG algorithm} is closely related to the algorithmic description of RaCInG in Definition~\ref{def:RaCInG algorithm}. To make this formal, we analyse under what conditions RaCInG and SpaCInG produce graph events $\mathcal{Q}_n$ with the same probabilities. Then, we apply Theorem~\ref{thm:CCI to IRD} to find the equivalence.

We can link RaCInG to SpaCInG by identifying {RaCInG's} type distribution as the random vector that describes the spatial location and vertex type distribution simultaneously. This identification will not be exact, because the amount of vertices per spatial location in SpaCInG is fixed, but we will show in this section that under appropriate conditions this identification will generate events with the same probability. We will now explicitly define what we mean {with} the RaCInG model describing SpaCInG.

\begin{definition}[RaCInG describing SpaCInG]\label{def:RaCInG SpaCInG}
	Let $\hat{T}$ be the vertex-type random variable supported on $\mathcal{S}^+ = \mathcal{S} \times \mathcal{X}$ such that $\prob(\hat{T} = (t, x)) = q_t^x / |\mathcal{X}|$. Moreover, let $\hat{C}$ be the arc-colour distribution supported on $\mathcal{C}^+ \times \mathcal{C}^+$ where $\mathcal{C}^+ := \mathcal{C} \times \mathcal{X}$ with probability mass function $\prob(\hat{C} = ((i, x), (j, y))) = p_{ij}^{xy}$. Finally, let $\hat{I}: \mathcal{S}^+ \times \mathcal{C}^+ \to \{0, 1\}$ be the function given by $\hat{I}((t, x), (i, y)) = I(t, i) \1\{x = y\}$ and $\hat{J}: \mathcal{S}^+ \times \mathcal{C}^+ \to \{0, 1\}$ the function given by $\hat{J}((t, x), (i, y)) = J(t, i) \1\{x = y\}$. We will call $\texttt{RaCInG}_{n, \mu}(\hat{T}, \hat{C}, \hat{I}, \hat{J})$ the non-spatial model corresponding to $\texttt{SpaCInG}_{n, \mu}$.
\end{definition}

To {state} our main result, we {are inspired by} Theorem~\ref{thm:CCI to IRD}. Therefore, we need to copy its event assumptions. However, because SpaCInG is not exactly a special instance of RaCInG in Definition~\ref{def:RaCInG SpaCInG}, we need to add an assumption on the RaCInG model that asserts that the probability of $\mathcal{Q}_n$ does not change when a limited number of vertices and edges get added to its output. Note that this is exactly the requirement from Definition~\ref{def: event insensitivity} when $r_n' \neq 0$. We are now in a position to state {Theorem~\ref{thm:SCCI to SIRD}}. To this end, we define
\begin{subequations}
	\begin{align}
		\kappa(t, x, s,y) &= \mu \sum_{(i,j) \in \mathcal{C}_{xy}} \frac{p_{ij}^{xy} I(t, i)J(s, j) }{\lambda_i^x \varrho_j^y }, \text{ with} \label{subeq: SpaCInG kernel} \\
		\lambda_i^x &= \sum_{k \in \mathcal{S}}  q_k^x I(k, i)\label{subeq:newlamba},   \\
		\varrho_j^y &= \sum_{k \in \mathcal{S}}  q_k^y J(k, j)\label{subeq:newrho}, \text{ and}\\
		\mathcal{C}_{xy} &:= \left\{ (i, j) \in {\mathcal{C}^2} : \lambda_i^x > 0 \wedge \varrho_j^y > 0 \right\}.\label{eq:active LR pairs}
	\end{align}   
\end{subequations}
Moreover, define the random vector $(X, T)$ by first realising $X \sim \texttt{Unif}\{1, \ldots, |\mathcal{X}|\}$, and then conditional on $\{X = x\}$ taking a sample from $T^{(x)}$. Note $(X, T)$ has the same probability mass function as $\hat{T}$ in Definition~\ref{def:RaCInG SpaCInG}.
\begin{theorem}[SpaCInG to IRD]\label{thm:SCCI to SIRD}
	Consider an instance of $\texttt{SpaCInG}_{n, \mu}$ such that Assumption~\ref{ass:SpaCInG} is satisfied. Furthermore, let $\mathcal{Q}_n$ be a monotone event that is completely insensitive in $\texttt{RaCInG}_{n, \mu}(\hat{T}, \hat{C}, \hat{I}, \hat{J})$ at rate $(c \log(n) \sqrt{n},  \log(n) \sqrt{n})$ for some $c$ large enough. If for all kernels $\kappa_n'$ that satisfy
	\begin{equation}\label{eq:kernel range SpaCInG}
		\left|\kappa(t,x,s,y) - \kappa_n'(t,x,s,y)  \right| \leq c n^{-1/10},
	\end{equation}
	we have for some $p \in [0, 1]$ that
	\begin{equation}\label{eq:convergence requirement SIRD}
		\prob(\IRD_{n \pm \log(n)\sqrt{n}}((X, T), \kappa_{n \pm \log(n)\sqrt{n}}') \in \mathcal{Q}_n) \to p.
	\end{equation}
	Then, we also have that
	\[
	\prob(\texttt{SpaCInG}_{n, \mu} \in \mathcal{Q}_n) \to p.
	\]
\end{theorem}

Theorem~\ref{thm:SCCI to SIRD} is the analogue to Theorem~\ref{thm:CCI to IRD} for SpaCInG. Again, we will apply the main theorem to counting arcs between vertices with given vertex types, to show how this result can be used to compute graph properties of interest in a biological context. This yields the following proposition, whose proof is found in Section~\ref{sec:proofs_propositions}.

\begin{prop}[Direct spatial interactions]\label{prop:direct spacing}
	Denote by $A_{ts}^{xy}$ the number of arcs from a vertex with type $t \in \mathcal{S}$ at location $x \in \mathcal{X}$ to a vertex with type $s \in \mathcal{S}$ at location $y \in \mathcal{X}$. We have that
	\[
	\frac{A_{ts}^{xy}}{n} \to \frac{\kappa(t,x,s,y) q_t^x q_s^y}{|\mathcal{X}|^2},
	\]
	in probability as $n \to \infty$.
\end{prop}

\section{Proofs and strategies of the main results}\label{sec:main proofs}

\subsection{Main results of Section~\ref{sec:dir to undir}}\label{sec: main result dir to indir}
To prove Theorem~\ref{thm:direct implies undirect} we first recall that we have assumed a relationship between IAG and ASRG. To transform this into a relationship between IEG and ESRG we need to translate events in IEG into events in IAG, and events in ESRG into events in ASRG. This is where Definition~\ref{def: forgetful map} comes into play (see Figure~\ref{fig:proof dir to undir idea}). We note that under $\mathcal{U}$ events will have the same probability if the random graph models involved place their edges/arcs between the same vertices. Thus, the key to connect e.g. IAG and ASRG is by coupling them such that arcs/edges appear between the same vertices.

\begin{figure}
	\centering
	\includegraphics[scale = 0.7]{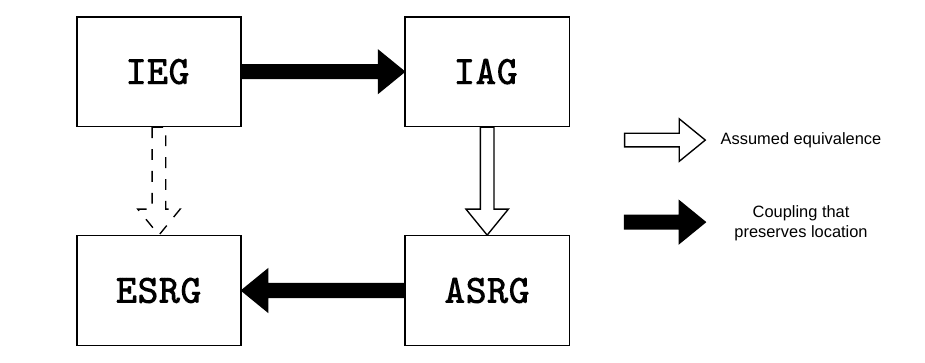}
	\caption{\emph{The chain of equivalences we will execute when proving Theorem~\ref{thm:direct implies undirect}.}}
	\label{fig:proof dir to undir idea}
\end{figure}

The rate of insensitivity ({recall} Definition~\ref{def: event insensitivity}) needed for an event comes from the number of errors that is made when coupling the random graph and digraph model. Such an error is an instance where an arc appears between t{w}o vertices in the directed model, but no edge appears in the undirected model or vice versa. {We call these errors \emph{location errors}.} 
\begin{definition}[Location errors]\label{def:loc coupl err}
	Let $\texttt{RG}_n = ({V_n}, E)$ be a random graph model and $\texttt{RD}_n = ({V_n}, \vec{E})$ a random digraph model defined on the same probability space. Consider the following two random variables:
	\begin{itemize}
		\item[$\Xi^{(1)}_n$ --] The number of unordered vertex-pairs $v, w \in {V_n}$ for which $\{v, w\} \in E$, but $(v, w), (w, v) \notin \vec{E}$.
		\item[$\Xi^{(2)}_n$ --] The number of unordered vertex-pairs $v, w \in {V_n}$ for which $(v, w) \in \vec{E}$ or $(w, v) \in \vec{E}$, but $\{v, w\} \notin E$.
	\end{itemize}
	We then define the \emph{total number of location errors} as $\Xi_n := \Xi^{(1)}_n + \Xi^{(2)}_n$.
\end{definition}
When coupling the random graph models to random digraph models, we will see that $\Xi_n^{(2)} = 0$ can always be achieved. {T}he value of $\Xi_n^{(1)}$ under the coupling will determine the rate of insensitivity ($r_n$ in Definition~\ref{def: event insensitivity}). If we need to translate from the directed model to the undirected model{,} then we would have to add arcs to make the {models} have edges/arcs in the same locations, requiring insensitivity for increase. However, if we need to translate from the undirected model to the directed one we need to remove edges, requiring insensitivity for decrease. The next lemma shows how this argument can be made formal.
\begin{lemma}[Location errors and insensitivity]\label{lem:loc err to event insens}
	Suppose $\texttt{RD}_n$ is a random digraph model and $\texttt{RG}_n$ a random graph model that can be coupled such that $\{\Xi_n^{(1)} \leq r_n \} \cap \{\Xi_n^{(2)} = 0\}$ occurs with probability tending to one {for some (possibly random) rate $r_n$}. Moreover, suppose $\mathcal{Q}_n$ is a graph event such that one of the following holds:
	\begin{enumerate}
		\item $\mathcal{U}^{-1}(\mathcal{Q}_n)$ is insensitive for increase at rate $(r_n, 0)$ on $\texttt{RD}_n$.
		\item $\mathcal{Q}_n$ is insensitive for decrease at rate $(r_n, 0)$ on $\texttt{RG}_n$.
	\end{enumerate} 
	Then,
	\[
	\prob(\texttt{RG}_n \in \mathcal{Q}_n) = \prob(\texttt{RD}_n \in \mathcal{U}^{-1}(\mathcal{Q}_n)) + o(1).
	\]
\end{lemma}

We will now link the two general classes of random graph models through couplings. First, we will link $\texttt{IAG}_n(V_n, \pi_n)$ to $\texttt{IEG}_n(V_n, \pi_n^\Box)$, where
\[
\pi_n^\Box(v, w) = 1 - (1 - \pi_n(v, w))(1 - \pi_n(w, v)) = \pi_n(v, w) + \pi_n(w, v) - \pi_n(v, w) \pi_n(w, v).
\]
It turns out that this is an exact coupling. However, the final product $\pi_n(v, w) \pi_n(w, v)$ is usually of lower order than the other terms. Therefore, it is more natural to couple $\IAG_n(V_n, \pi_n)$ to $\IEG_n(V_n, \pi_n^\circ)$ where $\pi_n^\circ = \pi_n(v, w) + \pi_n(w, v)$. This yields a coupling error. Note that $\Xi_n^{(1)} > 0$ and $\Xi_n^{(2)} = 0$, because $\pi_n^\Box \leq \pi_n^\circ$.

\begin{lemma}[Couple IAG to IEG]\label{lem:IAG to IEG approx}
	Let $\omega(n)$ be an arbitrary function such that $\omega(n) \to \infty$. There exists a coupling between $\texttt{IAG}_n(V_n, \pi_n)$ and $\texttt{IEG}_n(V_n, \pi_n^\circ)$ with $\pi_n^\circ(v, w) = \pi_n(v, w) + \pi_n(w, v)$ such that $\Xi_n^{(1)} \leq \max\{1, \omega(n)(\pi_n^\uparrow n)^2\}$ and $\Xi_n^{(2)} = 0$ with probability tending to one.
\end{lemma}

\begin{remark}
	The maximum with $1$ in the formulation of Lemma~\ref{lem:IAG to IEG approx} is needed in the boundary case where $n \pi_n^\uparrow \to 0$ in probability as $n \to \infty$. Also, note that insensitivity at rate $\omega(n) (\pi_n^\uparrow n)^2$ is too strict. Heuristically, this assumes that miss-couplings occur at every location of the graph with the largest possible probability, since for all $v, w \in V_n$ we have $\pi_n(v, w) \pi_n(w, v) \leq (\pi_n^\uparrow)^2$. Thus in practice it could be possible to work with a less restrictive insensitivity rate.
\end{remark}

The second lemma will couple ASRG to ESRG. Here, we can only resort to an indirect coupling. The idea behind the coupling is that arcs and edges are placed at the same locations; the direction of the arcs does not matter. This coupling would make errors when two arcs are placed between the same vertices (in opposite directions). Then, there is no matching location for the corresponding edge in ESRG. When we have placed $m_n - 1$ arcs already, then this coupling error occurs with a probability bounded by $\nu_n^\uparrow m_n$. Because we are placing $m_n$ arcs in total, the amount of miss-couplings made is bounded by $\nu_n^\uparrow m_n^2$. This is how $\Xi_n^{(1)}$ can be upper-bounded. Note that $\Xi_n^{(1)} > 0$ and $\Xi_n^{(2)} = 0$ for this coupling because a directed graph has more locations to place arcs (i.e., $n(n-1)$) when compared to undirected graphs (i.e., $n(n-1)/2$). We will now formulate the second lemma and subsequently prove Theorem~\ref{thm:direct implies undirect}.

\begin{lemma}[Couple ASRG to ESRG]\label{lem:ASRG to ESRG}
	Let $\omega(n)$ be an arbitrary function such that $\omega(n) \to \infty$, and consider $\ASRG_n(V_n, \nu_n, m_n)$ and suppose that $m_n \nu^\uparrow_n \to 0$ in probability. There exists a coupling between $\ASRG_n(V_n, \nu_n, m_n)$ and $\ESRG_n(V_n, \nu_n^\circ, m)$ with $\nu_n^\circ(v, w) = \nu_n(v,w) + \nu_n(w,v)$ such that $\Xi_n^{(1)} \leq \omega(n) m^2_n \nu_n^\uparrow $ and $\Xi_n^{(2)} = 0$ with probability tending to one.
\end{lemma}
\begin{remark}
	In case of the classical Erd\H{o}s-R\'enyi model where $\Theta(n)$ edges/arcs are placed, the conditions of Lemma~\ref{lem:ASRG to ESRG} would simplify to event {in}sensitivity at any rate $(r_n, 0)$ with $r_n \to \infty$ as expected. This is because in such setting $m_n^2 = \Theta(n^2)$ and $\nu_n^\uparrow = \Theta(1/n^2)$.
\end{remark}
\begin{proof}[\textbf{Proof of Theorem~\ref{thm:direct implies undirect}}]
	We only prove the theorem assuming insensitivity for increase of $\mathcal{U}^{-1}(\mathcal{Q}_n)$ twice. The proof assuming insensitivity for decrease of $\mathcal{Q}_n$ once or twice is analogous (see remark after the proof). We  first use Lemma~\ref{lem:IAG to IEG approx} together with Lemma~\ref{lem:loc err to event insens} and the event insensitivity of $\mathcal{U}^{-1}(\mathcal{Q}_n)$ to conclude
	\[
	\prob(\IEG_n(V_n, \pi_n^\circ) \in \mathcal{Q}_n) = \prob(\IAG_n(V_n, \pi_n) \in \mathcal{U}^{-1}(\mathcal{Q}_n)) + o(1).
	\]
	Next, we invoke \eqref{eq:IAG ASRG equivalence} to translate from $\IAG_n(V_n, \pi_n)$ to $\ASRG_n(V_n, \nu_n, m_n)$. We find
	\begin{equation}\label{eq:dir to undir equiv almost there}
		\prob(\IEG_n(V_n, \pi_n^\circ) \in \mathcal{Q}_n) = \prob(\ASRG_n(V_n, \nu_n, m_n) \in \mathcal{U}^{-1}(\mathcal{Q}_n)) + o(1).
	\end{equation}
	Finally, we use Lemma~\ref{lem:ASRG to ESRG} with Lemma~\ref{lem:loc err to event insens} and the event insensitivity of $\mathcal{U}^{-1}(\mathcal{Q}_n)$ to conclude
	\[
	\prob(\ASRG_n(V_n, \nu_n, m_n) \in \mathcal{U}^{-1}(\mathcal{Q}_n)) = \prob(\texttt{ESRG}_n(V_n, \nu_n^\circ, m_n) \in \mathcal{Q}_n) + o(1),
	\]
	which gives the desired result when substituted into \eqref{eq:dir to undir equiv almost there}. Namely,
	\[
	\prob(\IEG_n(V_n, \pi_n^\circ) \in \mathcal{Q}_n) = \prob(\ESRG_n(V_n, \nu_n^\circ, m_n) \in \mathcal{Q}_n) + o(1).
	\]
\end{proof}
\begin{remark}
	The different combinations of insensitivity requirements all have the same goal: bridging the black arrows in Figure~\ref{fig:proof dir to undir idea}. No matter whether we invoke insensitivity for increase for $\mathcal{U}^{-1}(\mathcal{Q}_n)$ or insensitivity for decrease for $\mathcal{Q}_n$, we can still apply Lemma~\ref{lem:loc err to event insens}. 
\end{remark}

With Theorem~\ref{thm:direct implies undirect} at our disposal, we now want to use Theorem~\ref{thm:CCI to IRD} to prove Corollary~\ref{cor:CCI to IRG}. To do this, we need three additional lemmas with the following purposes:
\begin{enumerate}
	\item Theorem~\ref{thm:CCI to IRD} can only be applied when the events $\mathcal{Q}_n$ are monotone. Since we will be considering events of the type $\mathcal{U}^{-1}(\mathcal{Q}_n)$ in the directed setting, we need to show that these are monotone when the input event $\mathcal{Q}_n$ is monotone.
	\item We need to understand the behaviour of kernels satisfying \eqref{eq:kernel bound undirect}, since this behaviour will dictate part of the event insensitivity required in Theorem~\ref{thm:direct implies undirect} through $\pi_n^\uparrow$.
	\item We need to understand the mass function in Example~\ref{ex:ESRG and ASRG}\ref{ex:RaCInG}, since it will dictate the other part of the event insensitivity required in Theorem~\ref{thm:direct implies undirect} through $\nu_n^\uparrow$.
\end{enumerate}
We will now formulate the three lemmas that tackle the above purposes. Thereafter, we will prove Corollary~\ref{cor:CCI to IRG}.
\begin{lemma}[Monotonicity]\label{lem:monotonicity preservation}
	If $\mathcal{Q}_n$ is a monotone event, then $\mathcal{U}^{-1}(\mathcal{Q}_n)$ is too.
\end{lemma}

\begin{lemma}[Kernel bound]\label{lem:kernel bound}
	Under Assumption~\ref{ass:finite support}, any kernel $\kappa_n$ satisfying \eqref{eq:kernel bound undirect} is bounded. 
\end{lemma}

\begin{lemma}[Mass bound]\label{lem:mass bound}
	For the mass function $\nu_n$ in Example~\ref{ex:ESRG and ASRG}\ref{ex:RaCInG} we have under Assumption~\ref{ass:finite support} for some fixed $c > 0$ that $\nu_n \leq c/n^2$ with probability tending to one conditioned on $V_n$.
\end{lemma}

Note that Lemma~\ref{lem:kernel bound} ensures that $\pi_n^\uparrow \leq c^\prime / n$ for some large enough constant $c^\prime > 0$. Hence, the required insensitivity through Lemma~\ref{lem:IAG to IEG approx} simplifies to rate $(\omega(n), 0)$ for any $\omega(n) \to \infty$. Similarly, Lemma~\ref{lem:mass bound} shows that $\nu_n^\uparrow \leq c / n^2$ for large enough $c > 0$. Thus, since we place $m_n \approx \mu n$ arcs, the required insensitivity due to Lemma~\ref{lem:ASRG to ESRG} only has constant rate. All in all, this show{s} that the total event insensitivity needed has rate $\omega(n)$ for any $\omega(n) \to \infty$. This is why said assumption appears in Corollary~\ref{cor:CCI to IRG}, but not in Theorem~\ref{thm:CCI to IRD}.

\begin{proof}[\textbf{Proof of Corollary~\ref{cor:CCI to IRG}}]
	First, we recall that $\IRG_n(T, \kappa_n)$ for all $\kappa_n$ satisfying \eqref{eq:kernel bound undirect} is a special instance of $\IAG_n(V_n, \pi_n)$ (see Example~\ref{ex:IEG and IAG}\ref{ex:IRD}). In this setting, we can use Lemma~\ref{lem:kernel bound} to conclude that the event $\mathcal{A}_n := \{\pi_n^\uparrow \leq c_{\kappa} / n\}$ occurs with probability tending to one for some constant $c_{\kappa} > 0$. Similarly, we recall that RaCInG is a special instance of $\ASRG_n(V_n, \nu_n, m_n)$ (see Example~\ref{ex:ESRG and ASRG}\ref{ex:RaCInG}). In this setting, the event $\mathcal{B}_n := \{ \nu_n^\uparrow \leq c_{\nu} / n^2\}$ occurs with probability tending to one for some $c_{\nu} > 0$ through Lemma~\ref{lem:mass bound}.
	
	Using the events $\mathcal{A}_n$ and $\mathcal{B}_n$, we now calculate the rate $r_n$ in \eqref{eq:Translation rate dir undir} required to apply Theorem~\ref{thm:direct implies undirect}. Since we know that $\mathcal{A}_n \cap \mathcal{B}_n$ occurs with probability tending to one, the bounds on $\nu_n^\uparrow$ and $\pi_n^\uparrow$ imply that the required insensitivity rate to apply the theorem equals
	\[
	r_n = \left\lceil \max\left\{ \omega'(n)(n c_\kappa / n)^2, \; \omega'(n) \mu^2 n^2 c_\nu / n^2 , \; 1 \right\} \right\rceil,
	\]
	where $\omega'(n) \to \infty$ is an arbitrarily slowly increasing function (different from $\omega(n)$ in the statement of the corollary). In particular, note that $r_n = \Theta( \omega'(n))$. Therefore, we can pick this function $\omega'(n)$ such that $r_n = \omega(n)$. We can now proceed to apply Theorem~\ref{thm:direct implies undirect}.
	
	Because $\mathcal{U}^{-1}(\mathcal{Q}_n)$ is insensitive for increase at rate $\omega(n)$ on $\texttt{IRD}_{n}(T,\kappa_n)$ for all kernels $\kappa_n$ satisfying \eqref{eq:kernel bound undirect}, Lemma~\ref{lem:IAG to IEG approx} together with Lemma~\ref{lem:loc err to event insens} shows us that
	\[
	\prob(\texttt{IRG}_n(T, \kappa_n^\circ) \in \mathcal{Q}_n) = \prob(\IRD_n(T, \kappa_n) \in \mathcal{U}^{-1}(\mathcal{Q}_n)) + o(1),
	\]
	Hence, since $\prob(\texttt{IRG}_n(T, \kappa_n^\circ) \in \mathcal{Q}_n) \to p$ for all $\kappa_n^\circ(t,s) = \kappa_n(t,s) + \kappa_n(s,t)$, we know that $\prob(\IRD_n(T, \kappa_n) \in \mathcal{U}^{-1}(\mathcal{Q}_n)) \to p$ for all $\kappa_n$ satisfying \eqref{eq:kernel bound undirect} as well. Moreover, given that $\mathcal{Q}_n$ is a monotone event, we know that $\mathcal{U}^{-1}(\mathcal{Q}_n)$ also is a monotone event too through Lemma~\ref{lem:monotonicity preservation}.  Thus, we can apply Theorem~\ref{thm:CCI to IRD} to conclude that $\prob(\texttt{RaCInG}_{n, \mu}(T, C, I, J) \in \mathcal{U}^{-1}(\mathcal{Q}_n)) \to p$. This shows that \eqref{eq:IAG ASRG equivalence} is satisfied.
	
	Now, since the previously constructed event $\mathcal{A}_n \cap \mathcal{B}_n$ happens with probability tending to one, we immediately also find that $m_n \nu_n^\uparrow \to 0$ in probability. Thus, since $\mathcal{U}^{-1}(\mathcal{Q}_n)$ is insensitive at rate $\omega(n)$ on $\texttt{RaCInG}_{n, \mu}(T, C, I, J)$, we can finally apply Theorem~\ref{thm:direct implies undirect}. Doing this yields the desired result {when recalling that $\texttt{URaCInG}_{n, \mu}(T, C, I, J)$ is the same as $\texttt{ESRG}_n(V_n, \nu_n^\circ, \lfloor \mu n\rfloor)$ (Definition~\ref{def:uracing}).}
	\[
	\prob({\texttt{URaCInG}_{n, \mu}(T, C, I, J)}  \in \mathcal{Q}_n) = \prob(\texttt{RaCInG}_{n, \mu}(T, C, I, J) \in \mathcal{U}^{-1}(\mathcal{Q}_n)) + o(1) \to p.
	\]
\end{proof}

\subsection{Main results of Section~\ref{sec:SpaCInG}}\label{sec: main result spacing}

The main idea behind the proof of Theorem~\ref{thm:SCCI to SIRD} {is} to apply Theorem~\ref{thm:CCI to IRD} to $\texttt{RaCInG}_{n, \mu}(\hat{T}, \hat{C}, \hat{I}, \hat{J})$. The key obstacle is that SpaCIn{G} is not exactly the same as $\texttt{RaCInG}_{n, \mu}(\hat{T}, \hat{C}, \hat{I}, \hat{J})$. Hence, the proof will consist {of two parts:}
\begin{enumerate}[label = \textbf{Part \Roman*.}]
	\item We show that SpaCInG is asymptotically the same as RaCInG.
	\item We check the conditions from Theorem~\ref{thm:CCI to IRD} and apply it.
\end{enumerate}
For Part I we want to identify SpaCInG as an instance of RaCInG by identifying the location and type of a vertex in SpaCInG as the vertex type in RaCInG. This will not be exact, because the number of vertices with a location $x \in \mathcal{X}$ is \emph{fixed} {in SpaCInG}, while in RaCInG it will be random. Hence, we need to show that the added randomness in RaCInG does not alter the probability of an event $\mathcal{Q}_n$. This is where the insensitivity assumption of Theorem~\ref{thm:SCCI to SIRD} comes into play.

To show that the SpaCInG model and RaCInG with the adapted vertex type distribution are roughly the same, we will use a coupling. In this coupling SpaCInG will be matched to instances of RaCInG with $n \pm \log(n)\sqrt{n}$ vertices. In this coupling the graphs will match on a large subset of the vertices. The vertices/edges where they do not match will be rewired and removed. Due to the event insensitivity, this process will not influence the limiting probability. See Figure~\ref{fig:coupl idea} for a visual depiction of the coupling. We will also make the coupling more formal in the following definition.

\begin{figure}
	\centering
	\includegraphics[width=0.99\linewidth]{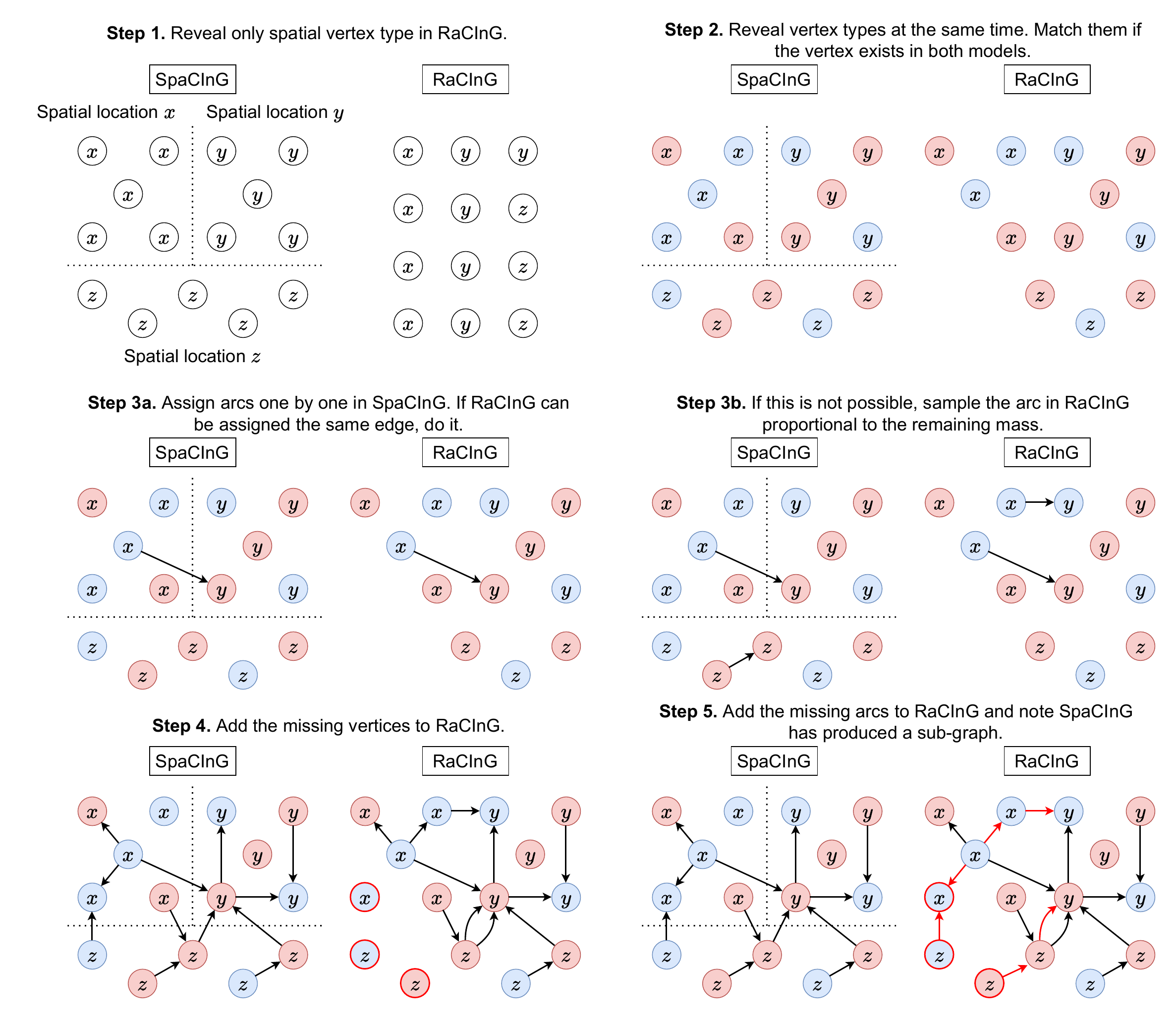}
	\caption{\emph{The visualisation of the coupling from Definition~\ref{def:coupling racing spacing}.} }
	\label{fig:coupl idea}
\end{figure}

\begin{notation}
	We set $n^\pm :=n \pm \log(n) \sqrt{n}$.
\end{notation}

\begin{definition}[Coupling between RaCInG and SpaCInG]\label{def:coupling racing spacing}
	Let $\texttt{RaCInG}_{n^-, \mu}^\uparrow$ denote the instance of RaCInG from Definition~\ref{def:RaCInG SpaCInG} with at most $\log(n)\sqrt{n}$ vertices added after which at most $c\log(n) \sqrt{n}$ arcs get added for some $c > 0$ (cf. Definition~\ref{def: event insensitivity}). We will couple $\texttt{SpaCInG}_{n, \mu}$ to $\texttt{RaCInG}_{n^-, \mu}^\uparrow$ by generating instances of them through the following joint process:
	\begin{enumerate}[label = \arabic*.]
		\item Generate the spatial parts of the vertex-types in $\texttt{RaCInG}_{n^-, \mu}$. {Then, reveal the deterministic spatial parts of the vertex-types in $\texttt{SpaCInG}_{n,\mu}$}. Finally, match each vertex in $\texttt{RaCInG}_{n^-, \mu}$ to a vertex in $\texttt{SpaCInG}_{n,\mu}$ with the same spatial vertex-type. If this is not possible, then generate the models independently and end the coupling.
		\item Generate the non-spatial parts of the vertex-types in $\texttt{RaCInG}_{n^-, \mu}$. If two vertices in $\texttt{SpaCInG}_{n,\mu}$ and $\texttt{RaCInG}_{n^-, \mu}$ are linked, {assign the linked vertex in $\texttt{SpaCInG}_{n,\mu}$ the same non-spatial vertex-type}. If no linked vertex exists for a vertex in $\texttt{SpaCInG}_{n,\mu}$, give it an independent {non-spatial} vertex-type.
		\item Assign arcs to vertices one by one, giving them the same colours {$(x,i,y,j)$} in both models. Always start {placing} each arc in $\texttt{SpaCInG}_{n,\mu}$. After this step, the realisation of $\texttt{SpaCInG}_{n,\mu}$ is complete.
		\begin{enumerate}[label = \alph*.]
			\item If the arc is placed in between two vertices that are also present in $\texttt{RaCInG}_{n^-, \mu}$, then place the arc between the linked vertices.
			\item If one of the vertices is \emph{not} present in $\texttt{RaCInG}_{n^-, \mu}$, {then choose a vertex pair that can accept an arc with colour $(x,i,y,j)$ independently uniformly at random in $\texttt{RaCInG}_{n^-, \mu}$.}
		\end{enumerate}
		\item For each vertex in $\texttt{SpaCInG}_{n,\mu}$ without a linked vertex in $\texttt{RaCInG}_{n^-, \mu}$, add a vertex to $\texttt{RaCInG}_{n^-, \mu}$ and give it the same location and type as the vertex in $\texttt{SpaCInG}_{n,\mu}$.
		\item Add all arcs incident to the new vertices in $\texttt{RaCInG}_{n^-, \mu}$, such that each arc in $\texttt{SpaCInG}_{n,\mu}$ is also present in $\texttt{RaCInG}_{n^-, \mu}$. This is the output of $\texttt{RaCInG}_{n^-, \mu}^\uparrow${.}
	\end{enumerate}
\end{definition}

\begin{remark}
	Note that $\texttt{SpaCInG}_{n, \mu}$ is a sub-graph of $\texttt{RaCInG}_{n^-, \mu}^\uparrow$. By noting that the analysed event $\mathcal{Q}_n$ is increasing, this will show that probabilities involving $\texttt{SpaCInG}_{n, \mu}$ are a lower-bound of the same probabilities in $\texttt{RaCInG}_{n^-, \mu}^\uparrow$. Of course, we also need an upper-bound. This will be provided through an analogous coupling similar to the one in Definition~\ref{def:coupling racing spacing} where $\texttt{RaCInG}_{n^+, \mu}^\downarrow$ becomes the sub-graph of $\texttt{SpaCInG}_{n, \mu}$. We will not give this coupling, because its only differences compared to Definition~\ref{def:coupling racing spacing} is a role-reversal of RaCInG and SpaCInG, and a change in the number of vertices from $n^-$ to $n^+$ in RaCInG.
\end{remark}

After Part I has been executed, we know that SpaCInG is roughly the same as RaCInG. The final thing that we then still need to do in Part II is apply Theorem~\ref{thm:CCI to IRD} to the instances of RaCInG and the IRDs. At this point we will also need the fact that $\mathcal{Q}_n$ is increasing, and convergence {in the} range specified in \eqref{eq:kernel range SpaCInG}.

We now make the heuristics formal. To execute Part I of the proof we will analyse the coupling in Definition~\ref{def:coupling racing spacing}. Specifically, we need to show two things. Firstly, we need to show that the procedure in Definition~\ref{def:coupling racing spacing} is a coupling, and that moreover the output $\texttt{SpaCInG}_{n, \mu}$ network is a sub-graph of $\texttt{RaCInG}_{n^-, \mu}^\uparrow$. This will allow us to translate from SpaCInG to RaCInG and exploit monotonicity of $\mathcal{Q}_n$ in the proof of Theorem~\ref{thm:SCCI to SIRD}. Secondly, we need to show that the amount of extra vertices and arcs in $\texttt{RaCInG}_{n^-, \mu}^\uparrow$ is bounded with probability tending to one by $c \log(n) \sqrt{n}$ for $c > 0$ large enough, so that we can use the assumed insensitivity in the proof of Theorem~\ref{thm:SCCI to SIRD}. These two facts are captured by the following lemmas, after which we prove the {Theorem~\ref{thm:SCCI to SIRD}}.

\begin{lemma}[Subset coupling]\label{lem:correct coupling}
	The procedure in Definition~\ref{def:coupling racing spacing} is a coupling of $\texttt{SpaCInG}_{n, \mu}$ and $\CCI_{n^-, \mu}^\uparrow$ such that
	\[
	\prob(\texttt{SpaCInG}_{n, \mu} \subseteq \CCI_{n^-, \mu}^\uparrow) \to 1.
	\]
\end{lemma}
\begin{remark}
	The alternative analogous coupling between $\texttt{SpaCInG}_{n, \mu}$ and $\CCI_{n^+, \mu}^\downarrow$ outlined in the remark after Definition~\ref{def:coupling racing spacing} is also a coupling for the same reasons as outlined in the proof of Lemma~\ref{lem:correct coupling}. For this coupling we have that
	\[
	\prob(\CCI_{n^+, \mu}^\downarrow \subseteq \texttt{SpaCInG}_{n, \mu}) \to 1.
	\]
\end{remark}

\begin{lemma}[Extra arcs]\label{lem:extra arcs}
	Consider $\texttt{SpaCInG}_{n, \mu}$ and $\CCI_{n^-, \mu}^\uparrow$ under the coupling in Definition~\ref{def:coupling racing spacing}. Denote by $A_n$ the number of arcs in $\CCI_{n^-, \mu}^\uparrow$ that are not in $\texttt{SpaCInG}_{n, \mu}$. We have for some $c > 0$ large enough that $\prob(A_n \leq c \log(n) \sqrt{n}) \to 1$.
\end{lemma}
\begin{remark}
	Between the coupled instances of $\texttt{SpaCInG}_{n, \mu}$ and $\CCI_{n^+, \mu}^\downarrow$ we also have that $\prob(A_n \leq c \log(n) \sqrt{n}) \to 1$. 
\end{remark}

\begin{proof}[\textbf{Proof of Theorem~\ref{thm:SCCI to SIRD}}] 
	We assume without loss of generality that $\mathcal{Q}_n$ is increasing. We start by computing an upper bound on the desired probability in SpaCInG and show that it converges to $p$. First, by Lemma~\ref{lem:correct coupling} we note that with probability tending to one $\texttt{SpaCInG}_{n, \mu} \subseteq \CCI_{n^-, \mu}^\uparrow$. Thus, because $\mathcal{Q}_n$ is increasing, we have that $\texttt{SpaCInG}_{n, \mu} \in \mathcal{Q}_n$ implies that $\CCI_{n^-, \mu}^\uparrow \in \mathcal{Q}_n$. Therefore, we conclude
	\begin{equation}\label{eq: spacing convergence proof upper bound}
		\prob(\texttt{SpaCInG}_{n, \mu} \in \mathcal{Q}_n) \leq \prob(\CCI_{n^-, \mu}^\uparrow \in \mathcal{Q}_n) + o(1).
	\end{equation}
	Next, we note that $\CCI_{n^-, \mu}^\uparrow$ has $\log(n)\sqrt{n}$ more vertices than $\CCI_{n^-, \mu}$. Moreover, we note that each arc that does not match between $\texttt{SpaCInG}_{n, \mu}$ and $\CCI_{n^-, \mu}^\uparrow$ under Definition~\ref{def:coupling racing spacing} corresponds to a two arc difference between $\CCI_{n^-, \mu}^\uparrow$ and $\CCI_{n^-, \mu}$ (because of Step 3b and 5 in Definition~\ref{def:coupling racing spacing}). 
	
	Thus, using Lemma~\ref{lem:extra arcs}, we know (with probability tending to one) that we need to delete at most $\log(n)\sqrt{n}$ vertices and $2 c' \log(n)\sqrt{n}$ arcs to transform $\CCI_{n^-, \mu}^\uparrow$ into $\CCI_{n^-, \mu}$ for some $c' > 0$. Because $\mathcal{Q}_n$ is assumed to be insensitive at rate $(c \log(n) \sqrt{n}, \log(n)\sqrt{n})$ for $c > 0$ large enough in $\CCI_{n^-, \mu}$, transforming $\CCI_{n^-, \mu}^\uparrow$ into $\CCI_{n^-, \mu}$ does not alter the limiting probability. Using this in \eqref{eq: spacing convergence proof upper bound} yields
	\begin{equation}\label{eq: spacing convergence proof final upper obund}
		\prob(\texttt{SpaCInG}_{n, \mu} \in \mathcal{Q}_n) \leq \prob(\CCI_{n^-, \mu} \in \mathcal{Q}_n) + o(1).
	\end{equation}
	Next, we need to show that $\prob(\CCI_{n^-, \mu} \in \mathcal{Q}_n) \to p$. To do this, we note for all kernels satisfying \eqref{eq:kernel range SpaCInG} that we assumed
	\[
	\prob(\IRD_{n^-}((X, T), \kappa_{n^-}') \in \mathcal{Q}_n) \to p.
	\]
	{We note that all kernels $\kappa_{n^-}'$ in \eqref{eq:kernel range SpaCInG} exactly form the range given by \eqref{eq:kern sequence CCI} when we consider RaCInG from Definition~\ref{def:RaCInG SpaCInG} in Theorem~\ref{thm:CCI to IRD}.} Thus, since $\mathcal{Q}_n$ is increasing we may conclude from Theorem~\ref{thm:CCI to IRD} that $\prob(\CCI_{n^-, \mu} \in \mathcal{Q}_n) \to p$, implying from \eqref{eq: spacing convergence proof final upper obund} that
	\begin{equation}\label{eq:limsup proof SpaCInG}
		\limsup_{n\to\infty} \prob(\texttt{SpaCInG}_{n, \mu} \in \mathcal{Q}_n) \leq p.
	\end{equation}
	
	In an analogous way we can also investigate the lower bound. Using a coupling between $\texttt{SpaCInG}_{n, \mu}$ and $\CCI_{n^+, \mu}^\downarrow$ similar to Definition~\ref{def:coupling racing spacing} we find through the fact that $\mathcal{Q}_n$ is increasing and the remark underneath Lemma~\ref{lem:correct coupling} that
	\begin{equation}\label{eq: spacing convergence proof lower bound}
		\prob(\texttt{SpaCInG}_{n, \mu} \in \mathcal{Q}_n) \geq \prob(\CCI_{n^+, \mu}^\downarrow \in \mathcal{Q}_n) + o(1).
	\end{equation}
	Analysing the amount of mismatched arcs between $\texttt{SpaCInG}_{n, \mu}$ and $\CCI_{n^+, \mu}^\downarrow$ using the remark underneath Lemma~\ref{lem:extra arcs} and insensitivity of $\mathcal{Q}_n$ in $\CCI_{n^+, \mu}$ allows us to conclude that
	\[
	\prob(\texttt{SpaCInG}_{n, \mu} \in \mathcal{Q}_n) \geq \prob(\CCI_{n^+, \mu} \in \mathcal{Q}_n) + o(1).
	\]
	Finally, using monotonicity of $\mathcal{Q}_n$ and Theorem~\ref{thm:CCI to IRD} allows us to deduce
	\[
	\liminf_{n\to \infty}\prob(\texttt{SpaCInG}_{n, \mu} \in \mathcal{Q}_n) \geq p.
	\]
	Together with \eqref{eq:limsup proof SpaCInG} we indeed conclude that
	\[
	\lim_{n \to \infty}\prob(\texttt{SpaCInG}_{n, \mu} \in \mathcal{Q}_n)  = p.
	\]
\end{proof}

\begin{remark}
	For decreasing events $\mathcal{Q}_n$ the proof is analogous. The only difference is that the inequalities \eqref{eq: spacing convergence proof upper bound} and \eqref{eq: spacing convergence proof lower bound} would flip.
\end{remark}

\section{Proofs of lemmas and propositions}\label{sec:lem proofs}

\subsection{Proofs of lemmas in Section~\ref{sec: main result dir to indir}}

We start by proving Lemma~\ref{lem:loc err to event insens}. Its proof showcases how insensitivity (see Definition~\ref{def: event insensitivity}) is used to turn a location coupling that is almost exact into an exact one.

\begin{proof}[\textbf{Proof of Lemma~\ref{lem:loc err to event insens}}]
	We will only prove the lemma assuming insensitivity of $\mathcal{U}^{-1}(\mathcal{Q}_n)$. The proof assuming insensitivity of $\mathcal{Q}_n$ is analogous. See the remark after the proof. We will consider $\texttt{RD}_n$ and $\texttt{RG}_n$ only under the coupled probability space. We define ${\mathcal{A}_n} := \{\Xi_n^{(1)} \leq r_n\} \cap \{\Xi_n^{(2)} = 0\}$. First, we use the union bound to obtain
	\begin{equation}\label{eq:RG up and low event bound}
		\prob(\{\texttt{RG}_n \in \mathcal{Q}_n\} \cap \mathcal{A}_n) \leq \prob(\texttt{RG}_n \in \mathcal{Q}_n)\leq \prob(\{\texttt{RG}_n \in \mathcal{Q}_n\} \cap \mathcal{A}_n) + \prob(\neg \mathcal{A}_n).
	\end{equation}
	Using the fact that $\mathcal{A}_n$ occur{s} with probability tending to one, we can now conclude that
	\begin{equation}\label{eq:union bound event insens prob}
		\prob(\texttt{RG}_n \in \mathcal{Q}_n) = \prob(\{\texttt{RG}_n \in \mathcal{Q}_n\} \cap \mathcal{A}_n) + o(1). 
	\end{equation}
	We note that under the coupling ${\mathcal{A}_n}$ implies that $\texttt{RG}_n$ has at most $r_n$ edges more tha{n} $\mathcal{U}(\texttt{RD}_n)$, but that apart from that both graphs agree. We can define $\texttt{RD}_n^\uparrow$ from an output of $\texttt{RD}_n$ by adding the $r_n$ arcs such that $\mathcal{U}(\texttt{RD}_n^\uparrow) = \texttt{RG}_n$. Thus, we can conclude
	\[
	\prob(\texttt{RG}_n \in \mathcal{Q}_n) = \prob(\{\mathcal{U}(\texttt{RD}_n^\uparrow) \in \mathcal{Q}_n\} \cap \mathcal{A}_n) + o(1). 
	\]
	Now, using the definition of $\mathcal{U}$ we can conclude that the above equality is the same as
	\[
	\prob(\texttt{RG}_n \in \mathcal{Q}_n) = \prob(\{\texttt{RD}_n^\uparrow \in \mathcal{U}^{-1}(\mathcal{Q}_n)\} \cap \mathcal{A}_n) + o(1). 
	\]
	Next, using {the assumed insensitivity of $\mathcal{U}^{-1}(\mathcal{Q}_n)$} we note that $\{\texttt{RD}_n^\uparrow \in \mathcal{U}^{-1}(\mathcal{Q}_n)\}$ and $\{\texttt{RD}_n \in \mathcal{U}^{-1}(\mathcal{Q}_n)\}$ have {asymptotically} the same probability. Hence, we find
	\[
	\prob(\texttt{RG}_n \in \mathcal{Q}_n) = \prob(\{\texttt{RD}_n \in \mathcal{U}^{-1}(\mathcal{Q}_n)\} \cap \mathcal{A}_n) + o(1). 
	\]
	Finally, we can remove the inclusion of $\mathcal{A}_n$ by using similar arguments as in \eqref{eq:RG up and low event bound}. This yields
	\[
	\prob(\texttt{RG}_n \in \mathcal{Q}_n) = \prob(\texttt{RD}_n \in \mathcal{U}^{-1}(\mathcal{Q}_n)) + o(1).
	\]
\end{proof}
\begin{remark}
	To prove Lemma~\ref{lem:loc err to event insens} assuming insensitivity of $\mathcal{Q}_n$ we would need to transform $\texttt{RG}_n$ into $\texttt{RG}_n^\downarrow$ (cf. Definition~\ref{def: event insensitivity}) instead of changing $\texttt{RD}_n$ into $\texttt{RD}_n^\uparrow$ every time. To make $\mathcal{U}(\texttt{RD}_n)$ and $\texttt{RG}_n$ agree, the heuristic idea is that we can either add arcs to $\texttt{RD}_n$ -- creating $\texttt{RD}_n^\uparrow$ -- or remove them from $\texttt{RG}_n$ -- creating $\texttt{RG}_n^\downarrow$.
\end{remark}

Since the remaining lemmas in Section~\ref{sec: main result dir to indir} deal with the four broad classes of models, and since all these models first generate a vector $V_n$ before they create a graph, we need to be clear where in the graph generation process we are. We need to introduce notation that specifies whether $V_n$ has already been generated or not.

\begin{notation}
	We will denote the probability measure for a model conditioned on $V_n$ by $\prob_n$. Similarly, we denote the conditional expectation of a model conditioned on $V_n$ by $\expec_n$. The unconditional operators remain $\prob$ and $\expec$.
\end{notation}

We first prove Lemma~\ref{lem:IAG to IEG approx}. We will do this by first giving an exact coupling between IAG and IEG. Thereafter, we check how close this is to the IEG formulation in Lemma~\ref{lem:IAG to IEG approx} to identify the required edge insensitivity. We will first give the lemma that will give us the exact connection between IAG and IEG, after which we will prove Lemma~\ref{lem:IAG to IEG approx}.

\begin{lemma}[Exact IAG to IEG]\label{lem:IAG to IEG exact}
	Consider $\IAG_n(V_n, \pi_n)$ with some vector $V_n$ and probability function $\pi_n$. Define a new probability function $\pi_n^\square$ as follows:
	\[
	\pi_n^\square(v, w) = 1 - (1 - \pi_n(v, w))(1 - \pi_n(w, v)). 
	\]
	We can couple $\mathcal{U}(\IAG_n(V_n, \pi_n))$ and $\IEG_n(V_n, \pi_n^{{\square}})$ such that the number of location errors (see Definition~\ref{def:loc coupl err}) is zero.
\end{lemma}
\begin{proof}
	We first note, because $V_n$ is the same in both models, we can consider $\IAG_n(V_n, \pi_n)$ and $\IEG_n(V_n, \pi^\square_n)$ to have the same vertex-sets. Now, given the vertices we can couple the graphs in $\IAG$ and $\IEG$. For this we define $(I_{vw})_{v, w \in V_n{, v < w}}$ to be the sequence of edge indicators in $\IEG_n(V_n, \pi_n^\square)$. Similarly, we denote by $(J_{vw})_{v, w \in V_n{, v \neq w}}$ the sequence of arc indicators in $\IAG_n(\pi_n; V_n)$.
	
	Note in both cases that the sequences $(I_{vw})_{v, w \in V_n{, v < w}}$ and $(J_{vw})_{v, w \in V_n{, v \neq w}}$ are independent, since vertices in $V_n$ have already been sampled. Moreover, every graph outputted by $\IEG_n(V_n, \pi_n^\square)$ and $\IAG_n(V_n, \pi_n)$ is captured fully by the sequences $(I_{vw})_{v, w \in V_n{, v < w}}$ and $(J_{vw})_{v, w \in V_n{, v \neq w}}$, respectively. Finally, note from Definition~\ref{def: forgetful map} that $\mathcal{U}(\IAG_n(\pi_n; V_n))$ will have an edge indicator sequence $(\widehat{I}_{vw})_{v, w \in V_n{, v < w}}$ given by
	\[
	\widehat{I}_{vw} = \1\{J_{vw} + J_{wv} \geq 1\}.
	\]
	All in all, if we prove an exact coupling between the sequences $(I_{vw})_{v, w \in V_n{, v < w}}$ and $(\widehat{I}_{vw})_{v, w \in V_n{, v < w}}$, then we have found the coupling between $\IEG_n(V_n, \pi_n^\square)$ and $\IAG_n(V_n, \pi_n)$ with zero location errors.
	
	To construct this coupling, we note that $(\widehat{I}_{vw})_{v, w \in V_n{, v < w}}$ is still an independent sequence, because the random vector sequence $(J_{vw}, J_{wv})_{v, w\in V_n{, v \neq w}}$ is also independent. This is due to the independence in $(J_{vw})_{v, w \in V_n{, v \neq w}}$. Moreover, we note that $\widehat{I}_{vw} = 0$ if and only if both $J_{vw} = 0$ and $J_{wv} = 0$. Therefore,
	\begin{align*}
		\prob(\widehat{I}_{vw} = 1) &= 1 - \prob(J_{vw} = 0, J_{wv} = 0) = 1 - \prob(J_{vw} = 0) \prob(J_{wv} = 0),\\
		&= 1 - (1 - \pi_n(v, w))(1 - \pi_n(w, v)) = \pi_n^\square(v, w).
	\end{align*}
	In other words, we see that the indicators $(I_{vw})_{v, w \in V_n{, v < w}}$ and $(\widehat{I}_{vw})_{v, w \in V_n{, v < w}}$ are the same. Therefore, an exact coupling between the two models exists. 
\end{proof}

\begin{proof}[\textbf{Proof of Lemma~\ref{lem:IAG to IEG approx}}]
	Fix a probability function $\pi_n$. The proof will consist of three main steps:
	\begin{enumerate}[label = \textbf{\Roman*.}]
		\item Conditional on $V_n$, we construct a coupling between $\IEG_n(\pi_n^\circ; V_n)$ and $\IEG_n(\pi_n'; V_n)$ with \begin{equation}\label{eq:succes prob phi IAG}
			\pi_n'(v, w)=  \pi_n^\circ(v, w) - \pi_n(v, w)\pi_n(w, v).
		\end{equation}
		\item Under the coupling from Step I, we identify what $\Xi_n^{(1)}$ and $\Xi_n^{(2)}$ look like (see Definition~\ref{def:loc coupl err}). Here we will apply Lemma~\ref{lem:IAG to IEG exact}.
		\item We bound the objects from Step II with probability tending to one.
	\end{enumerate}
	
	\paragraph{Step I.} Given $V_n$, set $(I_{vw}^\circ)_{v < w}$ to be the sequence of edge indicators in $\IEG_n(V_n, \pi_n^\circ)$, and $(I_{vw}')_{v < w}$ be the corresponding sequence in $\IEG_n(V_n, \pi_n')$. Since the realized vertices of both models are equal, we note that both sequences are independent, and that $I_{vw}^\circ$ has almost the same success probability as $I_{vw}'$, as can be seen in \eqref{eq:succes prob phi IAG}. Since each graph is fully determined by its edge sequence after realizing $V_n$, we will couple both graphs by constructing a coupling $(\hat{I}_{vw}^\circ, \hat{I}_{vw}')$ of $I_{vw}^\circ$ and $I_{vw}'$ for fixed $v < w$, but keeping the sequence of random vectors $(\hat{I}_{vw}^\circ, \hat{I}_{vw}')_{v < w}$ independent.
	
	The coupling of the edge indicators in the two models is given by the joint probabilities below.
	\begin{equation}\label{eq:IEG coupling}
		\begin{split}
			\prob_n(\hat{I}_{vw}^\circ = 0, \hat{I}_{vw}' = 0) &= 1 - \pi_n^\circ(v, w),\\
			\prob_n(\hat{I}_{vw}^\circ = 1, \hat{I}_{vw}' = 0) &= \pi_n(v, w) \pi_n(w, v),\\
			\prob_n(\hat{I}_{vw}^\circ = 1, \hat{I}_{vw}' = 1) &=  \pi_n^\circ(v, w) - \pi_n(v, w) \pi_n(w, v).
		\end{split}
	\end{equation}
	Note that the marginal probabilities of the indicators indeed coincide with the desired probabilities, and note in particular that a coupling mismatch will ensue only in cases where $\IEG_n(V_n, \pi_n^\circ)$ has edges, but $\IEG_n(V_n, \pi_n')$ has not. Therefore, $\Xi_n^{(2)} = 0$ ({recall} Definition~\ref{def:loc coupl err}).
	
	\paragraph{Step II.} Set $E'$ to be edge set of $\IEG_n(V_n, \pi_n')$, $E^\circ$ the edge set of $\IEG_n(V_n, \pi_n^\circ)$ and $\vec{E}$ the arc set of $\IAG_n(V_n, \pi_n)$. Because there exists an exact coupling between $\IEG_n(V_n, \pi_n')$ and $\IAG_n(V_n, \pi_n)$ (by Lemma~\ref{lem:IAG to IEG exact}), we know in particular that under said coupling it holds that
	\begin{equation}\label{eq:consequence exact coupling IAG and IEG}
		\1\left\{\{v, w\} \in E'\right\} - \1\left\{(v, w) \in \vec{E}\right\} -  \1\left\{(w, v) \in \vec{E}\right\} = 0,
	\end{equation}
	for all $v, w \in V_n$. Therefore, if we chain the coupling established in Step I to the exact coupling in Lemma~\ref{lem:IAG to IEG exact}, we can compare the $\IEG_n(V_n, \pi_n^\circ)$ to the model $\IAG_n(V_n, \pi_n)$. Under this chained coupling, using \eqref{eq:consequence exact coupling IAG and IEG}, we find that
	\begin{align*}
		\1\left\{\{v, w\} \in E^\circ\right\} &- \1\left\{(v, w) \in \vec{E}\right\} -  \1\left\{(w, v) \in \vec{E}\right\}\\ &= \1\left\{\{v, w\} \in E^\circ\right\} - \1\{\{v, w\} \in E'\}.
	\end{align*}
	Thus, recalling that $\Xi_n^{(2)} = 0$, we can bound
	\begin{equation}\label{eq:error counter IAG IEG}
		\Xi_n^{(1)} \leq \tfrac12 \sum_{v, w \in V_n} \left| \1\left\{\{v, w\} \in E^\circ\right\} - \1\left\{\{v, w\} \in E'\right\} \right| = \tfrac12 \sum_{v, w \in V_n} \left|\hat{I}_{vw}^\circ - \hat{I}_{vw}' \right|.
	\end{equation}
	Now, by \eqref{eq:IEG coupling}, $|\hat{I}_{vw}^\circ - \hat{I}_{vw}' |$ takes the value $1$ if and only if $\hat{I}_{vw}^\circ = 1$ and $\hat{I}_{vw}' = 0$. Note for each edge $\{v, w\}$ this happens independently with probability $\pi_n(v, w) \pi_n(w, v)$ conditional on $V_n$. Therefore, if we set $(J_{vw})_{v < w}$ to be an independent sequence, where $J_{vw} \sim \texttt{Bern}(\pi_n(v, w) \pi_n(w, v))$, we can bound \eqref{eq:error counter IAG IEG} as
	\[
	\Xi_n^{(1)} \preceq \sum_{v < w} J_{vw}.
	\]
	We will now show that $\Xi_n^{(1)} \leq  2\omega(n) \max\{1 , (n \pi_n^\uparrow)^2\}$ probability tending to one, to prove the lemma.

	\paragraph{Step III.} To show that with probability tending to one $\Xi_n^{(1)} \leq  2\omega(n) \max\{1 , (n \pi_n^\uparrow)^2\}$, we compute the expected value of $\Xi_n^{(1)}$ and apply the Chernoff bound for binomial random variables (see Theorem 2.21 in \cite{vanderHofstad2016RandomNetworks1}). First note that conditioned on $V_n$
	\[
	J_{vw} \preceq \hat{J}_{vw} \sim \texttt{Bern}( (\pi_n^\uparrow)^2).
	\]
	Using this domination, we have that $\Xi_n^{(1)} \preceq \hat{\Xi}_n^{(1)}$ for some new random variable $\hat{\Xi}_n^{(1)} \sim \texttt{Bin}(n^2, (\pi_n^\uparrow)^2)$. 
	In particular, this allows us to bound $\expec_n[\Xi_n^{(1)}] \leq (n \pi_n^\uparrow)^2$. This bound puts us in a situation where we can apply the Chernoff bound for binomial random variables. Particularly, given the function $\omega(n) \to \infty$, we can use Theorem 2.21 in \cite{vanderHofstad2016RandomNetworks1} to conclude that
	\begin{equation*}
		\begin{split}
			\prob_n&\left(\Xi_n^{(1)} \geq \expec_n[\hat{\Xi}_n^{(1)}] + \omega(n) \max\left\{1, \sqrt{\expec_n[\hat{\Xi}_n^{(1)}]}\right\} \right)\\ &\leq \prob_n\left(\hat{\Xi}_n^{(1)} \geq \expec_n[\hat{\Xi}_n^{(1)}] + \omega(n) \max\left\{1, \sqrt{\expec_n[\hat{\Xi}_n^{(1)}]}\right\} \right),\\
			&\leq \exp\left( - \frac{\omega(n)^2 \max\left\{1, \sqrt{\expec_n[\hat{\Xi}_n^{(1)}]}\right\}^2}{2\left(\expec_n[\hat{\Xi}_n^{(1)}] + \tfrac13 \omega(n) \max\left\{1, \sqrt{\expec_n[\hat{\Xi}_n^{(1)}]}\right\}\right)} \right),\\
			&\leq \exp(- 3 \omega(n)/2 ).
		\end{split}
	\end{equation*}
	Now, we define the event $\mathcal{B}_n := \{\Xi_n^{(1)} \leq  2\omega(n) \max\{1 , (n \pi_n^\uparrow)^2\} \}$. Note that $\expec_n[\hat{\Xi}_n^{(1)}] \leq \omega(n) \max\{1, (n \pi_n^\uparrow)^2\}$, due to the multiplication with $\omega(n) \to \infty$, and that \[\omega(n) \max\left\{1, \sqrt{\expec_n[\hat{\Xi}_n^{(1)}]}\right\} \leq \omega(n) \max\{1, (n \pi_n^\uparrow)^2\},\] due to the fact that taking the square root makes the expected value smaller if the maximum is not attained at one. Thus, the previously derived exponential inequalities show that
	\[
	\prob(\neg \mathcal{B}_n) = \expec[\prob_n(\neg \mathcal{B}_n)] \leq \expec[ \exp(- 3 \omega(n) / 2)] = \exp(- 3 \omega(n) / 2) \to 0.
	\]
\end{proof}

Now, we move on to prove Lemma~\ref{lem:ASRG to ESRG}. As stated in Section~\ref{sec:dir to undir}, we will prove this lemma by providing an explicit coupling between ASRG and ESRG and counting how many errors this coupling makes. The idea behind the coupling is to place arcs in ASRG between the same vertices a{s} edges in ESRG. We start by defining the coupling. The rules of this coupling are illustrated in Figure~\ref{fig:ESRG ASRG coupling}.

\begin{figure}
	\centering
	\includegraphics[scale = 0.6]{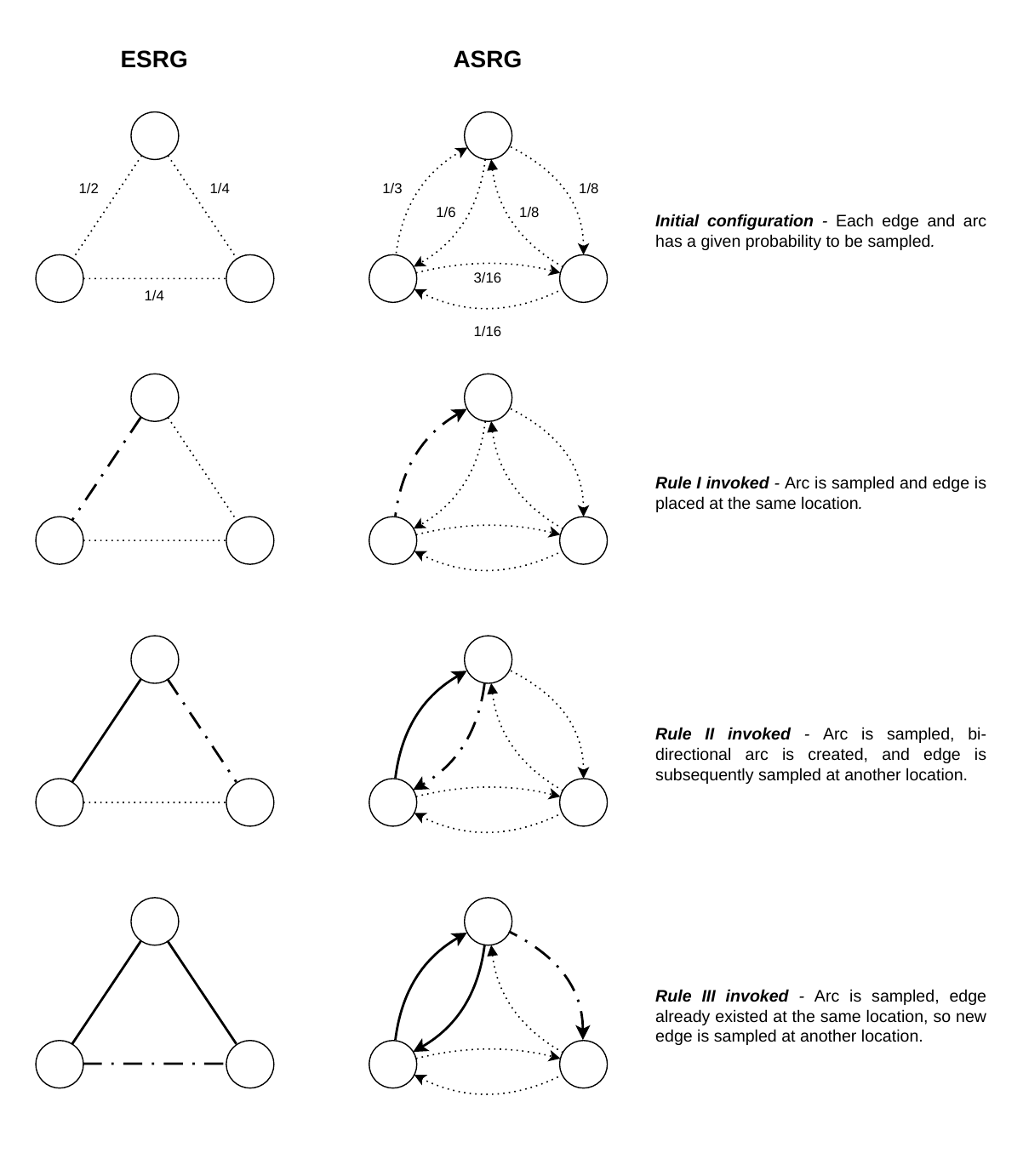}
	\caption{\emph{Illustration of the coupling in Definition~\ref{def:coupling ASRG ESRG}. In each line a new step in the graph generation process is taken. The dash-dotted lines are the newly generated arcs/edges.}}
	\label{fig:ESRG ASRG coupling}
\end{figure}

\begin{definition}[Approximate location coupling ASRG vs. ESRG]\label{def:coupling ASRG ESRG}
	Sample a vector $V_n$ and a (directed) mass function $\nu_n$ on $V_n \times V_n$. Also, define the (undirected) mass function $\nu_n^\circ$ from $\nu_n$ by setting $\nu_n^\circ(v ,w ) = \nu_n(v, w) + \nu_n(w, v)$. From these objects we will generate two graphs: a directed graph $\vec{G}_n = (V_n, \vec{E}_n)$ and an undirected graph $G_n = (V_n, E_n)$ with $m$ arc{s}/edges.
	
	We build both graphs one arc/edge at a time, where $\vec{G}_n$ follows the building algorithm of $\ASRG_n(V_n, {\nu_n}, m)$ in Definition~\ref{def:ESRG and ASRG}. The graph $G_n$ will be built simultaneously according to the following three rules:
	
	\begin{enumerate}[label = \Roman*.]
		\item If arc $(v, w)$ is generated in $\vec{G}_n$, and the arc $(w, v)$ does not exist yet, then generate $\{v, w\}$ in $G_n$.
		\item If the arc $(v, w)$ is generated in $\vec{G}_n$, and the arc $(w, v)$ already exists, then generate an edge in $G_n$ by sampling edges independently from $\nu_n^\circ$ until a non-existing edge is found.
		\item If arc $(v, w)$ is generated in $\vec{G}_n$, and the edge $\{v, w\}$ in $G_n$ already exists, then generate an edge in $G_n$ according to the procedure in rule II. 
	\end{enumerate}
\end{definition}

Before we can analyse the coupling, we first need to show that Definition~\ref{def:coupling ASRG ESRG} is a proper coupling between $\ASRG_n(V_n, \nu_n, m)$ and $\ESRG_n(V_n, \nu_n^\circ, m)$. To do so, note that an edge $\{v, w\}$ gets sampled proportional to $\nu_n(v, w) + \nu_n(w, v)$ in rule I, since both arcs $(v, w)$ and $(w, v)$ might lead to edge $\{v, w\}$. Additionally, rule II and III will trivially generate an edge $\{v, w\}$ proportional to $\nu_n^\circ(v, w)$. To make these observations formal, we state and prove an auxiliary lemma.

\begin{lemma}[ESRG and ASRG coupling exists]\label{lem:ESRG ASRG coupling exists}
	The procedure specified in Definition~\ref{def:coupling ASRG ESRG} is a coupling between $\ASRG_n(V_n, \nu_n, m)$ and $\ESRG_n(V_n, \nu_n^\circ, m)$.
\end{lemma}
\begin{proof}
	Start for both models by generating the same vertex type vector $V_n$. We then note that graph $\vec{G}_n$ in Definition~\ref{def:coupling ASRG ESRG} will trivially follow the generation procedure of $\ASRG_n(V_n, \nu_n, m)$. To show $G_n$ follows the generation procedure of $\ESRG_n(V_n, \nu_n^\circ, m)$ we will inductively generate an arc in $\vec{G}$ and show that the corresponding edge generated in $G_n$ will have the same probability as an edge being sampled according to the procedure in Definition~\ref{def:ESRG and ASRG}.
	
	\paragraph{Base case.} As our base case we will generate the first edge/arc pair. Since no edges or arcs have been generated yet, we note that only rule I could apply. In this case, the edge $\{v, w\}$ in $G_n$ will be generated if and only if either the arc $(v, w)$ or the arc $(w, v)$ is generated in $\vec{G}$. This event will happen with probability $\nu_n(v, w) + \nu_n(w, v) = \nu_n^\circ(v, w)$. Therefore, when generating the first edge/arc-pair the marginal probabilities in both models are still respected.
	
	\paragraph{Induction step.} Suppose that $s$ arcs/edge-pairs have already been generated, and assume marginally both graph $G_n$ and $\vec{G}_n$ follow the procedure in Definition~\ref{def:ESRG and ASRG}. We will consider the generation of arc $s + 1$. For this, denote by $\vec{E}_s$ the set of already generated arcs in $\vec{G}_n$, and denote by $\mathcal{A}(\vec{E}_s) \subset V_n \times V_n$ the set of occupied arc locations, viewed from the perspective of $\vec{G}_n$. In other words, $\mathcal{A}(\vec{E}_s) := \{(v, w) \in V_n \times V_n : (v, w) \in \vec{E}_s \vee (w, v) \in \vec{E}_s\}$. We will now check whether the procedure in Definition~\ref{def:ESRG and ASRG} is still respected when either of the rules in Definition~\ref{def:coupling ASRG ESRG} is invoked. 
	
	Note that the only interesting rule to check is rule I, since in both rule II and III the edge generated in $G_n$ is independent from the arc generated in $\vec{G}_n$. Moreover, the generated edge in $G_n$ has probability proportional to the remaining mass in $\nu_n^\circ$. Therefore, these rules will trivially satisfy Definition~\ref{def:ESRG and ASRG}. 
	
	Now, if the rule I is invoked, we note that either of the arcs $(v, w)$ and $(w, v)$ will result in the generation of $\{v, w\}$ in $G_n$. Therefore the raw mass of $\{v, w\}$ is, similar to the base case, given by $\nu_n^\circ(v, w)$. However, when rule I is invoked, we are conditioning on the event that the generated arc in $\vec{G}_n$ does not come from the set $\mathcal{A}(\vec{E}_s) \backslash \vec{E}_s$. Therefore, since arcs from $\mathcal{A}(\vec{E}_s) \backslash \vec{E}_s$ and $\vec{E}_s$ are off-limits in the arc-generation process in this step, the probability of $\{v, w\}$ being generated in $G_n$ under rule I is
	\begin{equation}\label{eq:rule I mass}
		\frac{\nu_n^\circ(v, w)}{1 - \sum_{(x, y) \in \mathcal{A}(\vec{E}_s)} \nu_n(x, y)} .
	\end{equation}
	Now, the idea is to transform the $\nu_n$ mass function in the denominator into the function $\nu_n^\circ$. To do this, we will define $V_n^{\vee 2}$ to be a set of unordered pairs of vertices from $V_n$. Using this definition, noting that $\mathcal{A}(\vec{E}_s)$ contains arc $(x, y)$ if and only if it contains $(y, x)$, we may rewrite \eqref{eq:rule I mass}
	\[
	\frac{\nu_n^\circ(v, w)}{1 - \sum_{(x, y) \in \mathcal{A}(\vec{E}_s) \cap V_n^{\vee 2}}[ \nu_n(x, y) + \nu_n(y, x)]} = \frac{\nu_n^\circ(v, w)}{1 - \sum_{(x, y) \in \mathcal{A}(\vec{E}_s) \cap V_n^{\vee 2}}\nu_n^\circ(x, y)}.
	\]
	Note that $\mathcal{A}(\vec{E}_s) \cap V_n^{\vee 2}$ is exactly the set of occupied edges in $G_n$. Thus, the probability that $\{v, w\}$ gets generated under rule I is given by $\nu_n^\circ(v, w)$, normalised over the masses of the leftover edges that can be chosen, as required in $\ESRG_n(V_n, \nu_n^\circ, m)$. Finally, this means that the marginals of $\vec{G}_n$ and $G_n$ in Definition~\ref{def:coupling ASRG ESRG} indeed follow the rules in Definition~\ref{def:ESRG and ASRG}.
\end{proof}

Lemma~\ref{lem:ESRG ASRG coupling exists} ensures that we have found a coupling. To calculate the rate of insensitivity required, we need to bound the number of times there is a location with an edge in $G_n$ but no arc in $\vec{G}_n$ and vice versa under the coupling in Definition~\ref{def:coupling ASRG ESRG}. To do this, we consider the coupling as a dynamic process, and note that a location error is only induced when rule II is invoked; in rule I the arc and edge are matched exactly, and in rule III a miss-match is both created and resolved.

We will first specify the process which counts miss-coupled edges/arcs under Definition~\ref{def:coupling ASRG ESRG} as the graphs are constructed. Thereafter, we will identify exactly when miss-couplings occur and show that this process induces a super-martingale. This will allow us to compute properties of this \emph{error counting process} more easily.

\begin{definition}[Error counting process]\label{def:error process}
	Consider the process in Definition~\ref{def:coupling ASRG ESRG} and denote by $G_n^{(s)}$ and $\vec{G}_n^{(s)}$ the intermediate graphs of $\ESRG_n(V_n, \nu^\circ_n,m_n)$ and $\ASRG_n(V_n, \nu_n,m_n)$ after $s$ arcs have been generated, respectively. We define the \emph{error counting process} $(\Psi_{s})_{1 \le {s} \le m_n}$ by letting $\Psi_s$ denote the number of location errors between these graphs. This is the number of times that an edge $\{v, w\}$ is present in $G_n^{(s)}$, but both $(v, w)$ and $(w, v)$ are not present in $\vec{G}_n^{(s)}$ together with the number of times $(v, w)$ or $(w, v)$ is present in $\vec{G}_n^{(s)}$, but $\{v, w\}$ is not present in $G_n^{(s)}$.
\end{definition}
\begin{remark}
	Note for fixed $n$ and $m_n$ that $\Psi_{m_n} = \Xi_n$, since we have found all coupling-errors in the graph once all edges/arcs have been placed.
\end{remark}

\begin{lemma}[Conditional error process]\label{lem:error process}
	We have that $\Psi_{s + 1} - \Psi_s$ is the indicator that rule II is invoked when generating arc/edge $s + 1$ in Definition~\ref{def:coupling ASRG ESRG}.
\end{lemma}
\begin{proof}
	We first note that whenever rule I is invoked no coupling error is created, since both arcs are placed at the same locations. Thus, under the execution of rule I we have that $\Psi_{s + 1} = \Psi_s$. Next, we see that rule II will create exactly one additional error. This is, because in the undirected model an additional edge is placed, while in the directed version of the model a bidirectional arc is created. This bidirectional arc is in the same location as a previously placed edge. Therefore, the newly generated edge in the undirected version of the model forms the coupling-error. This means that $\Psi_{s + 1} = \Psi_s + 1$ when rule II is invoked.
	
	Finally, note when rule III is invoked both a new miss-coupling is created, while another one is resolved. Since an arc is generated in the same place as an already existing edge, that miss-coupling is resolved, reducing the number by one. However, since we are also placing an additional edge, we are creating a new miss-coupling as well. We know that this new edge cannot be generated in a location where there already is an arc in the directed version of the model, since from rules I--III we see that arcs are only generated at locations where edges already exist (or an edge is immediately placed at the same location). Thus, we see that in case of rule {III} we have that $\Psi_{s + 1} = \Psi_s$. All in all, we see that only the application of rule II will increase the number of miss-couplings by one. Therefore, $\Psi_{s + 1}$ given $\Psi_s$ is the indicator that rule II is invoked.
\end{proof}

\begin{lemma}[Error super-martingale]\label{lem:error martingale}
	Consider $\ASRG_n(V_n, \nu_n, m)$ conditioned on $V_n$. We have that the process $(M_s)_{s \leq {m}}$ given by
	\[
	M_s = \frac{\Psi_s -  \sum_{i = 1}^s \prod_{j = i + 1}^s b_j a_i}{\prod_{i = 1}^s b_i},
	\]
	is a conditional super-martingale given $V_n$. Here, the sequences $(a_i)_{i \geq 1}$ and $(b_i)_{i \geq 1}$ are given by
	\[
	a_i = \frac{i}{1/\nu^\uparrow_n - i}, \qquad b_i = 1 - \frac{2}{1/\nu^\uparrow_n - i}.
	\]
\end{lemma}
\begin{proof}
	The proof will consist of the following two steps:
	\begin{enumerate}[label = \Roman*.]
		\item We will find an upper-bound on probability of $\Psi_{s + 1}$ increasing by one given $\Psi_s$.
		\item We use this to prove that $M_s$ is a super-martingale.
	\end{enumerate}
	A general observation we will often use through the proof, is the fact that $\nu_n^\uparrow \geq \nu_n(v, w)$ for all $v, w \in V_n$.
	
	\paragraph{Step I.} From Lemma~\ref{lem:error process} it follows that at step $s$, given the value of $\Psi_{s}$, an additional coupling error is made only when rule II needs to be invoked. Note that this rule is only invoked when an arc gets generated at the same location as an already generated arc, creating a bidirectional arc. If we denote by $\vec{E}_s$ the set of already placed arcs after $s$ arc generation steps, then we define $\vec{E}_s^- := \{(v, w) \in \vec{E}_s : (w, v) \not\in \vec{E}_s\}$ as the set of arcs that have not yet become a bidirectional arc. Under these definitions, we can write the probability of creating a bidirectional arc when generating arc number $s + 1$ as 
	\[
	\frac{\sum_{(v, w) \in \vec{E}_s^- } \nu_n(v, w)}{1 - \sum_{(v, w) \in \vec{E}_s} \nu_n(v, w)}.
	\]
	Now, note that $| \vec{E}_s| = s$ since $s$ arcs have been placed. Moreover, note that $| \vec{E}_s^- | = s - 2\Psi_{s}$. The factor two appears, since each coupling error (cf. Definition~\ref{def:error process}) corresponds to a bidirectional arc by Lemma~\ref{lem:error process}. Both these arcs need be removed from $\vec{E}_s$ to find $\vec{E}_s^-$. Using this, we find the following error upper-bound at step $s + 1$:
	\[
	\frac{\sum_{(v, w) \in \vec{E}_s^- } \nu_n(v, w)}{1 - \sum_{(v, w) \in \vec{E}_s} \nu_n(v, w)} \leq \frac{(s - 2\Psi_s)  \nu_n^\uparrow}{1 - s \nu^\uparrow_n} = \frac{s - 2\Psi_s}{1/\nu^\uparrow_n - s}.
	\]
	
	\paragraph{Step II.} Denote by $I_i$ the indicator that rule {II} had to be executed when generating arc/edge-pair $i$. From the result of Step I we know that conditional on $\Psi_s$ we have that $I_{s + 1} \preceq \texttt{Bern}((s - 2 \Psi_s)/(1/{\nu^\uparrow_n} - s))$. Using both these facts we will now compute $\expec[M_{s + 1} \;|\; M_{s}]$ for fixed $s$, establishing the martingale property. For this we note that $M_{s}$ is in the sigma algebra generated by $\Psi_s$, since it is only off-set by deterministic constants. Therefore, we can rewrite
	\begin{align*}
		\expec_n[M_{s + 1} \;|\; M_{s}] 
		&= \frac{\expec_n[\Psi_{s + 1} \;|\; M_{s}] - \sum_{i = 1}^s \prod_{j = i + 1}^s b_j a_i}{\prod_{i = 1}^s b_i}\\
		&= \frac{\expec_n[\Psi_{s + 1} \;|\; \Psi_{s}]  - \sum_{i = 1}^s \prod_{j = i + 1}^s b_j a_i}
		{\prod_{i = 1}^s b_i}\\
		&= \frac{\expec_n[\Psi_{s + 1} - \Psi_s \;|\; \Psi_{s}] + \Psi_s  - \sum_{i = 1}^s \prod_{j = i + 1}^s b_j a_i}{\prod_{i = 1}^s b_i},\\
		&= \frac{\expec_n[I_{s + 1} \;|\; \Psi_{s}] + \Psi_s  - \sum_{i = 1}^s \prod_{j = i + 1}^s b_j a_i}{\prod_{i = 1}^s b_i},
	\end{align*}
	where we used Lemma~\ref{lem:error process} for the last step.
	
	Next, using the previously established conditional stochastic domination, we upper-bound the conditional expectation as
	\begin{equation}\label{eq:condition prob martingale}
		\begin{split}
			\expec_n[M_{s + 1} \;|\; M_{s}] &\leq \frac{\frac{s - 2\Psi_s}{ 1/\nu^\uparrow_n - s} + \Psi_s - \sum_{i = 1}^s \prod_{j = i + 1}^s b_j a_i}{\prod_{i = 1}^s b_i},\\
			&= \frac{\frac{s}{1/\nu^\uparrow_n - s} + (1 - \frac{2}{1/\nu^\uparrow_n - s}) \Psi_s - \sum_{i = 1}^s \prod_{j = i + 1}^s b_j a_i}{\prod_{i = 1}^s b_i}, \\ &= \frac{a_s + b_s \Psi_s - \sum_{i = 1}^s \prod_{j = i + 1}^s b_j a_i}{\prod_{i = 1}^s b_i}.
		\end{split}
	\end{equation}
	Now, we note that we can rewrite the sum and product combination as follows:
	\[
	\sum_{i = 1}^s \prod_{j = i + 1}^s b_j a_i = a_s + \sum_{i = 1}^{s - 1} \prod_{j = i + 1}^{s} b_j a_i.
	\]
	Substituting it into \eqref{eq:condition prob martingale} then yields
	\[
	\expec_n[M_{s + 1} \;|\; M_{s}] \leq \frac{b_s \Psi_s - \sum_{i = 1}^{s - 1} \prod_{j = i + 1}^{s} b_j a_i}{\prod_{i = 1}^s b_i}.
	\]
	Finally, we rewrite the leftover expression by factoring the $b_s$-term and find the desired result
	\[
	\expec_n[M_{s + 1} \;|\; M_{s}] \leq \frac{b_s \left( \Psi_s - \sum_{i = 1}^{s - 1} \prod_{j = i + 1}^{s - 1} b_j a_i \right)}{b_s \prod_{i = 1}^{s - 1} b_i} = M_{s}.
	\]
\end{proof} 

\begin{remark}
	Note that $(M_s)_{s \geq 0}$ is a ``regular'' martingale if and only if $\nu_n(v, w) = c {\nu^\uparrow_n}$ for some constant $c$. This is only possible when ${\nu^\uparrow_n} = \Theta({1 / n^2} )$, which corresponds to a case where we sample edges/arcs uniformly at random. The only model in which this happens, is the classical (directed) Erd\H{o}s-R\'enyi model.
\end{remark}

Having established the super-martingale through Lemma~\ref{lem:error martingale}, we are now in a position to prove Lemma~\ref{lem:ASRG to ESRG}. The proof will mainly consist of deriving a probabilistic bound on $\Xi_n$ by applying the martingale maximal inequality, and then using this to guide the desired edge insensitivity. Before presenting the proof of Lemma~\ref{lem:ASRG to ESRG}, we will formulate and prove one final auxiliary lemma that bounds the complex growth speed of the martingale $M_s$ by a simpler expression (that equals the required edge insensitivity rate in the statement of Lemma~\ref{lem:ASRG to ESRG}).

\begin{lemma}[Martingale growth speed]\label{lem:martingale growth speed}
	Assume that $m_n = o(f(n))$ for some function $f(n)$. We have that
	\[
	\sum_{s = 1}^{m_n} \frac{s}{f(n) - s} \prod_{r = s + 1}^{m_n} \left[1 - \frac{2}{f(n) - r}\right] \leq \frac{{2} m_n^2}{f(n)}.
	\]
\end{lemma}
\begin{proof}
	For ease of notation we will omit the {subscript} $n$ from $m_n$. First, by noting that $s / (f(n) - s) \leq m / (f(n) - m)$ and that $ 2/(f(n) - r) \geq  2/f(n)$ we can bound the sum as
	\[
	\sum_{s = 1}^m \frac{s}{f(n) - s} \prod_{r = s + 1}^m \left[1 - \frac{2}{f(n) - r}\right] \leq \frac{m}{f(n) - m} \sum_{s = 1}^m \prod_{r = s + 1}^m \left[ 1 - \frac{2}{f(n)} \right].
	\]
	Then, by computing the product and reversing the sum's order, we find a partial sum of the geometric series.
	\begin{align*}
		\frac{m}{f(n) - m} \sum_{s = 1}^m \prod_{r = s + 1}^m &\left[ 1 - \frac{2}{f(n)} \right] \leq  \frac{m}{f(n) - m} \sum_{s = 1}^m \left[  1 - \frac{2}{f(n)} \right]^{m - s},\\
		&= \frac{m}{f(n) - m} \sum_{k = 0}^{m - 1} \left[  1 - \frac{2}{f(n)} \right]^{k} = \frac{m f(n) (1 - (1 - 2/f(n))^m)}{2(f(n) - m)}.
	\end{align*}
	Next, we will apply Bernoulli's inequality in the numerator of the previously derived upper-bound to find an expression close to the desired result.
	\[
	\frac{m f(n) (1 - (1 - 2/f(n))^m)}{2(f(n) - m)} \leq \frac{{2} m^2}{2(f(n) - m)} {= \frac{ m^2}{f(n) - m}}.
	\]
	Finally, since $m = o(f(n))$ we know that for $n$ large we will have that $f(n) - m \geq f(n)/2$. Thus, we obtain the desired result:
	\[
	\sum_{s = 1}^m \frac{s}{f(n) - s} \prod_{r = s + 1}^m \left[1 - \frac{2}{f(n) - r}\right] \leq  \frac{m^2}{f(n) - m} \leq \frac{{2}m^2}{f(n)}.
	\]
\end{proof}

\begin{proof}[\textbf{Proof of Lemma~\ref{lem:ASRG to ESRG}}]
	First, given a coupled directed and undirected random graph model, we recall that $\Xi_n^{(1)}$ is the number of unordered vertex pairs $v$ and $w$ such that $\{v, w\}$ is present in the undirected model, but both $(v, w)$ and $(w, v)$ is not in the directed model. Similarly, $\Xi_n^{(2)}$ are the number of instances where $(v, w)$ or $(w, v)$ is present in the directed model, but $\{v, w\}$ is not in the undirected model. See Definition~\ref{def:loc coupl err}.
	
	Consider $\ASRG_n(V_n, \nu_n, m)$ and $\ESRG_n(V_n, \nu_n^\circ, m)$ for which we couple the vectors $V_n$ to be equal. The proof boils down to first noting that $\Xi_n^{(2)} = 0$ by Definition~\ref{def:coupling ASRG ESRG}, and then bounding $\Xi_n^{(1)} = \Psi_m$ (cf. Definition~\ref{def:error process}). We analyse $\Psi_m$ through the maximal inequality (see \cite{Baldi2017StochasticExercises}; Theorem 5.3). After we do this, we use the coupling specified in Definition~\ref{def:coupling ASRG ESRG}. We can do this, because Lemma~\ref{lem:ESRG ASRG coupling exists} guarantees Definition~\ref{def:coupling ASRG ESRG} is a correct coupling between $\ASRG_n(V_n, \nu_n, m)$ and $\ESRG_n(V_n, \nu_n^\circ, m)$. The proof will now have the following two steps:
	\begin{enumerate}[label = \textbf{\Roman*.}]
		\item Conditioned on $V_n$, we will bound $\Psi_m$.
		\item We use the resulting probabilistic bound to derive the desired bound on $\Xi_n^{(1)}$ with probability tending to one.
	\end{enumerate}
	
	\paragraph{Step I.} 
	Recall the definition of the sequences $(a_i)_{i \geq 1}$, $(b_i)_{i \geq 1}$ and the super-martingale $(M_s)_{s \geq 0}$ from Lemma~\ref{lem:error martingale}. The maximal inequality on $M_s$ shows for any constant $\alpha > 0$ that
	\begin{equation}\label{eq:max inequality}
		\alpha  \cdot \prob_n\left( \sup_{0 \leq s \leq m} M_s \geq \alpha \right) \leq \expec_n[M_0] + \expec_n[M_m^-],
	\end{equation}
	where $M_m^- := \max\{ - M_m, 0\}$. Note that $\Psi_s$ is stochastically increasing in $s$ with $\Psi_0 = 0$. Therefore, we can bound $\expec_n[M_0]$ and $\expec_n[M_m^-]$ as follows:
	\[
	\expec_n[M_0] + \expec_n[M_m^-] \leq \left( \prod_{s = 1}^m b_s \right)^{-1} \sum_{s = 1}^m a_s \prod_{r = s + 1}^m b_r .
	\]
	Now, we set $\alpha = 1 /(r_n \prod_{s = 1}^m b_s)$ for some sequence $r_n$. With this, we can find a bound on the growth of $\Psi_s$ using \eqref{eq:max inequality}.
	\begin{align*}
		\prob_n\left( \Xi_n^{(1)} \geq  r_n^{-1} + \sum_{s = 1}^m a_s \prod_{r = s + 1}^m b_r \right) &\leq \prob_n\left( \sup_{0 \leq s \leq m} M_s \geq  \left( r_n \prod_{s = 1}^m b_s\right)^{-1} \right), \\ &\leq r_n \sum_{s = 1}^m a_s \prod_{r = s + 1}^m b_r \leq {2} r_n m_n^2 \nu_n^\uparrow.
	\end{align*}
	Here, we used the definition of $M_m$ for the first step and then applied \eqref{eq:max inequality} after dividing both sides by $\alpha$. Also, the final inequality is due to Lemma~\ref{lem:martingale growth speed} when taking $f(n) = 1/\nu_n^\uparrow$. To apply Lemma~\ref{lem:martingale growth speed} we need that $m_n \nu_n^\uparrow \to 0$ in probability. Note that this upper-bound implies that $\prob_n(\Xi_n^{(1)} \geq r_n^{-1} + m_n^2 \nu_n^\uparrow) \leq {2} r_n m_n^2 \nu_n^\uparrow$.
	
	\paragraph{Step II.} Fix a function $\omega(n) \to \infty$ arbitrarily slowly. We will set \[r_n = 2/(m_n^2 \nu_n^\uparrow \omega(n)),\] and note that conditioned on $V_n$ we have that
	\[
	\prob_n(\Xi_n^{(1)} \geq \omega(n) m_n^2 \nu_n^\uparrow) \leq \prob_n(\Xi_n^{(1)} \geq \tfrac12 \omega(n) m_n^2 \nu_n^\uparrow + m_n^2 \nu_n^\uparrow) = \prob_n(\Xi_n^{(1)} \geq r_n^{-1} + m_n^2 \nu_n^\uparrow).
	\]
	Now, to show that $\{\Xi_n^{(1)} \leq \omega(n) m_n^2 \nu_n^\uparrow \}$ occurs with probability tending to one, we apply the law of total probability together with the result of Step I and the definition of $r_n$. This shows us
	\begin{align*}
		\prob(\Xi_n^{(1)} \geq \omega(n) m_n^2 \nu_n^\uparrow ) &= \expec[\prob_n(\Xi_n^{(1)} \geq \omega(n) m_n^2 \nu_n^\uparrow)],\\ &\leq \expec[{2} r_n m_n^2 \nu_n^\uparrow] = \expec\left[\frac{{4}}{\omega(n)}\right] = \frac{{4}}{\omega(n)} \to 0.
	\end{align*}
\end{proof}

We end this section by proving Lemmas~\ref{lem:monotonicity preservation} through~\ref{lem:mass bound}. These were the lemmas needed to prove Corollary~\ref{cor:CCI to IRG} with Theorem~\ref{thm:direct implies undirect}.

\begin{proof}[\textbf{Proof of Lemma~\ref{lem:monotonicity preservation}}]
	Without loss of generality we will only show that $\mathcal{Q}_n$ being increasing implies $\mathcal{U}^{-1}(\mathcal{Q}_n)$ is increasing too. Consider a digraph $\vec{G}_n \in \mathcal{U}^{-1}(\mathcal{Q}_n)$ and suppose $\vec{G}_n \subseteq \vec{G}_n'$. Because $\vec{G}_n$ is a marked sub-graph of $\vec{G}_n'$ we know that $\vec{G}_n$ has arcs in locations {where} $\vec{G}_n'$ has arcs too. Therefore, we have $\mathcal{U}(\vec{G}_n) \subseteq \mathcal{U}(\vec{G}_n')$.
	
	Now, by definition $\vec{G}_n \in \mathcal{U}^{-1}(\mathcal{Q}_n)$ implies that $\mathcal{U}(\vec{G}_n) \in \mathcal{Q}_n$. Since $\mathcal{Q}_n$ is increasing, we have that $\mathcal{U}(\vec{G}_n) \subseteq \mathcal{U}(\vec{G}_n')$ implies $\mathcal{U}(\vec{G}_n') \in \mathcal{Q}_n$ too. However, by definition that entails $\vec{G}_n' \in \mathcal{U}^{-1}(\mathcal{Q}_n)$ too. Therefore, we indeed find that $\mathcal{U}^{-1}(\mathcal{Q}_n)$ is increasing.
\end{proof}

\begin{proof}[\textbf{Proof of Lemma~\ref{lem:kernel bound}}]
	Recall \eqref{eq:asymp connection number} and Assumption~\ref{ass:finite support}. Set $\lambda^\downarrow := \min_i \lambda_i$ and set $\varrho^\downarrow := \min_j \varrho_j$. {Note that $\lambda^\downarrow > 0$ and $\varrho^\downarrow > 0$, due to the fact that $\lambda_i, \varrho_j > 0$ for all $i$ and $j$ and that $|\mathcal{L}|$ and $|\mathcal{R}|$ has finite size.} Using these definitions, we can bound no matter the outcome of $V_n$:
	\[
	\kappa(t, s) =\mu \sum_{i \in \mathcal{L}} \sum_{j \in \mathcal{R}} \frac{p_{ij} \cdot I(t, i) J(s, j)}{\lambda_i \cdot \varrho_j} \leq \frac{\mu}{\lambda^\downarrow \varrho^\downarrow} \sum_{i \in \mathcal{L}} \sum_{j \in \mathcal{R}} p_{ij} \cdot I(t, i) J(s, j). 
	\]
	Note that the bound given in \eqref{eq:kernel bound undirect} converges to zero. So, every kernel $\kappa_n$ is bounded as well.
\end{proof}

\begin{proof}[\textbf{Proof of Lemma~\ref{lem:mass bound}}]
	Because of Assumption~\ref{ass:finite support} for each $i \in \mathcal{L}$ there will be at least one $k \in \mathcal{S}$ such that $I(k, i) = 1$. A similar thing is true for each $j \in \mathcal{R}$. Now, if an arc is generated in $\texttt{RaCInG}_{n, \mu}(T, C, I, J)$ then it will be generated proportional to $1 / (L_i R_j)$. In other words, the maximum of $\nu_n$ given $V_n$ will be dictated by the minimum of $L_i R_j$. By {noting} both $L_i$ and $R_j$ are at their core sets of vertices with specific types drawn from $T$. If we set $q^\downarrow := \min_{t} \prob(T = t)$, then we note that each individual $L_i$ and $R_j$ is stochastically bounded from below by a $\texttt{Bin}(n, q^\downarrow)$ random variable.
	
	We will now consider the event where one of the values $L_i$ or $R_j$ is ``too small'' and show this cannot happen. Specifically, we are interested in
	\[
	\mathcal{A}_n = \bigcap_{i \in \mathcal{L}} \{L_i \geq q^\downarrow n /2\} \cap \bigcap_{j \in \mathcal{R}} \{R_j \geq q^\downarrow n /2\}.
	\]
	We will show that $\mathcal{A}_n$ happens with probability tending to one on $V_n$, and subsequently that $\nu_n$ is sufficiently bounded when $\mathcal{A}_n$. First, applying a union bound on the compliment of $\mathcal{A}_n$, and applying the aforementioned stochastic domination, gives us
	\[
	\prob(\neg \mathcal{A}_n) \leq  {( |\mathcal{L}| + |\mathcal{R}|)} \prob(\texttt{Bin}(n, q^\downarrow) < q^\downarrow n /2) \leq {( |\mathcal{L}| + |\mathcal{R}|)} \prob(|\texttt{Bin}(n, q^\downarrow) - n q^\downarrow| > q^\downarrow n /2).
	\]
	Then, by applying the Chebyshev inequality {and recalling that $( |\mathcal{L}| + |\mathcal{R}|)$ is fixed and finite}, we find indeed that
	\[
	\prob(\neg \mathcal{A}_n) \leq \frac{{( |\mathcal{L}| + |\mathcal{R}|)}\text{Var}(\texttt{Bin}(n, q^\downarrow))}{n^2 (q^\downarrow)^2} = \frac{n q^\downarrow ( 1 - q^\downarrow){( |\mathcal{L}| + |\mathcal{R}|)}}{n^2 (q^\downarrow)^2} \to 0.
	\]
	Thus, $\mathcal{A}_n$ happens with probability tending to one. Finally, for every $V_n$ in $\mathcal{A}_n$ we have for all $i \in \mathcal{L}$ and $j \in \mathcal{R}$ that $L_i, R_j \geq q^\downarrow n / 2$. Therefore, for these vectors $V_n$ we may bound
	\begin{align*}
		\mu_n((v, T_v), (w, T_w)) &= \sum_{i\in\mathcal{L}} \sum_{j \in \mathcal{R}} \frac{\prob(C = (i, j)) \cdot I(T_v, i) J(T_w, j)}{L_i R_j},\\ &\leq  \sum_{i\in\mathcal{L}} \sum_{j \in \mathcal{R}} \frac{{4} \prob(C = (i, j)) \cdot I(T_v, i) J(T_w, j)}{n^2 (q^\downarrow)^2},\\
		&\leq \frac{{4}}{n^2 (q^\downarrow)^2}  \sum_{i\in\mathcal{L}} \sum_{j \in \mathcal{R}} \prob(C = (i, j)) = \frac{{4}}{n^2 (q^\downarrow)^2}.
	\end{align*}
\end{proof}

\subsection{Proofs of lemmas in Section~\ref{sec: main result spacing}}
We start by proving Lemma~\ref{lem:correct coupling}. When looking at Definition~\ref{def:coupling racing spacing} we note that $\texttt{SpaCInG}_{n, \mu} \subseteq \texttt{RaCInG}^\uparrow_{n^-, \mu}$ will always occur if $\texttt{SpaCInG}_{n, \mu}$ has more vertices at each spatial location than $\texttt{RaCInG}^\uparrow_{n^-, \mu}$ has vertices with the same spatial types (recall Figure~\ref{fig:coupl idea}). Therefore, to prove the probability in Lemma~\ref{lem:correct coupling} convergence, we need to ensure that $\texttt{RaCInG}_{n^-, \mu}$ always has less vertices with a vertex type $(x, \cdot)$ than $\texttt{SpaCInG}_{n, \mu}$ has vertices at location $x$. This will be proven in the next lemma.
\begin{lemma}[Number of vertices per spatial location]\label{lem:minmax vertices per spot}
	The following two statements are true {with probability tending to one}:
	\begin{enumerate}[label = (\arabic*)]
		\item In $\CCI_{n^+, \mu}$ each vertex type of the form $(\cdot, x)$ for fixed $x \in \mathcal{X}$ occurs at least $n/ |\mathcal{X}| + |\mathcal{X}|$ times.
		\item In $\CCI_{n^-, \mu}$ each vertex type of the form $(\cdot, x)$ for fixed $x \in \mathcal{X}$ occurs at most $n / |\mathcal{X}|$ times.
	\end{enumerate}
\end{lemma}
\begin{remark}
	The additional $|\mathcal{X}|$-term in statement (1) of Lemma~\ref{lem:minmax vertices per spot} is needed to compensate for the fact that the \emph{final} location $y$ in $\texttt{SpaCInG}$ (cf. Definition~\ref{def:SpaCInG algorithm}) can have at most $|\mathcal{X}|$ vertices more than the others.
\end{remark}
\begin{proof}
	We will only prove the first statement of the lemma, {because} the proof of the second is analogous. Throughout the proof, we will denote by $N_x^+$ the number of vertices in $\CCI_{n^+, \mu}$ with type $(\cdot, x)$ for fixed $x \in \mathcal{X}$. We start by noting that $N_x^+ \sim \texttt{Bin}(n^+, 1 / |\mathcal{X}|)$. Therefore, its expected value is given by $n^+ / |\mathcal{X}|$. Thus, by the Chernoff bound for binomial random variables (see e.g. Theorem 2.21 in \cite{vanderHofstad2016RandomNetworks1}) we have for any $\omega(n) \to \infty$ such that $\omega(n) < \sqrt{n}$ and some $c > 0$ that
	\begin{equation}\label{eq:spatial loc vertex type count spacing}
		\begin{split}
			\prob\left(\left|N_x^+ - \frac{n^+}{|\mathcal{X}|} \right| > \omega(n) \sqrt{n}  \right) &\leq 2 \exp\left( - \frac{\omega(n)^2 n}{2 ( n^+ / |\mathcal{X}| + \tfrac13 \omega(n) \sqrt{n})} \right), \\  &\leq \exp\left(-c\omega(n)^2\right).
		\end{split}
	\end{equation}
	
	To prove the (first statement of the) lemma, we need to show that the probability of the following event converges to zero:
	\[\bigcup_{x \in \mathcal{X}} \left\{ N_x^+ < \frac{n}{|\mathcal{X}|} + |\mathcal{X}|\right\}.\]
	To compute the probability of this event, we first apply the union bound to bound this probability as
	\begin{equation}\label{eq:union bound Nx}
		\prob\left( \bigcup_{x \in \mathcal{X}} \left\{ N_x^+ < \frac{n}{|\mathcal{X}|} + |\mathcal{X}|\right\} \right) \leq \sum_{x \in \mathcal{X}} \prob\left(N_x^+ < \frac{n}{|\mathcal{X}|} + |\mathcal{X}|\right).
	\end{equation}
	Now, we rewrite and bound the probability inside the sum as
	\begin{align*}
		\prob\left(N_x^+ < \frac{n}{|\mathcal{X}|} + |\mathcal{X}|\right) &= \prob\left(N_x^+ - \frac{n}{|\mathcal{X}|} -  \frac{\log(n) \sqrt{n}}{|\mathcal{X}|}< |\mathcal{X}| -  \frac{\log(n) \sqrt{n}}{|\mathcal{X}|}\right),\\ 
		&\leq \prob\left(\left| N_x^+ - \frac{n}{|\mathcal{X}|} -  \frac{\log(n) \sqrt{n}}{|\mathcal{X}|} \right| \geq   \frac{\log(n) \sqrt{n}}{|\mathcal{X}|} - |\mathcal{X}|\right), \\ 
		&\leq \prob\left(\left| N_x^+ - \frac{n^+}{|\mathcal{X}|}  \right| \geq  \frac{ \log(n) \sqrt{n}}{2 | \mathcal{X}|}\right).
	\end{align*}
	Next, we apply \eqref{eq:spatial loc vertex type count spacing} with $\omega(n) = \log(n)/(2 |\mathcal{X}|)$ to conclude there {exists} a constant $c'>0$ such that
	\[
	\prob\left(\left| N_x^+ - \frac{n^+}{|\mathcal{X}|}  \right| \geq  \frac{ \log(n) \sqrt{n}}{2 | \mathcal{X}|}\right) \leq \exp(- c' \log(n)^2).
	\]
	Substituting this back in \eqref{eq:union bound Nx} and noting that the upper bound is independent of $x$ yields the desired result, namely
	\[
	\prob\left( \bigcup_{x \in \mathcal{X}} \left\{ N_x^+ < \frac{n}{|\mathcal{X}|} + |\mathcal{X}|\right\} \right) \leq |\mathcal{X} |  \exp(- c' \log(n)^2) \to 0.
	\]
\end{proof}

To prove Lemma~\ref{lem:correct coupling} we also need to show that Definition~\ref{def:coupling racing spacing} is a valid coupling between RaCInG and SpaCInG. For this we will need a central lemma that provides the maximal coupling between two discrete random variables $X$ and $Y$ such that $\supp(X) \subseteq \supp(Y)$. This lemma will show that Step 3 in Definition~\ref{def:coupling racing spacing} is the best one can do to couple $\texttt{SpaCInG}_{n, \mu}$ and $\CCI_{n^-, \mu}$. We will first state and prove this central lemma, after which we will prove Lemma~\ref{lem:correct coupling}.
\begin{lemma}[Maximal coupling of overlapping random variables]\label{lem:max coupling overlap}
	Let $X$ and $Y$ be two discrete random variables such that $\supp(X) \subseteq \supp(Y)$. Moreover, assume for all $x \in \supp(X)$ that $\prob(X = x) \geq \prob(Y = x)$. Then, the following coupling is the maximal coupling between $X$ and $Y$:
	\begin{enumerate}[label = Step \Roman*.]
		\item Sample the value of $Y$. If $Y \in \supp(X)$, then set $X = Y$.
		\item If $Y \notin \supp(X)$, then sample the realisation of $X$ independently from $Y$ with its probability being proportional to $\prob(X = x) - \prob(Y = x)$.
	\end{enumerate}
\end{lemma}
\begin{proof}
	Set $p_x := \prob(X = x)$ and $q_x := \prob(Y = x)$. The maximal coupling theorem (Theorem 2.9 in \cite{vanderHofstad2016RandomNetworks1}) stipulates that the best coupling between $X$ and $Y$ satisfies
	\begin{align*}
		\prob(X \neq Y) &= \tfrac12 \sum_{x \in \supp(Y)} |p_x - q_x|,\\
		&= \tfrac12 \sum_{x \in \supp(X)}  |p_x - q_x| + \tfrac12 \sum_{x \notin  \supp(X)} |p_x - q_x|,\\
		&= \tfrac12 \sum_{x \in \supp(X)} (p_x - q_x) + \tfrac12 \sum_{x \notin  \supp(X)} q_x,\\ 
	\end{align*}
	Here, the final equality used that $p_x > q_x$ for $x \in \supp(X)$ and that $p_x = 0$ when $x \notin \supp(X)$. By noting that $\sum_{x \in \supp(X)} p_x = 1$, we can calculate $\prob(X \neq Y)$ under the maximal coupling to be
	\begin{align*}
		\prob(X \neq Y) &= \tfrac12 \left(1 - \sum_{x \in \supp(X)}q_x \right) + \tfrac12 \sum_{x \notin  \supp(X)} q_x,\\
		&= \tfrac12 \left(1 - \prob(Y \in \supp(X)) + \prob(Y \notin \supp(X))\right) = \prob(Y \notin \supp(X)).
	\end{align*}
	Note that this probability is equal to the probability that Step II is executed in our coupling, since when $X$ is sampled independently from $Y$, it is certain that $Y \notin \supp(X)$. Thus, the algorithm to sample $X$ and $Y$ from the theorem has the correct value of $\prob(X \neq Y)$ to be the maximum coupling. What remains to be shown, is that the described algorithm to sample $X$ and $Y$ is indeed a coupling.
	
	First, note that the marginal of $Y$ indeed has the correct distribution, since the value of $Y$ will always be sampled according to its distribution. To find the marginal of $X$ under the coupling, we condition on the value of $Y$. The law of total probability gives us
	\begin{align*}
		\prob(X = x) =  &\sum_{y \in \supp(X)} \prob(X = x \;|\; Y = y)\prob(Y = y)\\ &+ \sum_{y \notin \supp(X)} \prob(X = x \;|\; Y = y)\prob(Y = y).
	\end{align*}
	By noting that $y \in \supp(X)$ means that $y = x$ with probability $1$, due to Step I of the algorithm, we find
	\begin{align*}
		\prob(X = x) = \prob(X = x \;|\; Y = x)\prob(Y = x) + \sum_{y \notin \supp(X)} \prob(X = x \;|\; Y = y)\prob(Y = y).
	\end{align*}
	Also, note that for the same reason $\prob(X = x \;|\; Y = x) = 1$. Thus, we can further derive
	\begin{equation} \label{eq:law total prob coupling}
		\prob(X = x) = \prob(Y = x) + \sum_{y \notin \supp(X)} \prob(X = x \;|\; Y = y)\prob(Y = y).
	\end{equation}
	Recall that $\prob(X = x \;|\; Y = y)$ is independent of $y$ and proportional to $p_x - q_x$ when $y \notin \supp(X)$. Thus, we find that
	\begin{align*}
		\prob(X = x \;|\; Y = y) &= \frac{p_x - q_x}{\sum_{x \in \supp(X)}(p_x - q_x)},\\ &= \frac{p_x - q_x}{1 - \prob(Y \in \supp(X))} = \frac{p_x - q_x}{\prob(Y \notin \supp(X))}.
	\end{align*}
	Substituting everything into \eqref{eq:law total prob coupling} finally shows that
	\begin{align*}
		\prob(X = x ) &= q_x + \frac{p_x - q_x}{\prob(Y \notin \supp(X))} \sum_{y \notin \supp(X)} q_y,\\ &= q_x + \frac{(p_x - q_x) \prob(Y \notin \supp(X))}{\prob(Y \notin \supp(X))}  = p_x.
	\end{align*}
	Therefore, the described procedure is a correct coupling, and by the first part of the proof even the maximal coupling.
\end{proof}

\begin{proof}[\textbf{Proof of Lemma~\ref{lem:correct coupling}}]
	We need to show that (1) the algorithm of Definition~\ref{def:coupling racing spacing} is a valid coupling, and (2) that this coupling outputs $\texttt{SpaCInG}_{n, \mu} \subseteq \texttt{RaCInG}_{n^-, \mu}^\uparrow$ with probability tending to one. We will tackle these parts separately.
	
	\paragraph{Valid coupling.} We start by noticing that we automatically have a valid coupling if not all spatial types in $\texttt{RaCInG}_{n^-, \mu}$ can be linked to a vertex in  $\texttt{SpaCInG}_{n, \mu}$ (see Step 1 in Definition~\ref{def:coupling racing spacing}). This is, because in this situation both models are generated independently, meaning they are coupled under the independent coupling. Hence, we only need to show the coupling is valid in case all vertices in $\texttt{RaCInG}_{n^-, \mu}$ can be linked to a vertex in $\texttt{SpaCInG}_{n, \mu}$.
	
	If all vertices can be linked, we note that the vertex generation procedure (Step 1 and 2 in Definition~\ref{def:coupling racing spacing}) will always ensure that the marginal vertex generation procedures are respected. This is, because in $\texttt{RaCInG}_{n^-, \mu}$ each vertex is first assigned an independent spatial type $x \in \mathcal{X}$ with probability $1 / |\mathcal{X}|$ after which both $\texttt{RaCInG}_{n^-, \mu}$ and $\texttt{SpaCInG}_{n, \mu}$ generate their non-spatial vertex type $t \in \mathcal{S}$ independently given $x$ with probability $q_t^x$.
	
	Next, when placing the arcs in both models (Step 3 of Definition~\ref{def:coupling racing spacing}), we note that each arc can be placed between the same subset of vertex types from $\mathcal{S} \times \mathcal{X}$, because the generated arc type and the arc placement rules will always be the same in both models. However, when vertices are linked in Step 1 and 2 of Definition~\ref{def:coupling racing spacing}, this means that the number of available vertices in $\texttt{SpaCInG}_{n, \mu}$ to place an arc in between is always {more} than in $\texttt{RaCInG}_{n^-, \mu}$. {Moreover, since arcs are placed uniformly at random, the probability of an arc being placed at each location in $\texttt{SpaCInG}_{n, \mu}$ is always lower than for locations in $\texttt{RaCInG}_{n^-, \mu}$. Hence, through Lemma~\ref{lem:max coupling overlap} we may conclude that Step 3 of Definition~\ref{def:coupling racing spacing} coincides with the maximal coupling algorithm for arc placement in both $\texttt{SpaCInG}_{n, \mu}$ and $\texttt{RaCInG}_{n^-, \mu}$.} This is, because a uniform choice from a set conditioned on a subset is still a uniform choice from that subset.
	
	Finally, Step 4 and 5 simply add vertices and arcs to $\texttt{RaCInG}_{n^-, \mu}$ to ensure that both models have the same vertices and the same arcs incident to them. This matches the description in Definition~\ref{def: event insensitivity} and hence we may conclude that the coupling is valid.
	
	\paragraph{Sub-graph is often outputted.} Now we know Definition~\ref{def:coupling racing spacing} is a valid coupling, we investigate when the event
	\begin{equation}\label{eq:subset happens}
		\left\{ \texttt{SpaCInG}_{n, \mu} \subseteq \CCI_{n^-, \mu}^\uparrow \right\},
	\end{equation}
	occurs under the coupling. For this we recall in Step 1 of the coupling that two things might occur. It can happen (1) for all spatial locations that $\texttt{SpaCInG}_{n, \mu}$ has more vertices at said location than $\texttt{RaCInG}_{n^-, \mu}$, or (2) there exists at least one spatial location $x \in \mathcal{X}$ for which $\texttt{RaCInG}_{n^-, \mu}$ has more vertices at that location than $\texttt{SpaCInG}_{n, \mu}$.
	
	If the first option occurs, then all vertices in $\texttt{RaCInG}_{n^-, \mu}$ can be linked to a vertex in $\texttt{SpaCInG}_{n, \mu}$ and Step 2--5 of the coupling can be executed. If these steps are executed, then note from their description that \eqref{eq:subset happens} automatically happens. However, if the second option occurs, $\texttt{RaCInG}_{n^-, \mu}$ and $\texttt{SpaCInG}_{n, \mu}$ are generated independently, meaning that \eqref{eq:subset happens} might not occur. Therefore, we can lower bound the probability of \eqref{eq:subset happens} by the probability that $\texttt{RaCInG}_{n^-, \mu}$ has less vertices at all locations than $\texttt{SpaCInG}_{n, \mu}$. Hence, if we set $N_x$ to denote the number of vertices with type $(\cdot, x) \in \mathcal{S} \times \mathcal{X}$, then we can investigate the event
	\[
	\mathcal{T}_n := \bigcap_{x \in {\mathcal{X}}} \{N_x \leq n / |\mathcal{X}|\}.
	\]
	
	By Lemma~\ref{lem:minmax vertices per spot} we note that $\mathcal{T}_n$ occurs with probability tending to one. Therefore, we may conclude
	\begin{align*}
		\prob(\texttt{SpaCInG}_{n, \mu} \subseteq \CCI_{n^-, \mu}^\uparrow) &\geq \prob(\texttt{SpaCInG}_{n, \mu} \subseteq \CCI_{n^-, \mu}^\uparrow \;|\; \mathcal{T}_n) \prob(\mathcal{T}_n),\\
		&= 1 \cdot \left( 1 - o(1) \right) \to 1.
	\end{align*}
\end{proof}

We end this section by proving Lemma~\ref{lem:extra arcs}.

\begin{proof}[\textbf{Proof of Lemma~\ref{lem:extra arcs}}]
	From Lemma~\ref{lem:correct coupling} we know that there exists a coupling with probability tending to one such that all vertices in $\CCI_{n^-, \mu}$ have a linked vertex in $\texttt{SpaCInG}_{n, \mu}$ with the same type. To monitor how different arc placement probabilities in the two models are (cf. Step 3 in Definition~\ref{def:coupling racing spacing}), we need to keep track of the number of vertices that are not overlapping between $\texttt{SpaCInG}_{n, \mu}$ and $\CCI_{n^-, \mu}$. To this end, we define $N_x$ to be the number of vertices in $\CCI_{n^-, \mu}$ that have a type of the form $(\cdot, x)$ with $x \in \mathcal{X}$ fixed, and set the random variable
	\[
	N_x^\lr := \max_{0 \leq i \leq |\mathcal{X}|}\left\{ \left|\frac{n}{|\mathcal{X}|} + i - N_x \right|  \right\}.
	\]
	Note the maximum is required due to the \emph{final} location in $\texttt{SpaCInG}_{n, \mu}$ where up to $|\mathcal{X}|$ {more} vertices might be present. The proof will now consist of three steps:
	\begin{enumerate}[label = \Roman*.]
		\item We will bound the value of $N^\lr_x$ for all spatial locations $x \in \mathcal{X}$ simultaneously.
		\item We will use Step I to find an upper-bound on the probability that a fixed arc is placed between two vertices in $\texttt{SpaCInG}_{n, \mu}$ with given locations and types of which at least one does not exist in $\CCI_{n^-, \mu}$.
		\item We will use the probability in Step II to bound $A_n$.
	\end{enumerate}
	
	\paragraph{Step I.} We can conclude using the union bound that
	\begin{equation}\label{eq:sum Nlr bound}
		\begin{aligned}
			& \prob\left(  \bigcup_{x \in \mathcal{X}} \left\{N^{\lr}_x > \frac{2 \log(n) \sqrt{n}}{|\mathcal{X}|}   \right\} \right)
			\leq \sum_{x \in  \mathcal{X}} \prob\left( \max_{0 \leq i \leq |\mathcal{X}|}\left|N_x -  \frac{n}{|\mathcal{X}|} -  i \right| > \frac{2 \log(n) \sqrt{n}}{|\mathcal{X}|} \right).
		\end{aligned}
	\end{equation}
	Since $|\mathcal{X}| < \infty$, it suffices to prove that the probability inside the sum is $o(1)$. Next, we bound the probabilities inside the sum by splitting them up in their positive and negative part. By taking $i = 0$ for the positive part and $i = |\mathcal{X}|$ for the negative part, the probability inside the sum can be upper-bounded by
	\begin{equation} \label{eq:max probabilities spacing bound}
		\prob\left(N_x - \frac{n}{|\mathcal{X}|} > \frac{2 \log(n) \sqrt{n}}{|\mathcal{X}|} \right) + \prob\left(\frac{n}{|\mathcal{X}|} + |\mathcal{X}| - N_x > \frac{2 \log(n) \sqrt{n}}{|\mathcal{X}|}\right).
	\end{equation}
	Rewriting the previous upper-bound shows that it is equal to
	\[
	\prob\left(N_x  > \frac{n +2 \log(n) \sqrt{n}}{|\mathcal{X}|} \right) + \prob\left(\frac{n - \log(n)\sqrt{n}}{|\mathcal{X}|} - N_x > \frac{ \log(n) \sqrt{n}}{|\mathcal{X}|} - |\mathcal{X}|\right).
	\]
	From Lemma~\ref{lem:minmax vertices per spot} we can conclude that the first of the two probabilities converges to zero. For the second one, we first bound it as
	\[
	\prob\left(\frac{n - \log(n)\sqrt{n}}{|\mathcal{X}|} - N_x > \frac{ \log(n) \sqrt{n}}{|\mathcal{X}|} - |\mathcal{X}|\right) \leq \prob\left(\left|N_x - \frac{n^-}{|\mathcal{X}|} \right|> \frac{ \log(n) \sqrt{n}}{2|\mathcal{X}|} \right).
	\]
	From the Chernoff bound for binomial random variables we may conclude that this upper-bound converges to zero. Recall also \eqref{eq:spatial loc vertex type count spacing}.    Thus, we may conclude for all $x \in \mathcal{X}$ simultaneously that there exists an event -- which we will denote by $\mathcal{V}_n^\lr$ -- on which $N_x^\lr \leq 2 \log(n) \sqrt{n} / |\mathcal{X}|$ with probability tending to one. 
	
	\paragraph{Step II.} We will use the result of Step I to bound $A_n$. For this, we first calculate the probability that a fixed arc $a$ connects to at least one of the vertices that appears in $\texttt{SpaCInG}_{n, \mu}$, but not in $\CCI_{n^-, \mu}$. We denote by $\mathcal{E}_{ts}^{xy}$ the event that the out-part of arc $a$ connects to one of the extra vertices at location $x$ with type $t$ in $\texttt{SpaCInG}_{n, \mu}$ or that the in-part of $a$ connects to one of the extra vertices at location $y$ with type $s$. If we set $C_a$ to be the colour of the arc, then the probability we want to compute is ({see} Step 3 in Definition~\ref{def:SpaCInG algorithm} and Notation~\ref{not:spacing mathematical notation})
	\begin{equation}\label{eq: arc excess prob fixed types}
		\prob(\mathcal{E}_{ts}^{xy}) = \sum_{i \in \mathcal{C}} \sum_{j \in \mathcal{C}} \prob\left(  \mathcal{E}_{ts}^{xy} \;|\; C_a = (x, i, y, j) \right) p_{ij}^{xy}.
	\end{equation}
	{We next note that arcs will only be able to attach to any of the extra vertices if at least one vertex exists that can accept the arc. Otherwise, the arc gets thrown out (see Step 4 in Definition~\ref{def:SpaCInG algorithm}). Thus, \eqref{eq: arc excess prob fixed types} does not change if we only sum over elements from $\mathcal{C}$ that can be accepted by some vertex type only. Recalling the notation in \eqref{eq:active LR pairs}, we can rewrite \eqref{eq: arc excess prob fixed types} as}
	\begin{equation}\label{eq: arc excess prob fixed types active}
		{	\prob(\mathcal{E}_{ts}^{xy}) = \sum_{(i, j) \in \mathcal{C}_{xy}} \prob\left(  \mathcal{E}_{ts}^{xy} \;|\; C_a = (x, i, y, j) \right) p_{ij}^{xy}.}
	\end{equation}
	
	Denote by $N_x^t$ the number of vertices at location $x \in \mathcal{X}$ with type $t \in \mathcal{S}$ in $\CCI_{n^-, \mu}$. Given the realisation of the first step of Definition~\ref{def:coupling racing spacing}, the probability of the event $\mathcal{E}_{ts}^{xy}$ conditioned on $C_a$ is upper-bounded by
	\begin{equation}\label{eq:missmatch upper bound probability}
		\frac{N_x^\lr I(t, i)}{\sum_{k \in \mathcal{S}}N^k_x I(k, i)} + \frac{N_y^\lr J(s, j)}{\sum_{k \in \mathcal{S}}N^k_y J(k, j)}.
	\end{equation}
	We now seek to control all the random variables that appear in this expression. Note that $N_x^\lr$ and $N_y^\lr$ are controlled by the event $\mathcal{V}_n^\lr$ from the conclusion of Step I. We can control all random variables $N_x^k$ with a reasoning analogous to \eqref{eq:spatial loc vertex type count spacing}. This would show that
	\[
	\prob\left(\left|N_x^k - \frac{n q_x^k}{|\mathcal{X}|} \right| > \log(n) \sqrt{n}  \right) \to 0.
	\]
	If we define
	\[
	\mathcal{V}_n^{(x,k)} := \left\{ \left|N_x^k - \frac{n q_x^k}{|\mathcal{X}|} \right| \leq \log(n) \sqrt{n} \right\},
	\]
	then we know that
	\[
	\mathcal{W}_n := \left( \bigcap_{x \in \mathcal{X}} \bigcap_{k \in \mathcal{S}} \mathcal{V}_n^{(x, k)} \right) \cap \mathcal{V}_n^\lr,
	\]
	occurs with probability tending to one, because $\mathcal{X}$ and $\mathcal{S}$ are finite sets (cf. Assumption~\ref{ass:SpaCInG}). Conditioned on $\mathcal{W}_n$ we can upper bound \eqref{eq:missmatch upper bound probability} by
	\begin{equation}\label{eq:upper excess prob}
		\frac{2 \log(n) |\mathcal{X}|^{-1} \sqrt{n} \cdot I(t, i)}{\sum_{k \in \mathcal{S}}(q_k^x |\mathcal{X}|^{-1} n - \log(n)  \sqrt{n}) I(k,i)} + \frac{2 \log(n) |\mathcal{X}|^{-1} \sqrt{n} \cdot J(s, j)}{\sum_{k \in \mathcal{S}}(q_k^y |\mathcal{X}|^{-1} n - \log(n) \sqrt{n}) J(k,j)}.
	\end{equation}
	We will now use a Taylor series approximation to bound these leftover fractions. We will only explicitly bound the first. The bound for the second is similar. In this calculation we first set
	\[
	{\lambda^\downarrow := \min\{\lambda_i^x : (i, x) \in \mathcal{C}^+ \wedge \lambda_i^x > 0 \}, \text{ and } \varrho^\downarrow := \min\{ \varrho_j^y : (j, y) \in \mathcal{C}^+ \wedge \varrho_j^y > 0 \}, }
	\]
	where $\lambda_i^x$ and $\varrho^y_j$ are defined in \eqref{subeq:newlamba} and \eqref{subeq:newrho}. {In other words, $\lambda^\downarrow$ and $\varrho^\downarrow$ indicate the smallest fraction of vertices that an arc with admissible colours can choose between.} Then, {because in \eqref{eq: arc excess prob fixed types active} we only sum over arc colours that can connect to at least one vertex type}, we can bound {\eqref{eq:upper excess prob} as}
	\begin{align*}
		\frac{2 \log(n) |\mathcal{X}|^{-1} \sqrt{n} \cdot I(t, i)}{\sum_{k \in \mathcal{S}}(q_k^x |\mathcal{X}|^{-1} n - \log(n) \sqrt{n}) I(k,i)} &\leq \frac{2 \log(n) {I(t, i)}  / (\lambda_i^x \sqrt{n}) }{1 - |\mathcal{X}| |\mathcal{S} | \log(n) / (\lambda^\downarrow \sqrt{n})},\\
		&\leq \frac{2 \log(n)}{\lambda^\downarrow \sqrt{n}} \left( 1 + \mathcal{O}\left( \frac{|\mathcal{X}| |\mathcal{S} | \log(n)}{\lambda^\downarrow \sqrt{n}} \right) \right).
	\end{align*}
	Thus, if we substitute this into \eqref{eq:upper excess prob}, then we find that it is bounded by
	\[
	\frac{4 \log(n) (1 + o(1))}{(\lambda^\downarrow \wedge \varrho^\downarrow)\sqrt{n}}.
	\]
	When we substitute this back into \eqref{eq: arc excess prob fixed types active}, noting that the arc colour sampling procedure is independent from the vertex sampling procedure (cf. Definition~\ref{def:SpaCInG algorithm}), then we find our desired bound for this step:
	\begin{align*}
		\prob(\mathcal{E}_{ts}^{xy}) &=  \prob(\mathcal{E}_{ts}^{xy} \;|\; \mathcal{W}_n) + o(1),\\
		&= {\sum_{(i, j) \in \mathcal{C}_{xy}}} \prob\left(  \mathcal{E}_{ts}^{xy} \;|\; \mathcal{W}_n \cap C_a = (x, i, y, j) \right) p_{ij}^{xy},\\
		&\leq {\sum_{(i, j) \in \mathcal{C}_{xy}}} \frac{4  \log(n) (1 + o(1))}{(\lambda^\downarrow \wedge \varrho^\downarrow)\sqrt{n}} \cdot p_{ij}^{xy} \leq  \frac{5 \log(n)}{(\lambda^\downarrow \wedge \varrho^\downarrow)\sqrt{n}}
	\end{align*}
	
	\paragraph{Step III.} In Step II we bounded the probability that a fixed arc is generated in $\texttt{SpaCInG}_{n, \mu}$ from a vertex with type $(t, x) \in \mathcal{S}^+$ to a vertex with type $(s, y) \in \mathcal{S}^+$ of which at least one is not linked to any of the vertices in $\CCI_{n^-, \mu}$. From this, we may define the event $\mathcal{E}$ that a fixed arc $a$ is placed incident to at least one vertex in $\texttt{SpaCInG}_{n, \mu}$ that is not linked to another vertex in $\CCI_{n^-, \mu}$ (regardless of the types of the vertices). Using the previous step, we can write
	\[
	\mathcal{E} = \bigcup_{x \in \mathcal{X}} \bigcup_{y \in \mathcal{X}} \bigcup_{t \in \mathcal{S}} \bigcup_{s \in \mathcal{S}} \mathcal{E}_{ts}^{xy},
	\]
	and hence bound, using Step II, the union bound and Assumption~\ref{ass:SpaCInG} that
	\begin{equation}\label{eq:probbound E}
		\prob(\mathcal{E}) \leq \frac{5|\mathcal{X}|^2 |\mathcal{S}|^2 \log(n)}{(\lambda^\downarrow \wedge \varrho^\downarrow) \sqrt{n}}.
	\end{equation}
	Now we know this probability, we also know that $A_n \sim \texttt{Bin}(\lfloor \mu n \rfloor, \prob(\mathcal{E}))$, since all arcs are generated and placed {independently} in $\texttt{SpaCInG}_{n, \mu}$, and since self-loops and multi-arcs are allowed.  Using \eqref{eq:probbound E} we can stochastically dominate $A_n \preceq \texttt{Bin}(\mu n, 5 |\mathcal{X}|^2 |\mathcal{S}|^2 \log(n) / ( \lambda^\downarrow \wedge \varrho^\downarrow \cdot \sqrt{n}))$. We finally use the Chernoff bound for binomial random variables on $A_n$ to find the desired result. For some $c > 0$ we have
	\begin{align*}
		\prob(A_n \geq c \log(n) \sqrt{n} + \log(n)^2 \sqrt[4]{n}) &\leq \exp\left( - \frac{\log(n)^4 \sqrt{n}}{2 (c \log(n) \sqrt{n} + \log(n)^2 \sqrt[4]{n})} \right), \\ &\leq \exp\left(- \log(n)^2 \right) \to 0.
	\end{align*}
	The statement of the lemma follows, since $c \log(n) \sqrt{n} + \log(n)^2 \sqrt[4]{n} \leq c' \log(n) \sqrt{n}$ for some new $c' > 0$.
\end{proof}

\subsection{Proofs of Propositions}\label{sec:proofs_propositions}
In this section we will first use Corollary~\ref{cor:CCI to IRG} to prove Proposition~\ref{prop:direct racing}. Heuristically, we will prove the proposition by first using the second moment method to show the desired convergence in the inhomogeneous random graph model. Then, we show that Corollary~\ref{cor:CCI to IRG} can be applied, due to the fact that adding some edges will not change the limiting probability.

\begin{proof}[\textbf{Proof of Proposition~\ref{prop:direct racing}}]
	Set $\alpha := (\kappa(t, s) + \kappa(s, t))q_tq_s$. For all $\varepsilon > 0$ we consider the event $\mathcal{Q}_n^\pm(\varepsilon) := \{E_n^{ts} \geq (\alpha \pm \varepsilon)n\}$. Note that this event is increasing (cf. Definition~\ref{def:monotone events}). To prove the proposition we will do two things:
	\begin{enumerate}[label = \textbf{\Roman*.}]
		\item We show that $E_n^{ts} / n \to \alpha$ in probability for $\IRG_n(T, \kappa_n^\circ)$ for all $\kappa_n$ satisfying \eqref{eq:kernel bound undirect}.
		\item We show that $\mathcal{Q}_n^\pm$ is insensitive for decrease at some rate $\omega(n) \to \infty$ in $\texttt{IRG}_n(T, \kappa_n^\circ)$ for all $\kappa_n$ satisfying \eqref{eq:kernel bound undirect}, and that $\mathcal{U}^{-1}(\mathcal{Q}_n)$ is insensitive for increase in $\texttt{RaCInG}_{n, \mu}(T, C, I, J)$ at the same rate $\omega(n)$.
	\end{enumerate}
	Together, I and II will imply through Corollary~\ref{cor:CCI to IRG} that both $\prob(\ESRG_n(V_n, \nu_n^\circ, \lfloor \mu n \rfloor)  \in \mathcal{Q}_n^+) \to 0$ and that $\prob(\ESRG_n(V_n, \nu_n^\circ, \lfloor \mu n \rfloor)  \in \mathcal{Q}_n^-) \to 1$, implying that $E_n^{ts} / n \to \alpha$ in {probability} for $\ESRG_n(V_n, \nu_n^\circ, \lfloor \mu n \rfloor)$, because $\alpha$ is a constant. We now prove the two leftover steps.
	
	\paragraph{Step I -- first moment.} Fix a kernel $\kappa_n^\circ$ satisfying \eqref{eq:kernel bound undirect} and set $I_{vw}$ to be the indicator that edge $\{v, w\}$ is between a vertex with type $t$ and a vertex with type $s$. We can write
	\[
	E_n^{ts} = \sum_{v = 1}^n \sum_{w < v} I_{vw}.
	\]
	Using homogeneity of the model, we can calculate the expected value to be
	\begin{equation}\label{eq:first moment undirected edges}
		\expec[E_n^{ts}] = \frac{n (n - 1)}{2} \prob(I_{12} = 1).
	\end{equation}
	To compute the leftover probability we condition on the vertex types $T_1$ and $T_2$. Note that we can create a valid edge if $\{T_1 = t\}\cap\{T_2 = s\}$ occurs or when $\{T_1 = s\}\cap\{T_2 = t\}$ occurs. Therefore, using the definition of the IRG model, we get
	\begin{align*}
		\prob(I_{12} = 1) 
		&= \prob(I_{12} = 1 \mid T_1 = t, T_2 = s)q_tq_s + \prob(I_{12} = 1 \mid T_1 = s, T_2 = t)q_sq_t,\\
		&= \frac{\kappa_n^\circ(t, s) q_t q_s + \kappa_n^\circ(s, t) q_s q_t}{n} = \frac{2 \kappa_n^\circ(t, s) q_tq_s}{n}.
	\end{align*}
	Now, we use \eqref{eq:kernel bound undirect} to conclude
	\[
	\prob(I_{12} = 1) = \frac{2(\kappa(t, s) + \kappa(s, t))q_t q_s}{n} + \mathcal{O}(n^{-11/10}).
	\]
	Substituting this back into \eqref{eq:first moment undirected edges} shows
	\[
	\expec[E_n^{ts}/n] = \frac{n-1}{2} \cdot \left( \frac{2(\kappa(t, s) + \kappa(s, t))q_t q_s}{n} + \mathcal{O}(n^{-11/10}) \right) \to \alpha.
	\]
	
	\paragraph{Step I -- second moment.} Using the previously defined sum of indicators, we can write the second moment as
	\begin{equation}\label{eq:second moment edges undirected}
		\expec[(E_n^{ts}/n)^2] = \frac1{n^2} \sum_{v_1 =1}^n \sum_{w_1 < v_1} \sum_{v_2 = 1}^n \sum_{w_2 < v_2} \prob(I_{v_1w_1} = 1, I_{v_2 w_2} = 1).
	\end{equation}
	We note in this expression that the values of $v_i$ and $w_i$ cannot be the same. However, the values may be the same in other cases. Thus, we will determine the second moment by considering these separate cases where certain indices in \eqref{eq:second moment edges undirected} overlap. In every case, we will identify the number of terms in the sum that have this overlap, and (an upper bound on) the probabilities inside the sum under the given overlap. 
	
	First, we consider the case where all four indices are distinct. Then, due to independence, we find using the first moment calculation that the contribution of this case to \eqref{eq:second moment edges undirected} equals
	\[
	\frac1{n^2} \cdot \underbrace{n(n-1)(n-2)(n-3)}_{\text{Number of terms without index overlap}} \cdot \underbrace{\prob(I_{12} = 1) \prob(I_{34} = 1)}_{\text{Probability without index overlap}} \to \alpha^2.
	\] 
	
	Next, when one index overlaps the sum in \eqref{eq:second moment edges undirected} will have $\Theta(n^3)$ terms. At the same time, the two indicators in the probability will still be conditionally independent given the vertex types. Because there are only a finite {number} of vertex types (see Assumption~\ref{ass:finite support}), this implies the probability in \eqref{eq:second moment edges undirected} will have a contribution of $\mathcal{O}(n^{-2})$. This means that the total contribution of the case where one index overlaps to \eqref{eq:second moment edges undirected} equals
	\[
	\frac1{n^2} \cdot \Theta(n^3) \cdot \mathcal{O}(n^{-2}) \to 0.
	\]
	
	Similarly, when two indices overlap there are $\Theta(n^2)$ terms in the sum for which this is the case. At the same time, the contribution of the probability is $\mathcal{O}(n^{-1})$ for each term, because only one edge is involved (the two indicators are the same). Therefore, the contribution of this case to \eqref{eq:second moment edges undirected} equals
	\[
	\frac1{n^2} \cdot \Theta(n^2) \cdot \mathcal{O}(n^{-1}) \to 0.
	\]
	
	Taken all together, we find that $\expec[(E_n^{ts}/n)^2] \to \alpha^2$, meaning that $\text{Var}(E_n^{ts}/n) \to 0$. From this we may conclude that indeed $E_n^{ts}/n \to \alpha$ in probability for $\IRG_n(T, \kappa_n^\circ)$.
	
	\paragraph{Step II.} First to show insensitivity for decrease in $\IRG_n(T, \kappa^\circ_n)$ we recall from Step I that $E_n^{ts}/n \to \alpha$ in probability. Therefore, the probability of $\{E_n^{ts} \geq (\alpha \pm \varepsilon - \varepsilon')n\}$ will be asymptotically the same as the probability of $\mathcal{Q}_n^\pm(\varepsilon)$ for some $\varepsilon' > 0$. This means that removing $\varepsilon' n$ edges from $\IRG_n(T, \kappa^\circ_n)$ does not influence the asymptotic probability. Therefore, if we pick the rate $\omega(n)$ such that $\omega(n)/n \to 0$, then indeed we find that $\mathcal{Q}_n^\pm(\varepsilon)$ has the required insensitivity for all $\varepsilon$.
	
	To show edge insensitivity {for} RaCInG we first denote by $\vec{E}_{n}^{ts}$ the number of arcs from a vertex with type $t$ to one with type $s$, and by $\tilde{E}_{n}^{ts}$ the number of times there is an arc in both directions between a vertex with type $t$ and $s$. Note that $\mathcal{U}^{-1}(\mathcal{Q}_n^\pm(\varepsilon)) := \{\vec{E}_{n}^{ts} + \vec{E}_{n}^{st} - \tilde{E}_{n}^{ts} \geq (\alpha \pm \varepsilon)n\}$. If we add $\omega(n)$ edges to this model, then note that the count on the left hand side of $\mathcal{U}^{-1}(\mathcal{Q}_n^\pm(\varepsilon))$ increases by at most $\omega(n)$ too (disregarding the instances where bidirectional arcs are created and the count does not increase). Therefore, if $\omega(n)$ is chosen such that $\omega(n)/n \to 0$, we would obtain the required insensitivity for increase.
\end{proof}
\begin{remark}
	Note that the strategy applied in the proof can be applied to different features too (like wedges). In Step II, the idea would always be to (stochastically) upper- or lower-bound the feature of interest by a worst-case scenario. For edges, the above proof for example works because the number of edges/arcs will only increase/decrease by $\omega(n)$ if we add/remove $\omega(n)$ edges/arcs. Insensitivity follows because $\omega(n) \ll n$. For other features, like wedges, this worst-case increase/decrease will be more complicated, but can often still be analysed. For example, in the case of wedges, adding one edge will increase the number of wedges at most proporiontally to the maximal degree in the graph. If one can then show that this maximal degree remains much smaller than $n$, we can prove similar results to Proposition~\ref{prop:direct racing} for wedges.
\end{remark}

To prove Proposition~\ref{prop:direct spacing} we will use a very similar strategy to the previous proof.

\begin{proof}[\textbf{Proof of Proposition~\ref{prop:direct spacing}}]
	Define $\gamma_{ts}^{xy} := \kappa(t, x, s, y) q_t^x q_s^y / |\mathcal{X}|^2$. Next, define the parametrized set of monotone events $\mathcal{Q}_n(\varepsilon) := \{A_{ts}^{xy} > (\gamma_{ts}^{xy} + \varepsilon)n\}$ for fixed $\varepsilon$. We will start with the following steps for fixed $\varepsilon$.
	\begin{enumerate}[label = \textbf{\Roman*.}]
		\item We show \eqref{eq:convergence requirement SIRD} for $\mathcal{Q}_n(\varepsilon)$.
		\item We show that $\mathcal{Q}_n(\varepsilon)$ is completely insensitive in $\texttt{RaCInG}_{n, \mu}(\hat{T}, \hat{C}, \hat{I}, \hat{J})$ at rate $c \log(n) \sqrt{n}$ for any $c$.
	\end{enumerate}
	After these two steps we can conclude convergence of $A_{ts}^{xy} / n \to \gamma_{ts}^{xy}$ in probability for $\texttt{SpaCInG}_{n, \mu}$.
	
	\paragraph{Step I.} We investigate convergence $\mathcal{Q}_n$ in $\IRD_{n^-}((X, T), \kappa_{n^-})$. We will not treat the computation for $\IRD_{n^+}((X,T), \kappa_{n^+})$, since its proof is analogous. Our strategy will be to fix a kernel $\kappa_n'$ that adheres to {\eqref{eq:kernel range SpaCInG}}, and we will show that under this kernel $A_{ts}^{xy}/n \to \gamma_{ts}^{xy}$ in probability for $\IRD_{n^-}((X,T), \kappa_{n^-})$. In particular, this will imply that probabilities involving $\mathcal{Q}_n(\varepsilon)$ converge to the same values for all kernels $\kappa_n'$ for each fixed $\varepsilon$. 
	
	We will now fix a kernel $\kappa_n'$ that adheres to {\eqref{eq:kernel range SpaCInG}}. This implies that we can write it for fixed $t, s \in \mathcal{S}$ and $x, y \in \mathcal{X}$ as {$\kappa(t,x, s,y) + o(1)$}. Moreover, if we define $I_{vw}$ as the indicator that arc $(v, w)$ is present, and $J_{vw}^{txsy} := \1\{T_v =t,X_v = x, T_w = s,  X_w = y \}$. Then, we can write \[A_{ts}^{xy}=\sum_{v \in [n^-]} \sum_{\substack{w \in [n^-] \\ w \neq v}} I_{vw} J_{vw}^{txsy} .\]
	We will now apply the second moment method to show convergence. To compute the first moment we first use homogeneity and linearity to show that
	\[
	\expec[A_{ts}^{xy}] = (n^-)^2 \expec[I_{12}J_{12}^{txsy} ].
	\]
	We can then show that the leftover expectation equals
	\[
	\expec[A_{ts}^{xy}] = (n^-)^2 \cdot  \frac{{(\kappa(t, x, s, y) + o(1))} q_t^x q_s^y}{n^-|\mathcal{X}|^2} = n {\cdot} \frac{\kappa(t, x, s, y)q_t^x q_s^y}{|\mathcal{X}|^2} {+ o(n)}  .
	\]
	Therefore, indeed $\expec[A_{ts}^{xy}/n] \to \gamma_{ts}^{xy}$.
	
	Next, we compute the second moment. For that, we first note that 
	\begin{equation}\label{eq:second moment arc count}
		\expec[(A_{ts}^{xy})^2] = \sum_{v_1 \in \mathcal{V}} \sum_{\substack{w_1 \in \mathcal{V} \\ w_1 \neq v_1}} \sum_{v_2 \in \mathcal{V}} \sum_{\substack{w_2 \in \mathcal{V} \\ w_2 \neq v_2}} \expec[I_{v_1 w_1}I_{v_2w_2}J_{v_1w_1}^{txsy} J_{v_2w_2}^{txsy} ].
	\end{equation}
	To compute \eqref{eq:second moment arc count}, we split up the four-fold sum into the case where $v_1, v_2, w_1$ and $w_2$ are distinct, and the case where they are not. There are $n^4 (1 + o(1))^4$ terms in \eqref{eq:second moment arc count} where the vertices are distinct, and in this case both arcs do not depend on each other. Thus, we find that the contribution of this case to \eqref{eq:second moment arc count} equals
	\begin{equation}\label{eq:distict arcs}
		n^2 (1 + o(1))^{{4}} (\gamma_{ts}^{xy} {+ o(1)})^2.
	\end{equation}
	
	In the case where vertices overlap in the four-fold sum of \eqref{eq:second moment arc count}, we note that an edge can only overlap if two vertices overlap (i.e., when $v_1 = v_2$ and $w_1 = w_2$). If two vertices overlap, this means that the expected value in \eqref{eq:second moment arc count} will be $\mathcal{O}(1/n)$, since at least one arc has to be generated (and the type space is finite). At the same time, the number of terms in the sum where two vertices overlap equals $n^2(1+o(1))^2$. However, if only one vertex overlaps, then the expected value in \eqref{eq:second moment arc count} will equal $\mathcal{O}(1/n^2)$ while the number of terms in the sum with only one overlap will be $n^3(1+o(1))^3$. All in all, if we divide these findings by {$n^2$}, then we see that these terms will only contribute $\mathcal{O}(n)$ to $\expec[(A_{ts}^{xy})^2]$. Therefore, substituting these findings with \eqref{eq:distict arcs} into \eqref{eq:second moment arc count} shows
	\[
	\expec[(A_{ts}^{xy})^2] = n^2 (1 + o(1))^{{4}} (\gamma_{ts}^{xy} {+ o(1)})^2 + \mathcal{O}(n).
	\]
	Thus, we indeed find that $\expec[(A_{ts}^{xy})^2/n^2] \to (\gamma_{ts}^{xy})^2$, meaning that indeed $A_{ts}^{xy} / n \to \gamma_{ts}^{xy}$ in probability.
	
	\paragraph{Step II.} We first note that the result of Step I together with Theorem~\ref{thm:CCI to IRD} applied to $\CCI_{n, \mu}(\hat{T}, \hat{C}, \hat{I}, \hat{J})$ shows us that $A_{ts}^{xy}{/n} \to \kappa(t, x, s, y) q_t^x q_s^y / |\mathcal{X}|^2 = \gamma_{ts}^{xy}$ in probability for the RaCInG model in Definition~\ref{def:RaCInG SpaCInG} too. Moreover, If we denote by $(A_{ts}^{xy})^\pm$ the value of $A_{ts}^{xy}$ in $\CCI_{n, \mu}(\hat{T}, \hat{C}, \hat{I}, \hat{J})$ after adding/removing $c \log(n) \sqrt{n}$ arcs, then we note that
	\[
	A_{ts}^{xy} - c\log(n)\sqrt{n} \leq (A_{ts}^{xy})^- \leq A_{ts}^{xy} \leq (A_{ts}^{xy})^+ \leq  A_{ts}^{xy} + c\log(n)\sqrt{n}.
	\]
	Therefore, we can conclude that $(A_{ts}^{xy})^- / n \to \gamma_{ts}^{xy}$ and $(A_{ts}^{xy})^+ / n \to \gamma_{ts}^{xy}$ in probability too. This establishes the required event insensitivity.    
\end{proof}

\section{Examples}\label{sec:examples}
In this section we give some illustrative examples of some key definitions in this article.

\begin{example}[Independent edge/arc graphs]\label{ex:IEG and IAG}
	For each of the examples we will specify the set $\mathcal{V}$, the vector $V_n$ and the probability function $\pi_n$.
	\begin{enumerate}[label = (\alph*)]
		\item\label{ex:Gil} We can define both the directed and undirected \emph{Gilbert model} (see \cite{Gilbert1959RandomGraphs}) by taking $\mathcal{V} = \mathbb{N}$ with $V_n = [n]$ and setting $\pi_n(v, w) = p_n$ for some constant $p_n \in [0, 1]$. We will abbreviate this undirected model by $\texttt{Gil}_n(p_n)$ and its directed version by $\texttt{DGil}_n(p_n)$.
		
		\item The \emph{geometric inhomogeneous random (di)graph} model (see \cite{Bringmann2019GeometricGraphs}) sets $\mathcal{V} = \mathbb{N} \times \mathbb{T}^d \times \mathbb{R}^+$ where $\mathbb{T}^d$ indicates the $d$-dimensional torus. The vector $V_n$ consists of $n$ vertices of the form $(k, x_k, w_k)$ where $k \in [n]$ and $(x_k, w_k)$ for fixed $k$ is sampled according to some distribution on $ \mathbb{T}^d \times \mathbb{R}^+$. Finally, $\pi_n$ satisfies
		\[
		\pi_n((k, x, v), (l, y, w)) = \left(\frac{1}{\|x - y \|^{\alpha d}} \cdot \left( \frac{vw}{n \lambda} \right)^{\alpha}\right) \wedge 1,
		\]
		where $\alpha, \lambda > 0$ are constants and $\| \cdot \|$ is the $\infty$-norm on $\mathbb{T}^d$.
		
		\item\label{ex:IRD} The \emph{inhomogeneous random (di)graph} model sets $\mathcal{V} = \mathbb{N} \times \mathcal{S}$ for some ground space $\mathcal{S}$. The vector $V_n$ consists of $n$ tuples of the form $(k, T_k)$ where $k \in [n]$ and $T_k$ is an i.i.d. sample from $\mathcal{S}$ according to some distribution. Finally, $\pi_n$ satisfies
		\[
		\pi_n((k, t), (l, s) ) = \left( \frac{\kappa(t, s)}{n} \right) \wedge 1,
		\]
		for some function $\kappa : \mathcal{S} \times \mathcal{S} \to \mathbb{R}_+$.
	\end{enumerate}
\end{example}
\begin{remark}
	Compare Example~\ref{ex:IEG and IAG}\ref{ex:IRD} to Definition~\ref{def:IRD} to verify that both describe the same (directed) model.
\end{remark}

\begin{example}[Edge/arc selection random graphs]\label{ex:ESRG and ASRG}
	For each example we specify the set $\mathcal{V}$, the random vector $V_n$ and the measure $\nu_n$.
	\begin{enumerate}[label = (\alph*)]
		\item\label{ex:ER} We can define both the directed and undirected \emph{Erd\H{o}s-R\'enyi model} (see \cite{Erdos1959OnI.}) by taking $\mathcal{V} = \mathbb{N}$ with $V_n = [n]$ and setting $\nu_n$ be the uniform distribution. We will abbreviate the undirected version of this model by $\texttt{ER}_n(m)$ and its directed counterpart by $\texttt{DER}_n(m)$.
		
		\item\label{ex:EERG} We can define (undirected) \emph{edge-exchangeable random graphs} (see \cite{Janson2018OnGraphs}) by setting $\mathcal{V} = \mathbb{N}$ with $V_n = [n]$, and by fixing a global function $\nu : \mathbb{N}^2 \to \mathbb{R}^+$. To sample edges we first realize $V_n$, then construct $\nu_n$ form $\nu$ by restricting it to and normalizing it with respect to $V_n$.
		
		\item\label{ex:RaCInG} $\texttt{RaCInG}_{n, \mu}(T, C, I, J)$ that disregards self-loops and multi-edges is a special instance of $\ASRG$ by setting $\mathcal{V} = \mathbb{N} \times \mathcal{S}$ with elements in $V_n$ of the form $(k, T_k)$, where $k \in [n]$ and $T_k$ is an i.i.d. random variable supported on $\mathcal{S}$. To define $\nu_n$ we set 
		\begin{align*}
			L_i = \sum_{k \in \mathcal{S}} |\{(v, T_v) \in V_n : T_v = k\}|  \cdot I(k, i), \qquad \text{and} \qquad
			R_j = \sum_{k \in \mathcal{S}} |\{(v, T_v) \in V_n : T_v = k\}|  \cdot J(k, j),
		\end{align*}
		and $p_{ij} = \prob(C = (i, j))$ to get
		\[
		\nu_n((v, T_v), (w, T_w)) = \sum_{i \in \mathcal{L}} \sum_{j \in  \mathcal{R}} \frac{p_{ij} I(T_v, i) J(T_w, j)}{L_i R_j}.
		\] 
		Finally, the model chooses $m_n = \lfloor \mu n \rfloor$ for some fixed $\mu > 0$. 
	\end{enumerate}
\end{example}
\begin{remark}
	{Note in Example~\ref{ex:ESRG and ASRG}\ref{ex:EERG} that this is a special case of Definition 4.2 in \cite{Janson2018OnGraphs} where a finite number of edges is drawn and multi-edges are merged. Moreover, note in Example~\ref{ex:ESRG and ASRG}\ref{ex:RaCInG} that $L_i = 0$ or $R_j = 0$ occurs only if $I(T_v, i) = 0$ or $J(T_w, j) = 0$. This might render $p_{ij}I(T_v, i) J(T_w, j)/(L_i R_j)$ undefined. To alleviate this issue, we set $p_{ij}I(T_v, i) J(T_w, j)/(L_i R_j) = 0$ any time when $I(T_v, i) = 0$ or $J(T_w, j) = 0$. However, from Assumption~\ref{ass:finite support} we know that for large enough $n$ the probability of $L_i = 0$ or $R_j = 0$ occurring vanishes.}
\end{remark}

\begin{example}[Insensitive events in IRG]\label{ex:insensitivity}
	Fix some {bounded} kernel $\kappa$ {(as is common throughout this paper; see Lemma~\ref{lem:kernel bound})} and consider $\texttt{IRG}_n(T, \kappa)$. The following events are insensitive for increase at any rate $(n^{1-\tau}, 0)$ for any $\tau > 0$ fixed.
	\begin{enumerate}[label = (\alph*)]
		\item If $|E_n|$ denotes the number of edges in the model and $\alpha \in \mathbb{R}$ is fixed, then $\mathcal{Q}_n := \{|E_n| > \alpha n\}$ is insensitive. We know that $|E_n| / n \to \expec[\kappa(T_1, T_2)] /2$ in probability (see Proposition 3.4 in \cite{Bollobas2007TheGraphs}). Therefore, adding a sub-linear number of edges will not change the asymptotic probability of exceeding a fixed linear threshold.
		\item If $W_n$ denotes the number of paths of length 2 in the model and $\alpha \in \mathbb{R}$ is fixed, then $\mathcal{Q}_n := \{W_n > \alpha n\}$ is insensitive. To see this, note that each added edge will in a worst-case scenario glue the two vertices with largest degree $D_{\max}$ to each other. This will at most increase the number of wedges by $2 D_{\max}$. {Since the kernel is bounded, the degrees in the graph are (stochastically) bounded by the degrees in the Erd\H{o}s-R\'enyi with parameter $\kappa_+/n$, where $\kappa_+$ is the maximal kernel value. This means that $D_{\max} \leq \log(n)^2$ with high probability (see Theorem 3.13 in \cite{Bollobas2007TheGraphs}). Thus, $2 D_{\max} n^{1 - \tau}/n \leq 2 \log(n)^2 n^{1-\tau}/n \to 0$, showing that the probability of $\mathcal{Q}_n$ does not change as $n \to \infty$.}
		\item If $|\mathcal{C}_{\max}|$ denotes the size of the largest connected component and $\alpha \in \mathbb{R}$ is fixed, then $\mathcal{Q}_n := \{|\mathcal{C}_{\max}| > \alpha n\}$ is insensitive. We know from Theorem 3.1 and 3.12 in \cite{Bollobas2007TheGraphs} that {for bounded kernels} the largest component in $\texttt{IRG}_n(T, \kappa)$ is linear and the second largest logarithmic in size. Since adding an edge will in a worst-case scenario glue a smaller component to the largest one, we have that the largest component can at most become bigger by $\log(n) n^{1 - \tau}$ vertices when you add $n^{1 - \tau}$ edges. Because this is a sub-linear amount of vertices, the probability of exceeding a fixed linear amount of vertices will not change as $n \to \infty$.
	\end{enumerate}
\end{example}

\paragraph{Acknowledgments.} The research of Mike van Santvoort is funded by the Institute for Complex Molecular Systems (ICMS) at Eindhoven University of Technology.

\bibliographystyle{plain} % We choose the "plain" reference style
\bibliography{library.bib} % Entries are in the refs.bib file

\end{document}